%% file: homologytwistedGrings.tex
\documentclass[11pt,reqno]{amsart}
\usepackage[margin=1.5in]{geometry}
\usepackage{setspace}
\newlength{\storeparskip}
\input{preamble.tex}

\usepackage{csquotes} 

\usepackage{graphicx}  
\graphicspath{{./EU-symbol.jpg}}
\allowdisplaybreaks

\title[Topological $\Delta \bfG$-homology of rings with twisted $G$-action]{Topological $\Delta \bfG$-homology of rings\\ with twisted $G$-action}
\author[Angelini-Knoll]{Gabriel Angelini-Knoll}
\address{Department of Mathematics, Applied Mathematics, and Statistics, Case
Western Reserve University, Cleveland, OH, USA}
\email{gja39@case.edu}
\author[Merling]{Mona Merling} 
\address{Department of Mathematics, University of Pennsylvania, 209 33rd St, Philadelphia, PA 19104, USA}
\email{mmerling@math.upenn.edu}
\author[P\'eroux]{Maximilien P\'eroux}
\address{Department of Mathematics, Michigan State University, 619 Red Cedar Road, East Lansing, MI 48824, USA}
\email{peroux@msu.edu}
 
\begin{document} 

\begin{abstract}
We construct topological $\Delta \bfG$-homology for rings with twisted $G$-action. Here a ring with twisted $G$-action is a common generalization of a ring with anti-involution and a ring with $G$-action. This construction recovers as special cases topological Hochschild homology (THH) of rings, with its $S^1$-action, and Real topological Hochschild homology (THR) of rings with anti-involution, with its O(2)-action. A new example of this construction is quaternionic topological Hochschild homology (THQ) of rings with twisted $C_4$-action, which carries a $\Pin(2)$-action. We prove that THQ of a loop space with twisted $C_4$-action can be $\Pin(2)$-equivariantly identified with a twisted free loop space. Other new examples of interest are topological symmetric homology and topological hyperoctrahedral homology and more generally topological twisted symmetric homology. We prove a homotopical version of results of Fiedorowicz, Ault, and Graves computing these new topological homology theories on loop spaces with twisted $G$-action. A key step of independent interest in this program is the construction of a new family of crossed simplicial groups, which correspond to operads that encode the structure of rings with twisted $G$-action.
\end{abstract}


\renewcommand{\subjclassname}{\textup{2020} Mathematics Subject Classification}

\maketitle

\begin{spacing}{0.05}
\setcounter{tocdepth}{1}
\tableofcontents
\end{spacing}

\section{Introduction}

Cyclic homology was introduced by Connes~\cite{Con83} as an analogue of de Rham cohomology and it appeared independently in work of Tsygan~\cite{Tsy83} on the homology of Lie algebras. A key observation of Connes~\cite{Con83}, Tsygan~\cite{Tsy83}, and Loday--Quillen~\cite{LQ84} in the 1980's was that, in a sense that can be made precise, there is a circle action on the Hochschild complex $\HH(A)$ of an associative algebra, which when $A$ is commutative is witnessed by the de Rham differential. 
Then cyclic homology $\HC(A)$ is the Borel homology associated with this $S^1$-action on $\HH(A)$, and one can also define the associated Borel cohomology and Tate cohomology, which give rise to the variants called negative cyclic homology $\HC^-$ and periodic cyclic homology $\HP$. 
Negative cyclic homology, introduced by Jones~\cite{Jon87}, is the target of the trace map from algebraic $K$-theory 
\( \rmK_n(A)\to \HC^-_n(A)\) 
refining the trace map $\rmK_n(A)\longrightarrow \HH_n(A)$
defined by Dennis~\cite{Den76} in the 1970's. Goodwillie~\cite{Goo86} showed that this trace map can be used to compute algebraic K-theory rationally. 

The idea  of generalizing the Hochschild and cyclic homology from the homological context to the homotopical context, working over the deeper base of the sphere spectrum instead of the integers, goes back to ideas of Goodwillie and Waldhausen, who referred to this as ``brave new algebra". The dream of extending Hochschild homology to the homotopical context was first realized in work of~\cite{Bok85}. Notably,  topological Hochschild homology is even a more refined invariant than Hochschild homology when applied to discrete rings as exemplified by the fact that the Hochschild homology of a finite field over the integers is a divided power algebra whereas the topological Hochschild homology of a finite field is a polynomial algebra by~\cite{BokZ,Bre78}. 

Topological cyclic homology first appeared in work of B\"okstedt--Hsiang--Madsen \cite{BHM93}, marking the dawn of a new approach to algebraic K-theory computations known as trace methods. Notably, topological cyclic homology is not the homotopical analogue of cyclic homology, but rather a more refined invariant that requires one to work in the setting of spectra, see \cite[Remark~III.1.9]{NS18} and \cite[\S~12]{Law21}. The fact that topological cyclic homology closely approximates algebraic K-theory was proven by \cite{McC97,Dun89} and later extended by~\cite{DGM12}, \cite{CM21} and~\cite{LS25}. Since the 1990's, trace methods have proven to be the most powerful technique for computing algebraic $K$-theory of ring spectra, e.g.~\cite{HM97,HM03,BM08,AR02,AR08,AKACHR22}.

In 2018, Nikolaus--Scholze~\cite{NS18} reimagined the foundations of topological cyclic homology and this led to several conceptual and computational breakthroughs. For example, this new formulation plays an important role in the construction of syntomic cohomology and Nygaard--complete prismatic cohomology in~\cite{BMS18}, an important advance in the field of $p$-adic Hodge theory. Syntomic cohomology and Nygaard-complete prismatic cohomology have also been extended to the setting of ring spectra~\cite{HRW22,Pst23}. This perspective has led to several recent new breakthroughs computationally as well~\cite{LW22,AKAR23,AKN24,AKHW24}. 

Crossed simplicial groups, a generalization of Connes' cyclic category, were introduced  in the early 1990's by Fiedorowicz--Loday~\cite{FL91} and Krasauskas~\cite{Kra87} independently. The idea is to extend the simplex category $\Delta$ by allowing for additional automorphisms of $[n]$ given by groups $G_n$. Together the automorphisms form a simplicial set $\bfG_{\sbt}$ such that the $n$-simplices are groups, but the face and degeneracy maps are allowed to be crossed homomorphisms. Nevertheless, the geometric realization $|\bfG_{\sbt}|$ is a topological group. 

For example, in the case of the cyclic category the associated simplicial set is the minimial model for the circle with $n$-simplices given by the cyclic group of order $n+1$. The $S^1$-action on the Hochschild complex comes from this fact. Other examples include the dihedral category, where the automorphism groups are the dihedral groups $\{D_{2(n+1)}\}$, and the quaternionic category, where the automorphism groups are the generalized quaternion groups $\{Q_{4(n+1)}\}$. 
Fiedorowicz--Loday~\cite{FL91} introduced algebraic homology theories associated to any crossed simplicial group, for example dihedral homology and quaternionic homology are analogous to cyclic homology. 
 
A homotopical analogue of dihedral homology, known as Real topological Hoch\-schild homology introduced in \cite{HM23} is under active investigation. It can be used to study  Hermitian $K$-theory and $L$-theory of ring spectra with anti-involution~\cite{DMPR21,DMP24,AKGH21,nine21,nine23,nine23b}. This has applications to symmetric bilinear forms~\cite{RS24} as well as to surgery theory of manifolds~\cite{Ran08}. However, homotopical analogues of the homologies associated to other crossed simplicial groups had not been explored until our work to our knowledge. 

Our paper provides a unifying framework for studying variants of topological Hochschild homology, such as Real topological Hochschild homology, but it also introduces new interesting invariants. We set some language to describe our main construction. Any crossed simplicial group $\Delta \bfG$  determines a group homomorphism $G_0\to \{-1,1\}$, or parity, by the classification of crossed simplicial groups \cite{Abo87,Kra87,FL91}. Elements in $G$ that map to $1$ are called even and elements in $G$ that map to $-1$ are called odd. A ring spectrum with twisted $G$-action is an associative ring spectrum with a $G$-action on the underlying spectrum such that even elements of $G$ act by ring homomorphisms and odd elements of $G$ act by ring anti-homomorphisms. A crossed simplicial group is self-dual if it is equipped with an isomorphism $\Delta \bfG^\op\cong \Delta \bfG$. With these definitions in place, our main construction, which appears in \autoref{sec:homology-selfdual} and builds on the results in \autoref{sec:css} and \autoref{sec:twisted-G-actions} can be summarized as follows.

\begin{thmx}
For a self-dual crossed simplicial group $\Delta \bfG$, there is an associated topological homology $\THG(R)$ which carries a $|\mathbf{G}_{\sbt}|$-action, whose input is a ring spectrum $R$ with twisted $G_0$-action. 
\end{thmx}


In the self-dual case, the construction $\THG$ recovers for $\Delta \bfC$ the topological Hochschild homology THH of an associative ring, which has an $S^1$-action, and recovers for $\Delta \bfD$ the Real topological Hochschild homology THR of an associative ring with anti-involution, which has an $O(2)$-action. For $\Delta \bfQ$, the construction $\THG$ gives a new invariant which we call quaternionic topological Hochschild homology THQ, which is equipped with a $\Pin(2)$-action. 

Our approach to the construction of topological $\Delta \bfG$-homology generalizes the higher categorical refinement of the Loday construction due to Nikolaus--Scholze~\cite[Definition~III.2.3]{NS18}. To accomplish this, we prove three key technical results, which are of independent interest: (1) for any group with parity $\varphi \colon G\to\{-1,1\}$ we construct a new crossed simplicial group $\Delta \bfGS$, and we show that any crossed simplicial group $\Delta \bfG$  maps to the crossed simplicial group associated to the canonical parity $G_0\to \{-1,1\}$ (\autoref{csg} and \autoref{canonical map}), (2) we show that twisted $G$-rings are algebras over a twisted operad $\Assoc^\varphi$ (\autoref{TwistedGringoperad}), and (3) we show symmetric monoidal functors from a pointed version of the crossed simplicial group $\Delta \bfGS$ we constructed to spectra encodes algebras over $\Assoc^\varphi$ (\autoref{iso of categories} and \autoref{algebras description}).

Topological Hochschild homology can be viewed as the norm from the trivial group to $S^1$ by \cite{ABGHLM18} and Real topological Hochschild homology can be viewed as the norm from the cyclic group of order two to $O(2)$ by~\cite{AKGH21}.  We therefore view our construction as providing combinatorial models for the (Borel completion) of norms for more general compact Lie groups. Norms play a fundamental role in equivariant homotopy theory, for example see~\cite{HHR16,BDS22}. In the setting of topological groups, the theory of norms is under active investigation, see~\cite{BMM22} for a more geometric approach. As a special case, we plan to lift our construction of quaternionic topological Hochschild homology to a genuine $\Pin(2)$-equivariant spectrum satisfying the universal property of the norm from the cyclic group of order four to $\Pin(2)$. 

The norm perspective informs our computations. As a special case, we give a description of quaternionic topological Hochschild homology in terms of \emph{twisted free loop spaces}  with a $\Pin(2)$-action. Twisted free loop spaces arise naturally in the study of fixed point theory, see for example~\cite{KW07,CP19,KY23}. Here we consider the twisted free loop space 
$ \cL^{\tau}X\coloneqq
\{ \gamma \colon [0,1] \to X \mid t^2\gamma(0)=\gamma(1)\} 
$
where $t$ is a generator for $C_4\subset \Pin(2)$. Alternatively, we show that this can be identified with the norm $\Map^{C_2}(\Pin(2),X)\simeq \cL^{\tau}X$
as a space with left $\Pin(2)$-action. We show that topological quaternionic homology of a loop space with twisted $C_4$-action $\Omega^{q}X$ defined in \autoref{loop-space-twisted-G-action}
is a twisted free loop space.

\begin{thmx}[{\autoref{thm: main computation}}]
Let $X$ be a space with $C_4$-action and let $q\colon C_4\to C_2$ be the quotient. 
Then there is an equivalence
\begin{align*}
\THQ(\Sigma^\infty_+\Omega^{q}X) &\simeq \Sigma_{+}^{\infty}\cL^{\tau}X 
\end{align*}
of spectra with left $\Pin(2)$-action.
\end{thmx}

We also recover a result of \cite{Hog16}, \cite{DMPR21}, \cite{HHKWZ24}, and~\cite{DMP24} about Real topological Hochschild homology of loop spaces with anti-involution.

Beyond the self-dual case, there are many crossed simplicial groups of interest such as the symmetric, braid, reflexive, and hyperoctahedral crossed simplicial groups. Fiedorowicz--Loday~\cite{FL91} construct homology theories associated to arbitrary crossed simplicial groups in the algebraic context. In the case of the symmetric, hyperoctahedral, and reflexive categories these have been studied in the algebraic context by \cite{Fie,Aul10,Gra22,Gra22b,BR23,LR24}. We construct a homotopical homology theory associated to an arbitrary crossed simplicial group in \autoref{sec:homology}, again building on \autoref{sec:css} and \autoref{sec:twisted-G-actions}. 
  
\begin{thmx}\label{theorem C}
    For a crossed simplicial group $\Delta \bfG$, there is an associated positive topological $\Delta \bfG$-homology $\mathrm{T}\mathbf{G}^+(R)$ spectrum, whose input is a ring spectrum $R$ with twisted $G_0$-action. When $\Delta \bfG$ is self-dual, $\mathrm{T}\mathbf{G}^+(R)$ is equivalent to the homotopy $|\bfG_{\sbt}|$-orbits of $\THG(R)$.
\end{thmx}

Our construction of $\Delta \bfGS$ from \autoref{csg} associated to a group with parity $\varphi\colon G\to \{-1, 1\}$, which recovers the symmetric crossed simplicial group if $\varphi$ is the trivial group homomorphism from the trivial group, and the hyperoctahedral crossed simplicial group if $\varphi$ is the identity, is also not self-dual. 
The associated positive topological $\Delta \bfG$-homology of \autoref{theorem C} when $\Delta \bfG=\Delta \bfGS$ yields a new example, which we call topological twisted symmetric homology $\mathrm{T}\varphi(R)$ of a ring spectrum $R$ with twisted $G$-action.
Our construction $\mathrm{T}\varphi(R)$ provides new homotopical analogues of symmetric homology and hyperoctahedral homology in particular. 

As an application, we prove a homotopical analogue of theorems of Fiedoricz \cite{Fie}, Ault~\cite{Aul10} and Graves~\cite{Gra22} as well as generalize this result to the setting of topological twisted symmetric homology. Below, we state a special case.

\begin{thmx}[{\autoref{thm: topological positive hyperoctahedral homology}}]
Let $\varphi\colon G\to C_2$ be a group homomorphism.
If $X$ is a connected space with $G$-action, then there is an equivalence of spectra
\[
\mathrm{T}\varphi(\Sigma^\infty_+\Omega^{\varphi} X)\simeq (\Sigma^\infty_+ \Omega QX)_{hG}\,.
\] 
\end{thmx}
Here $\Omega^{\varphi}X$ is a loop space with twisted $G$-action as defined in \autoref{loop-space-twisted-G-action}. 
We write $QX$ for the Borel $G$-space with underlying space $\colim_{k}\Omega^k\Sigma^kX$ and $\Omega Y$ for the Borel $G$-space whose underyling space is the loop space of $Y$, see \autoref{twisted symmetric} for details. 

Hochschild homology and cyclic homology can be interpreted as functor homology~\cite{PR02}. In the same vein, we present our constructions as a form of homotopical functor homology, see \autoref{sec:functor-homology}. In particular, this leads to natural bivariant analogues of the constructions in this paper. We hope that this perspective can connect our construction to the growing body of research in the area of functor homology, for example~\cite{DT24,KV24}. See~\cite{FT15} for a survey. 

\subsection{Outline}
In \autoref{sec:css}, after we give an overview of crossed simplicial groups, we show our  technical result (1), the construction of  a new family of crossed simplicial groups $\Delta \bfGS$ associated to groups with parity; i.e. group homomorphisms $\varphi\colon G\to \{-1,1\}$, which are a generalization of the hyperoctahedral crossed simplicial groups, and (2) showing that the usual functor from a crossed simplicial group $\Delta \bfG$ to the hyperoctahedral crossed simplicial group $\Delta \bfH$ factors through this new construction $\Delta \bfGS$ for a parity $\varphi$ of the automorphism group $G_0$, which is part of the structure of the crossed simplicial group $\Delta \bfG$. The combinatorics of checking the axioms of a crossed simplicial group and this factorization occupy most of \autoref{sec:css}.

In \autoref{sec:twisted-G-actions}, we introduce twisted $G$-rings for groups with parity in the sense of \cite{DK15} and we show our technical result (2), namely that they are algebras over a twisted operad $\Assoc^\varphi$. This is a particular case of the semidirect product operads defined in \cite{SW03}. Lastly, we prove our technical result (3), namely that a pointed version $\Delta \bfGS_+$ is the active part of the category of operators associated to the operad $\Assoc^\varphi$. Consequently, symmetric monoidal functors from $\Delta \bfGS_+$ to spectra correspond to twisted $G$-rings. In other words, the category $\Delta \varphi \wr \bfS_+$ is the symmetric monoidal envelope of $\Assoc^\varphi$.

In \autoref{sec:thgdef}, using all the ingredients from the previous two sections, we build the $\infty$-categorical generalization of the cyclic bar construction for twisted $G$-rings. This defines analogues of Hochschild homology $\THG$, topological negative cyclic homology $\mathrm{T}\mathbf{G}^{-}$, and topological periodic cyclic homology $\mathrm{T}\mathbf{G}^{\textup{per}}$ for an arbitrary self-dual crossed simplicial group. For an arbitrary crossed simplicial group, we also define a homotopical analogue of the homology of crossed simplicial groups. We also give an interpretation of our construction in terms of homotopical functor homology. 

In \autoref{computations}, we compute the Real and quaternionic topological Hochschild homologies of loop spaces with twisted action in terms of twisted free loop spaces. For this, we introduce categorical models of $S^1$, $O(2)$ and $\Pin(2)$ and use them to define unstable versions of
Real topological Hochschild homology and quaternionic topological Hochschild homology, which serve as intermediaries in the computation. We also compute analogues of topological negative cyclic homology and topological periodic cyclic homology for the quaternionic crossed simplicial group in the case of loop spaces with twisted $C_4$-action in \autoref{TQ+}. 

In \autoref{twisted symmetric}, we compute the topological twisted symmetric homology of loop spaces with twisted $G$-action. As special cases this provides a homotopical analogue of a result of Fiedorowicz and Ault on the symmetric homology of loop spaces and a result of Graves on the hyperoctahedral homology of loop spaces with anti-involution. 

\subsection{Generalizations of Hochschild homology}
Here, we highlight some special cases of the main construction in this paper and describe some future research.

\subsubsection{Quaternionic topological Hochschild homology}
The quaternionic crossed simplicial group $\Delta \bfQ$ gives rise to a new homology that we call quaternionic topological Hochschild homology equipped with a left $\mathrm{Pin}(2)$-action. We expect that this can be refined to a genuine $C_4$-spectrum and reserve the name hyperreal Hochschild homology for such an invariant.\footnote{The authors thank J.D. Quigley for suggesting the name \emph{hypperreal} following the convention for hyperreal Bordism (cf. \cite[p.2]{HZ18}).}
We expect that hyperreal Hochschild homology enjoys the universal property of the norm from $C_4$ to $\Pin(2)$. We also expect that this theory is equipped with a notion of hyperreal cyclotomic structure, which can be used to define a $C_4$-equivariant invariant called hyperreal topological cyclic homology.

The presence of a $\Pin(2)$-action on quaternionic topological Hochschild homology suggests that this theory could shed light on involutive Heegaard Floer homology~\cite{HM17} and involutive bordered Floer homology~\cite{HL19}  just as topological Hochschild homology is related to bordered Floer homology~\cite{LOT15}. For example, in the setting of topological Hochschild homology, this was employed by \cite{Law21} to recast the non-commutative Hodge-to-de Rham spectral sequence in Heegaard Floer homology due to Lipshitz--Treuman~\cite{LT16}. One might also hope to construct a quantum variant of quaternionic topological Hochschild homology by working with relative quaternionic topological Hochschild homology with suitable twisted bimodule coefficients as in~\cite[Remark~3.18]{AGW25}. 
 
\subsubsection{Equivariant topological Hochschild homology}
Our work also provides a combinatorial model for $G$-equivariant topological Hochschild homology.
It was suggested that in the case of $C_n$-equivariant topological Hochschild homology there is a $C_n$-equivariant trace map
$\mathsf{K}_{C_n}(R)\longrightarrow \iota_{\Delta_n}^*\operatorname{THH}_{C_n}(R)$,
in~\cite[Conjecture~6.2.2]{AGHKK23}. Here $\mathsf{K}_{C_n}(R)$ is $C_n$-equivariant algebraic $K$-theory in the sense of~\cite{Mer17} and $\iota_{\Delta_n}^*\operatorname{THH}_{C_n}(R)$ is the restriction to the diagonal cyclic subgroup $\Delta_n$ of order $n$ of $C_n$-equivariant topological Hochschild homology.\footnote{Note that $\operatorname{THH}_{C_n}(R)$ would be $N_{C_n}^{S^1\times C_n}(R)$ in the notation of~\cite{AGHKK23} and the notation $\THH_{C_n}$ has a different meaning in loc.\ cit.} 
Trace methods in this setting is currently being actively developed in work of Chan--Gerhardt--Klang~\cite{CGK}. We therefore hope that our combinatorial model for equivariant topological Hochschild homology will aid in the computation of equivariant algebraic $K$-theory groups.  

\subsubsection{Equivariant Real topological Hochschild homology}

Our construction specializes to $G$-equivariant Real topological Hochschild homology. One might ask whether there exists a $G$-equivariant Real algebraic $K$-theory, along the same lines as $G$-equivariant Real cobordism explored in~\cite{CWY25}. We hope our construction of $G$-equivariant Real topological Hochschild homology could be useful for studying this conjectural $G$-equivariant Real algebraic $K$-theory. For example, one might generalize results of Kylling--R\"ondigs--{\O}stvaer~\cite{KRO20} to the setting of Artin L-functions along the lines of Elmanto--Zhang~\cite{EZ24}.

\subsubsection{Topological twisted symmetric homology and topological braid homology}

Our construction also specializes to topological symmetric homology, topological hyperoctahedal homology and more generally topological twisted symmetric homology. In the algebraic case, this was studied originally by Fiedorowcz~\cite{Fie}, Ault~\cite{Aul10} and Graves~\cite{Gra22}. There are several open conjectures about topological symmetric homology for example that we believe our methods could be useful for shedding light on. This is especially interesting in light of the recent connections to representation homology~\cite{BR23} where the authors resolve some of these conjectures in the characteristic zero case. Results in the case of symmetric homology suggest that there might be analogues for topological braid homology where the role of the commutative operad is replaced by an $\mathrm{E}_2$ operad. 

\subsection{Acknowledgements}
The authors would like to acknowledge contributions  to this paper arising from conversations with
Daniel Berwick-Evans, Thomas Blom, David Chan, Emanuele Dotto, Daniel Graves, Liam Keenan, Inbar Klang, Connor Malin, Thomas Nikolaus, Birgit Richter, JD Quigley, Brian Shin, Sarah Whitehouse, and Foling Zou. They especially thank David Chan for key insights and an anonymous referee for valuable feedback that significantly improved the paper.
Merling acknowledges partial support from NSF DMS grants CAREER 1943925 and FRG 2052988. Angelini-Knoll is grateful to Max Planck Institute for Mathematics in
Bonn for its hospitality and financial support. This project has received funding from the European Union's Horizon 2020 research and innovation programme under the Marie Sk\l{}odowska-Curie grant agreement No 1010342555. 
\thinspace \includegraphics[scale=0.1]{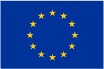} 

\section{Crossed simplicial groups}\label{sec:css}
In this section, we introduce groups with parity and construct their associated crossed simplicial groups, which generalize the hyperoctahedral crossed simplicial group \cite[Example 6]{FL91}. We start with a quick survey of the theory of crossed simplicial groups following \cite{FL91,DK15} and refer the reader to these sources for additional details. 

\subsection{Survey of crossed simplicial groups}
Recall that $\Delta$ is a small category with objects the totally ordered sets $[n]=\{0,1, \dots, n\}$ and order preserving maps, and it is generated by the morphisms $\sigma_i\colon [n+1]\to [n]$, which repeats $i$,  and $\delta_i\colon [n-1]\to[n]$, which skips $i$, and which satisfy the co-simplicial identities \cite[p.~4]{GJ99}. 
\begin{defin}
A \emph{crossed simplicial group} is a small category $\Delta \bfG$ with  objects $[n]$ for $n\ge 0$ and morphism sets
 \[\Delta \bfG ([n],[m])\coloneqq \Delta([n],[m])\times \Aut_{\Delta\bfG}([n]),\]
where $\Aut_{\Delta \bfG}([n])=G_n^{\op}$ for some sequence of groups $\{G_n \}_{n\ge 0}$ satisfying the following \emph{fundamental axiom of crossed simplicial groups}: every morphism $f\in \Delta\bfG([n], [m])$ has a unique  factorization as a composite $f=\phi\circ g$ where $\phi\in \Delta([n],[m])$ and $g\in  G_n^{\op}$. 
\end{defin} 

\begin{rem}
Note that the underlying set of $G_{n}^{\op}$ is the same as that of $G_{n}$, so when speaking of the underlying set we omit the superscript. The definition is arranged so that $\Aut_{\Delta\bfG^{\op}}([n])=G_{n}$. 
\end{rem}

By definition, every  crossed simplicial group $\Delta \bfG$ comes equipped with a faithful functor 
$
    \iota \colon \Delta \hookrightarrow \Delta\bfG 
$
whose essential image is a wide subcategory.  
Let $g\in G_m^{\op}$ and $\phi \in \Delta \bfG([n], [m])$. By the fundamental axiom of crossed simplicial groups, the composite $g\circ \phi$  in $\Delta \bfG$  can be uniquely written as $g^*\phi \circ \phi^*g$ for some $g^*\phi\in \Delta([n],[m])$ and $\phi^*g\in G_n^{\op}$. 
In other words, the diagram
\[
\begin{tikzcd}
    & {[m]} \ar{d}{g}\\
    {[n]}\ar{r}[swap]{\phi} & {[m]}
\end{tikzcd}
\]
can be uniquely completed to a commuting diagram 
\[
\begin{tikzcd}
    {[n]} \ar[dashed]{r}{g^*\phi} \ar[dashed]{d}[swap]{\phi^*g} & {[m]}\ar{d}{g}\\
    {[n]}\ar{r}[swap]{\phi} & {[m]}
\end{tikzcd}
\]
\noindent in $\Delta \bfG$ where by convention the order of morphism composition $g\circ h$ in $G_n^{\op}$ corresponds to the order of multiplication $g\cdot h$ in $G_n$ when we vertically compose diagrams of this form. 
Thus, for every $g\in G_m$ and every $\phi \in \Delta([n],[m])$, this rule defines maps of sets $g^\ast\colon \Delta([n],[m]) \to \Delta([n],[m])$ and $\phi^*\colon G_m \to G_n$ that determines the crossed simplicial structure in the following way.

\begin{prop}[{\cite[1.3, 1.6]{FL91}}]\label{FL characterization of CSG}
    Given a crossed simplicial group $\Delta \bfG$, the maps of sets $g^*$ and $\phi^*$ defined above have the following properties.\footnote{We are using the numbering of \cite[1.6]{FL91} for the six identities that characterize a crossed simplicial group.}
  
    \begin{enumerate}
        \item The maps $\phi^*$ assemble into  a simplicial set:
        \begin{align*}
           \bfG_{\sbt}\colon \Delta^{\op} & \longrightarrow \mathrm{Set}\\
            {[n]} & \longmapsto G_n\\
            \Big( [n]\stackrel{\phi}\rightarrow [m]\Big) & \longmapsto \Big( G_m \stackrel{\phi^*}\rightarrow G_n\Big)
        \end{align*}
     Moreover, $\phi^*$ preserves the group identities. In other words, the identities
    \begin{equation}\tag{1.h}\label{eq: 1.h}
            (\phi \circ \psi)^*(g)=\psi^*(\phi^*(g))
        \end{equation}
        \begin{equation}\tag{3.h}\label{eq: 3.h}
            (\id_{[m]})^*(g)=g, \quad \quad \phi^*(e_m)=e_n,
        \end{equation}
    hold for any $g\in G_m$, $\phi\in\Delta([n],[m])$ and $\psi\in \Delta([k],[n])$. Here $e_n\in G_n$ denotes the group identity in $G_n$.
    \item The maps $g^*$ determine a right action of $G_m$ on $\Delta([n],[m])$:
    \begin{align*}
        \Delta([n], [m])\times G_m & \longrightarrow \Delta([n],[m])\\
        (\phi, g) & \longmapsto g^*\phi.
    \end{align*}
    Moreover, when $m=n$, the action preserves identities on $\Delta$. In other words, the identities
    \begin{equation}\tag{1.v}\label{eq: 1.v}
        (gh)^*(\phi)=h^*(g^*(\phi))
    \end{equation}
    \begin{equation}\tag{3.v}\label{eq: 3.v}
        e_m^*\phi=\phi, \quad \quad g^*(\id_{[n]})=\id_{[n]},
    \end{equation}
   hold for any $g,h\in G_m$ and $\phi\in \Delta([n],[m])$.
    \item Furthermore, the set maps $\phi^*$ and $g^*$ preserve compositions with a crossing, in the following sense:
    \begin{equation}\tag{2.h}\label{eq: 2.h}
        g^*(\phi\circ \psi)=g^*\phi \circ (\phi^*g)^*(\psi),
    \end{equation}
    \begin{equation}\tag{2.v}\label{eq: 2.v}
        \phi^*(gh)=\phi^*g \cdot (g^*\phi)^*(h),
    \end{equation}
    for all $g,h\in G_m$, $\phi\in\Delta([n], [m])$ and $\psi\in \Delta([k], [n])$.
    \end{enumerate}
    Conversely, a sequence of groups $\{G_n\}_{n\geq 0}$ with two sets of functions $g^*$ and $\phi^*$ satisfying the relations above defines a crossed simplicial group.
\end{prop}

Every simplicial group $\bfG_{\sbt}$ satisfies the axioms where the right $G_m$-action on $\Delta([n],[m])$ is trivial. By \cite{Lod87}, \cite[Theorem 5.3 (i)]{FL91}, the geometric realization 
$|\bfG_{\sbt}|$ of a crossed simplicial group is a topological group. We give a table showing some common examples of crossed simplicial groups in \autoref{fig: common crossed simplicial g}.

\begin{figure}
    \centering
    $\begin{array}{c|c|c|c}
    \text{crossed simplicial group} & \Delta \bfG & G_n & |\bfG_{\sbt}| \\ \hline 
\text{cyclic category} & \Delta \bfC & C_{n+1} & S^1 \\ 
\text{dihedral category} & \Delta \bfD & D_{2(n+1)} & O(2) \\ 
\text{quaternionic category} & \Delta \bfQ & Q_{4(n+1)} & \text{Pin}(2) \\ 
\text{reflexive category} & \Delta \bfR & C_2 & C_2  \\
\text{braid category} & \Delta \bfB & B_{n+1} & *\\
\text{symmetric category} & \Delta \bfS & \Sigma_{n+1} & *  \\
\text{hyperoctahedral category} & \Delta \bfH     & C_2\wr \Sigma_{n+1} & *  \\
    \end{array}$
    \caption{Table of common crossed simplicial groups}
    \label{fig: common crossed simplicial g}
\end{figure}

For proofs that each of these form crossed simplicial groups, we refer the reader to \cite[Examples 4--6]{FL91}. We will also recover a proof that $\Delta \bfH$ is a crossed simplicial group as a special case of \autoref{csg}. 

We review in detail the construction of the crossed simplicial group $\Delta \bfS$ since we build on it in \autoref{crossedforpar}.
\begin{exm}\label{sym_example}
The symmetric groups $\{\Sigma_{n+1}\}$ define a crossed simplicial group $\Delta \bfS$ as follows. For a diagram
\[
\begin{tikzcd}
    & {[n]} \ar{d}{\gamma}\\
    {[m]}\ar{r}[swap]{\phi} & {[n]}
\end{tikzcd}
\]
where $\gamma \in \Sigma_{n+1}$ is a permutation and $\phi$ is a map in $\Delta$, the maps $\phi^\ast \gamma$ and $\gamma^\ast \phi$ are defined as the only set maps that make the following set diagram commute
\[
  \xymatrix{
    [m] \ar[d]_-{\phi^\ast \gamma} \ar[r]^-{\gamma^\ast \phi } & [n] \ar[d]^-\gamma\\
    [m] \ar[r]_-{\phi } &[n]
    }
\]
    and such that the permutation $\phi^\ast \gamma$ preserves the order of the preimages of $\gamma^\ast \phi$. More precisely, since the map $\gamma^\ast \phi\colon  [m]\to [n]$ is completely determined by the sizes of the inverse images, by the commutativity of the diagram these need to agree with the sizes of the inverse images of $\phi$. Then the permutation $\phi^\ast \gamma$ is also completely determined by the commutativity of the set diagram together with the requirement that orders of preimages get preserved. See an example in \autoref{figure of delta sigma}.
\begin{figure}
\centering
\begin{tikzpicture}[baseline= (a).base]
\node[scale=0.60] (a) at (1,1){
\begin{tikzcd}[row sep=small]
	&&&&&& \textcolor{rgb,255:red,214;green,92;blue,214}{5} &&& {\gamma^\ast \phi} \\
	&&&&& \textcolor{rgb,255:red,214;green,92;blue,214}{4} \\
	&&&& \textcolor{rgb,255:red,92;green,214;blue,92}{3} &&&&&&&& \textcolor{rgb,255:red,214;green,92;blue,214}{3} \\
	&&& \textcolor{rgb,255:red,92;green,214;blue,92}{2} &&&&&&&& \textcolor{rgb,255:red,92;green,214;blue,92}{2} \\
	&& \textcolor{rgb,255:red,92;green,214;blue,92}{1} &&&&&&&& \textcolor{rgb,255:red,92;green,92;blue,214}{1} \\
	& \textcolor{rgb,255:red,228;green,158;blue,78}{0} &&&&&&&& \textcolor{rgb,255:red,228;green,158;blue,78}{0} \\
	&&&&&& \textcolor{rgb,255:red,92;green,214;blue,92}{5} &&&&&& {\gamma} \\
	&&&&& \textcolor{rgb,255:red,92;green,214;blue,92}{4} \\
	{\phi^\ast \gamma} &&&& \textcolor{rgb,255:red,92;green,214;blue,92}{3} &&&&&&&& \textcolor{rgb,255:red,92;green,92;blue,214}{3} \\
	&&& \textcolor{rgb,255:red,228;green,158;blue,78}{2} &&&&&&&& \textcolor{rgb,255:red,92;green,214;blue,92}{2} \\
	&& \textcolor{rgb,255:red,214;green,92;blue,214}{1} &&&&&&&& \textcolor{rgb,255:red,228;green,158;blue,78}{1} \\
	& \textcolor{rgb,255:red,214;green,92;blue,214}{0} &&&&&&&& \textcolor{rgb,255:red,214;green,92;blue,214}{0} \\
	&&&&& {\phi}
	\arrow[draw={rgb,255:red,214;green,92;blue,214}, no head, from=1-7, to=3-13, thick]
	\arrow[draw={rgb,255:red,214;green,92;blue,214}, no head, from=1-7, to=11-3, thick]
	\arrow[draw={rgb,255:red,214;green,92;blue,214}, no head, from=2-6, to=3-13, thick]
	\arrow[draw={rgb,255:red,214;green,92;blue,214}, no head, from=2-6, to=12-2, thick]
	\arrow[draw={rgb,255:red,92;green,214;blue,92}, no head, from=3-5, to=7-7, thick]
	\arrow[draw={rgb,255:red,92;green,214;blue,92}, no head, from=4-4, to=4-12, thick]
	\arrow[draw={rgb,255:red,92;green,214;blue,92}, no head, from=4-4, to=8-6, thick]
	\arrow[draw={rgb,255:red,92;green,214;blue,92}, no head, from=4-12, to=3-5, thick]
	\arrow[draw={rgb,255:red,92;green,214;blue,92}, no head, from=5-3, to=4-12, thick]
	\arrow[draw={rgb,255:red,92;green,214;blue,92}, no head, from=5-3, to=9-5, thick]
	\arrow[draw={rgb,255:red,228;green,158;blue,78}, no head, from=6-2, to=6-10, thick]
	\arrow[draw={rgb,255:red,228;green,158;blue,78}, no head, from=6-2, to=10-4, thick]
	\arrow[draw={rgb,255:red,92;green,214;blue,92}, no head, from=7-7, to=10-12, thick]
	\arrow[draw={rgb,255:red,92;green,214;blue,92}, no head, from=8-6, to=10-12, thick]
	\arrow[draw={rgb,255:red,92;green,214;blue,92}, no head, from=9-5, to=10-12, thick]
	\arrow[draw={rgb,255:red,92;green,92;blue,214}, no head, from=9-13, to=5-11, thick]
	\arrow[draw={rgb,255:red,228;green,158;blue,78}, no head, from=10-4, to=11-11, thick]
	\arrow[draw={rgb,255:red,92;green,214;blue,92}, no head, from=10-12, to=4-12, thick]
	\arrow[draw={rgb,255:red,214;green,92;blue,214}, no head, from=11-3, to=12-10, thick]
	\arrow[draw={rgb,255:red,228;green,158;blue,78}, no head, from=11-11, to=6-10, thick]
	\arrow[draw={rgb,255:red,214;green,92;blue,214}, no head, from=12-2, to=12-10, thick]
	\arrow[draw={rgb,255:red,214;green,92;blue,214}, no head, from=12-10, to=3-13, thick]
\end{tikzcd}
};
\end{tikzpicture}
\caption{Example in $\Delta \bfS$ for some $\phi\colon[5]\to [3]$ and $\gamma\in \Sigma_4$.}
\label{figure of delta sigma}
\end{figure}
It is shown in \cite{Lod92} that these formulas do indeed define a crossed simplicial group. 
\end{exm} 
We end this section with a lemma that produces new crossed simplicial groups from previously known ones. This result will be useful for some examples appearing in later sections. 
\begin{lem}\label{extension by constant}
Suppose $\Delta \bfG$ is a crossed simplicial group and $H$ is a {discrete group}, regarded as a constant simplicial group. Then there is a crossed simplicial group $\Delta \bfGH$ where 
$\Aut_{\Delta \bfGH }([n])=G^{\op}_n\times H^\op$ and for
$\phi\in \Delta([m], [n])$, $g\in G_n$ and $h\in H$:
\[
\phi^*(g,k)=(\phi^*g, h), \quad \quad (g,h)^*\phi=g^*\phi.\]
\end{lem}
 \begin{proof}
    This follows from \autoref{FL characterization of CSG} and the fact that the crossed simplicial structure of $\Delta \bfG\times \mathrm{H}$ is entirely determined by the crossed simplicial group structure of $\Delta \bfG$.
\end{proof}

\subsection{The twisted symmetric crossed simplicial group}
The following definition will be useful for defining a certain class of crossed simplicial groups. We write $C_2=\{-1,1\}$ for the cyclic group of order two.

\begin{defin}\label{gppar}
Let $G$ be a group equipped with a group homomorphism $\varphi\colon G \to C_2$. We call $\varphi\colon G\to C_2$ a \emph{parity} and the pair $(G,\varphi)$ a \emph{group with parity}.
We refer to 
$\varphi^{-1}(1)$ and $\varphi^{-1}(-1)$, as even and odd elements of $G$,  respectively. 
\end{defin}

\begin{rem}
In~\cite{KP18} a group with parity is called an oriented group and in~\cite{SS94}, for example, a group with parity is called a group with sign structure. We follow the terminology of~\cite{DK15}. 
\end{rem}

The hyperoctahedral crossed simplicial group $\Delta \bfH$ is defined by generalizing the definition of $\Delta \bfS$ from \autoref{sym_example}. Recall that the hyperoctahedral groups satisfy $H_{n+1} = C_2\wr \Sigma_{n+1}$, so their elements can be viewed as signed permutations. The formulas that define the associated crossed simplicial group are the same as in \autoref{sym_example}, except that the labels on the permutation are used to either preserve or reverse inverse images. For the explicit definition, see \cite[Section 3.1]{FL91}. We now give a more general definition of a crossed simplicial group associated to a group with parity, which will recover the hyperoctahedral group as an example. 

\begin{const}[The twisted symmetric crossed simplicial group]\label{crossedforpar}
Let $(G,\varphi)$ be a group with parity, and consider the collection of groups $\{G\wr \Sigma_{n+1}\}$. For $\phi\colon [m]\to [n]$ in $\Delta$ and $(g_0, \dots, g_n;\ \gamma)\in G\wr \Sigma_{n+1}$ we define  
\[
(g_0, \dots, g_n;\ \gamma)^\ast \phi \colon [m] \to [n]
\]
as $\gamma^\ast \phi$, where the latter is the definition of the formula for $\Delta \bfS$ from \autoref{sym_example}. Moreover, if $e$ is the identity in $G$, we define 
$$\phi^\ast(e,\dots, e;\ \gamma)=(e, \dots, e;\ \phi^\ast \gamma) \in G\wr \Sigma_{m+1},$$ 
where again $\phi^\ast \gamma$ is the definition in $\Delta \bfS$.

On the generators $\delta_i$ and $\sigma_i$ of $\Delta$, we define
$$\delta_i^\ast(g_0, \dots, g_n;\ \id)=(g_0, \dots, g_{i-1}, g_{i+1}, \dots, g_n;\ \id)$$ and 
$$\sigma_i^\ast(g_0,\dots, g_n;\ \id)=(g_0, \dots, g_i, g_i, \dots, g_n;\ \tau),$$
where $$\tau= \begin{cases} \id & \text{ if } \varphi(g_i)=1 \\
(i, i+1) & \text{ if } \varphi(g_i)=-1\end{cases}.$$
We can extend this by the rule $(\phi\circ \psi)^\ast=\psi^\ast\circ \phi^\ast$. Then note that the formula for $\phi^\ast (g_0,\dots, g_n;\ \id)$ repeats each $g_i$ as many times as the size of $\phi^{-1}(i)$ and the permutation is the one that permutes within the blocks of preimages, either preserving or reversing the order depending on the parity of the corresponding element $g_i$.

Now in general, we define
\[
\phi^\ast (g_0,\dots, g_n;\ \gamma)= \phi^\ast (g_0,\dots, g_n;\ \id)\cdot \phi^\ast(e, \dots, e;\ \gamma)\,.
\]
See for instance \autoref{fig: new csg}.
\begin{figure}
    \centering
    \begin{tikzpicture}[baseline= (a).base]
 \node[scale=0.60] (a) at (1,1){
\begin{tikzcd}[row sep=small]
	&&&&& \textcolor{rgb,255:red,214;green,92;blue,214}{g_0^+} \\
	&&&& \textcolor{rgb,255:red,214;green,92;blue,214}{g_0^+} && \textcolor{rgb,255:red,214;green,92;blue,214}{5} &&& {\gamma^\ast \phi} \\
	&&& \textcolor{rgb,255:red,92;green,214;blue,92}{g_2^-} && \textcolor{rgb,255:red,214;green,92;blue,214}{4} \\
	&& \textcolor{rgb,255:red,92;green,214;blue,92}{g_2^-} && \textcolor{rgb,255:red,92;green,214;blue,92}{3} &&&&&&&& \textcolor{rgb,255:red,214;green,92;blue,214}{3} \\
	& \textcolor{rgb,255:red,92;green,214;blue,92}{g_2^-} && \textcolor{rgb,255:red,92;green,214;blue,92}{2} &&&&&&&& \textcolor{rgb,255:red,92;green,214;blue,92}{2} \\
	\textcolor{rgb,255:red,228;green,158;blue,78}{g_1^-} && \textcolor{rgb,255:red,92;green,214;blue,92}{1} &&&&&&&& \textcolor{rgb,255:red,92;green,92;blue,214}{1} \\
	& \textcolor{rgb,255:red,228;green,158;blue,78}{0} &&&&&&&& \textcolor{rgb,255:red,228;green,158;blue,78}{0} \\
	&&&&&& \textcolor{rgb,255:red,92;green,214;blue,92}{5} &&&&&&& {\gamma} \\
	&&&&& \textcolor{rgb,255:red,92;green,214;blue,92}{4} \\
	{\phi^\ast \gamma} &&&& \textcolor{rgb,255:red,92;green,214;blue,92}{ 3} &&&&&&&& \textcolor{rgb,255:red,92;green,92;blue,214}{3} \\
	&&& \textcolor{rgb,255:red,228;green,158;blue,78}{2} &&&&&&&& \textcolor{rgb,255:red,92;green,214;blue,92}{2} && \textcolor{rgb,255:red,92;green,92;blue,214}{g_3^+} \\
	&& \textcolor{rgb,255:red,214;green,92;blue,214}{1} &&&&&&&& \textcolor{rgb,255:red,228;green,158;blue,78}{1} && \textcolor{rgb,255:red,92;green,214;blue,92}{g_2^-} \\
	& \textcolor{rgb,255:red,214;green,92;blue,214}{0} &&&&&&&& \textcolor{rgb,255:red,214;green,92;blue,214}{0} && \textcolor{rgb,255:red,228;green,158;blue,78}{g_1^-} \\
	&&&&& {\phi} &&&&& \textcolor{rgb,255:red,214;green,92;blue,214}{g_0^+}
	\arrow[color={rgb,255:red,214;green,92;blue,214}, no head, from=2-7, to=4-13, thick]
	\arrow[color={rgb,255:red,214;green,92;blue,214}, no head, from=3-6, to=4-13, thick]
	\arrow[color={rgb,255:red,92;green,214;blue,92}, no head, from=4-5, to=5-12, thick]
	\arrow[color={rgb,255:red,92;green,214;blue,92}, no head, from=4-5, to=10-5, thick]
	\arrow[color={rgb,255:red,92;green,214;blue,92}, no head, from=5-4, to=5-12, thick]
	\arrow[color={rgb,255:red,92;green,214;blue,92}, no head, from=5-4, to=9-6, thick]
	\arrow[color={rgb,255:red,92;green,214;blue,92}, no head, from=6-3, to=5-12, thick]
	\arrow[color={rgb,255:red,92;green,214;blue,92}, no head, from=6-3, to=8-7, thick]
	\arrow[color={rgb,255:red,92;green,92;blue,214}, no head, from=6-11, to=10-13, thick]
	\arrow[color={rgb,255:red,228;green,158;blue,78}, no head, from=7-2, to=7-10, thick]
	\arrow[color={rgb,255:red,228;green,158;blue,78}, no head, from=7-2, to=11-4, thick]
	\arrow[color={rgb,255:red,92;green,214;blue,92}, no head, from=8-7, to=11-12, thick]
	\arrow[color={rgb,255:red,92;green,214;blue,92}, no head, from=9-6, to=11-12, thick]
	\arrow[color={rgb,255:red,92;green,214;blue,92}, no head, from=10-5, to=11-12, thick]
	\arrow[color={rgb,255:red,228;green,158;blue,78}, no head, from=11-4, to=12-11, thick]
	\arrow[color={rgb,255:red,92;green,214;blue,92}, no head, from=11-12, to=5-12, thick]
	\arrow[color={rgb,255:red,214;green,92;blue,214}, no head, from=12-3, to=2-7, thick]
	\arrow[color={rgb,255:red,214;green,92;blue,214}, no head, from=12-3, to=13-10, thick]
	\arrow[color={rgb,255:red,228;green,158;blue,78}, no head, from=12-11, to=7-10, thick]
	\arrow[color={rgb,255:red,214;green,92;blue,214}, no head, from=13-2, to=3-6, thick]
	\arrow[color={rgb,255:red,214;green,92;blue,214}, no head, from=13-2, to=13-10, thick]
	\arrow[color={rgb,255:red,214;green,92;blue,214}, no head, from=13-10, to=4-13, thick]
\end{tikzcd}
};
\end{tikzpicture}
    \caption{Example in $\Delta \varphi\wr\bfS$ for some $\phi\colon[5]\to [3]$ and $\gamma\in \Sigma_4$.}
    \label{fig: new csg}
\end{figure}
\end{const}

\begin{thm}\label{csg}
\textup{\autoref{crossedforpar}} defines a crossed simplicial group $\Delta \bfGS$.
\end{thm}

\begin{proof}
We apply  \autoref{FL characterization of CSG}.
By construction, we have forced \autoref{eq: 1.h} to be true. It is straightfoward to verifiy \autoref{eq: 3.h} and \autoref{eq: 3.v}. \autoref{eq: 2.h} and \autoref{eq: 1.v} follow from the fact that the maps $(g_0, \ldots, g_n ; \gamma)^*$ are determined entirely by the crossed simplicial structure of $\Delta \bfS$.

Therefore we only have to verify \autoref{eq: 2.v}. We first argue that it is enough to verify the property on the generators $\delta_i$ and $\sigma_i$ of $\Delta$. Indeed, suppose we are given a sequence of groups $\{H_n\}_{n\geq 0}$
with set of operations $\phi^*$ and $h^*$ in which  \autoref{eq: 1.h} and \autoref{eq: 2.h} hold for all $\phi\in \Delta([m],[n])$ and $h\in H_n$. Suppose for fixed $\phi\in\Delta([m],[n])$ and $\psi\in \Delta([k],[m])$ one has verified equations: 
\[
\phi^*(rr')=\phi^*(r)\cdot (r^*\phi)(r'), \quad \quad \psi^*(ss')=\psi^*s \cdot (s^*\psi)^*(s'),
\]
for any  $r,r'\in H_n$ and $s,s'\in H_m$. Then for any $t,t'\in H_m$:
\begin{align*}
    (\phi \circ \psi)^*(tt') & = \psi^*(\phi^*(tt')) & \text{by \autoref{eq: 1.h}},\\
    & = \psi^* \Big( \phi^*t \cdot (t^*\phi)^*(t') \Big) \\
    & = \psi^*(\phi^*t) \cdot \left((\phi^*t)^*(\psi)\right)^* ((t^*\phi)^*(t'))\\
    & =  (\phi \circ \psi)^*t \cdot \Big(t^*\phi \circ (\phi^*t)^*(\psi)\Big)^*(t') & \text{by \autoref{eq: 1.h},}\\
    & = (\phi \circ \psi)^*(t) \cdot (t^*(\phi \circ \psi))^*(t') & \text{by \autoref{eq: 2.h}.}
\end{align*}
In other words, we have shown that if \autoref{eq: 2.v} is true for $\phi$ and $\psi$, then it is also true for $\phi\circ \psi$, provided \autoref{eq: 1.h} and \autoref{eq: 2.h} hold. Therefore, in order to show \autoref{eq: 2.v} holds for $\Delta \bfGS$ it suffices to check it on the generators $\delta_i$ and $\sigma_i$ of $\Delta$.

Let us show that \autoref{eq: 2.v} is true for $\delta_i\colon [n-1]\rightarrow [n]$. For any elements $(g_0, \ldots, g_n ; \gamma)$, $(g'_0, \ldots, g'_n ; \gamma')\in G\wr \Sigma_{n+1}$, we need to show:
\begin{align*}
\delta_i^* \Big( (g_0, \ldots, g_n ; \gamma) & \cdot (g'_0, \ldots, g'_n ; \gamma') \Big) \\
 &=\delta_i^*(g_0, \ldots, g_n ; \gamma) \cdot ((g_0, \ldots, g_n ; \gamma)^*\delta_i)^*(g'_0, \ldots, g'_n ; \gamma')\\
 & =\delta_i^*(g_0, \ldots, g_n ; \gamma) \cdot (\gamma^*\delta_i)^*(g'_0, \ldots, g'_n ; \gamma').
\end{align*}
On the left hand side, we have:
\begin{align*}
    &\delta_i^*\Big(  (g_0, \ldots, g_n ; \gamma)  \cdot (g'_0, \ldots, g'_n ; \gamma') \Big) \\
    & = \delta^*_i \Big( g_0g'_{\gamma^{-1}(0)}, \ldots, g_n g'_{\gamma^{-1}(n)} ; \gamma \gamma' \Big)\\
     & =\delta^*_i(g_0g'_{\gamma^{-1}(0)}, \ldots, g_n g'_{\gamma^{-1}(n)}; \id) \cdot \delta^*_i(e, \ldots, e ; \gamma \gamma')\\
     & = (g_0g'_{\gamma^{-1}(0)}, \ldots, g_{i-1}g'_{\gamma^{-1}(i-1)},g_{i+1}g'_{\gamma^{-1}(i+1)},  \ldots, g_n g'_{\gamma^{-1}(n)} ; \delta_i^*(\gamma \gamma')),
     \end{align*}
while on the right hand side, we have:
\begin{align*}
    & \delta_i^*(g_0, \ldots, g_n ; \gamma) \cdot (\gamma^*\delta_i)^*(g'_0, \ldots, g'_n ; \gamma')\\
    & = \delta^*_i(g_0, \ldots, g_n; \id) \cdot \delta^*_i(e, \ldots, e ; \gamma)\cdot \delta^*_{\gamma^{-1}(i)}(g_0', \ldots, g_n' ; \id)  \cdot\delta^*_{\gamma^{-1}(i)}(e,\ldots, e;\gamma')\\
    & = (g_0, \ldots, g_{i-1}, g_{i+1}, \ldots , g_n ; \delta^*_i\gamma) \cdot (g_0', \ldots, g'_{\gamma^{-1}(i-1)},g'_{\gamma^{-1}(i+1)}, \ldots, g'_n ; \delta^*_{\gamma^{-1}(i)} \gamma').
\end{align*}
Consequently, we obtain the desired equality by noticing that
\[\delta^*_i\gamma \cdot \delta^*_{\gamma^{-1}(i)}\gamma'=\delta_i^*\gamma \cdot (\gamma^*\delta_i)^*\gamma'=\delta_i^*(\gamma \gamma')\]
and that $\delta_i^*\gamma$ by definition makes the following diagram commute
\[
\begin{tikzcd}
    {[n]}\ar{r}{\delta_{\gamma^{-1}(i)}} \ar{d}[swap]{(\delta_i^*\gamma)}& {[n+1]} \ar{d}{\gamma}\\
    {[n]} \ar{r}[swap]{\delta_i} & {[n+1]},
\end{tikzcd}
\]
so deleting the element $g'_{\gamma^{-1}(i)}$ and then permuting the rest of the $g'$'s according to $\delta_i^*\gamma$  gives the same sequence of elements that we get by first permuting all the $g'$'s according to $\gamma$ and deleting the element $g'_{\gamma^{-1}(i)}$ in the $i$-th position.

Let us now show that \autoref{eq: 2.v} is true for degeneracies $\sigma_i\colon [n+1]\rightarrow [n]$.
For any $(g_0, \ldots, g_n ; \gamma), (g'_0, \ldots, g'_n ; \gamma')\in G\wr \Sigma_{n+1}$, we need to show:
\begin{align*}
\sigma_i^* \Big( (g_0, \ldots, g_n ; \gamma) & \cdot (g'_0, \ldots, g'_n ; \gamma') \Big) \\
 & =\sigma_i^*(g_0, \ldots, g_n ; \gamma) \cdot (\gamma^*\sigma_i)^*(g'_0, \ldots, g'_n ; \gamma').
\end{align*}
On the left hand side, we have:
\begin{align*}
& \sigma_i^* \Big( (g_0, \ldots, g_n ; \gamma)  \cdot (g'_0, \ldots, g'_n ; \gamma') \Big) \\
 & =\sigma_i\Big( g_0g'_{\gamma^{-1}(0)}, \ldots, g_n g'_{\gamma^{-1}(n)} ; \gamma \gamma' \Big)\\
 & =(g_0g'_{\gamma^{-1}(0)}, \ldots, g_ig'_{\gamma^{-1}(i)}, g_ig'_{\gamma^{-1}(i)}, \ldots, g_n g'_{\gamma^{-1}(n)};\tau \sigma^*_i(\gamma \gamma')),
\end{align*}
where $\tau$ is either $\id$ or $(i, i+1)$ depending on the value of the parity $\varphi(g_i g'_{\gamma^{-1}(i)})$. Note that since $\varphi$ is a homomorphism, this is the product of the parities $\varphi(g_i)\varphi(g'_{\gamma^{-1}(i)})$.
On the right hand side, we have: 
\begin{align*}
&\sigma_i^*(g_0, \ldots, g_n ; \gamma) \cdot (\gamma^*\sigma_i)^*(g'_0, \ldots, g'_n ; \gamma')\\
&=(g_0, \ldots, g_i, g_i, \ldots, g_n; \tau_g \sigma^*_i\gamma) \cdot (g_0', \ldots, g'_{\gamma^{-1}(i)}, g'_{\gamma^{-1}(i)}, \ldots, g_n ; \tau_{g'}\sigma^*_{\gamma^{-1}(i)}\gamma')  
\end{align*}
where $\tau_g$ is either $\id$ or $(i, i+1)$ depending on the value of $\varphi(g_i)$ and $\tau_{g'}$ is either $\id$ or $(\gamma^{-1}(i), \gamma^{-1}(i)+1)$ depending on the value of $\varphi(g'_{\gamma^{-1}(i)})$. We go through the possible cases.

\begin{description}
   \item[Case 1]
$\varphi(g_i)=1=\varphi'(g'_{\gamma^{-1}(i)})$. The left hand side becomes  
    \[
(g_0g'_{\gamma^{-1}(0)}, \ldots, g_ig'_{\gamma^{-1}(i)}, g_ig'_{\gamma^{-1}(i)}, \ldots, g_n g'_{\gamma^{-1}(n)};\sigma^*_i(\gamma \gamma')),
    \]
    while the right hand side becomes
    \[
(g_0, \ldots, g_i, g_i, \ldots, g_n;  \sigma^*_i\gamma) \cdot (g_0', \ldots, g'_{\gamma^{-1}(i)}, g'_{\gamma^{-1}(i)}, \ldots, g_n ; \sigma^*_{\gamma^{-1}(i)}\gamma').
    \]
     We can conclude these are equal by using that $\sigma^*_i(\gamma \gamma')=\sigma^*_i\gamma\cdot \sigma^*_{\gamma^{-1}(i)}\gamma'$ and the fact that $\sigma_i\gamma$ is entirely determined by the commutative diagram:
    \[
\begin{tikzcd}
    {[n+1]}\ar{r}{\sigma_{\gamma^{-1}(i)}} \ar{d}[swap]{\sigma_i^*\gamma}& {[n]}  \ar{d}{\gamma} \\
    {[n+1]}  \ar{r}[swap]{\sigma_i} & {[n]}.
\end{tikzcd}
    \]
   \item[Case 2] 
 $\varphi(g_i)=-1=\varphi(g'_{\gamma^{-1}(i)})$. The left hand side is as in Case 1, but the right hand side now becomes
  \begin{align*}
& (g_0, \ldots, g_i, g_i, \ldots, g_n; (i, i+1) \sigma^*_i\gamma)  \\ 
& \cdot (g_0', \ldots, g'_{\gamma^{-1}(i)}, g'_{\gamma^{-1}(i)}, \ldots, g_n ; (\gamma^{-1}(i), \gamma^{-1}(i)+1)\sigma^*_{\gamma^{-1}(i)}\gamma') 
  \end{align*}
We can conclude as in Case 1 since we have: 
  \[
(i, i+1)\cdot \sigma^*_i\gamma= \sigma_i^*\gamma \cdot  (\gamma^{-1}(i), \gamma^{-1}(i)+1).
  \]
  
   \item[Case 3] 
 $\varphi(g_i)=-1$ and $\varphi(g'_{\gamma^{-1}(i)})=1$. 
    The left hand side becomes: 
    \[
(g_0g'_{\gamma^{-1}(0)}, \ldots, g_ig'_{\gamma^{-1}(i)}, g_ig'_{\gamma^{-1}(i)}, \ldots, g_n g'_{\gamma^{-1}(n)};(i, i+1)\sigma^*_i(\gamma \gamma')),
    \]
    while the right hand side becomes
    \begin{align*}
& (g_0, \ldots, g_i, g_i, \ldots, g_n; (i, i+1) \sigma^*_i\gamma)   \\
& \cdot (g_0', \ldots, g'_{\gamma^{-1}(i)}, g'_{\gamma^{-1}(i)}, \ldots, g_n ; \sigma^*_{\gamma^{-1}(i)}\gamma').
  \end{align*}
  Since the $\gamma^{-1}(i)$-th and $(\gamma^{-1}(i)+1)$-th terms are identical, then we can conclude as in Case 1 and Case  2.

 \item[Case 4] 
$\varphi(g_i)=1$ and $\varphi(g'_{\gamma^{-1}(i)})=-1$. The left hand side is as in Case 3 while the right hand side becomes
  \begin{align*}
& (g_0, \ldots, g_i, g_i, \ldots, g_n;  \sigma^*_i\gamma)  \\ 
& \cdot (g_0', \ldots, g'_{\gamma^{-1}(i)}, g'_{\gamma^{-1}(i)}, \ldots, g_n ; (\gamma^{-1}(i), \gamma^{-1}(i)+1)\sigma^*_{\gamma^{-1}(i)}\gamma'). 
  \end{align*}
  We can again conclude as in Case 1 and Case 2. \qedhere
\end{description}
\end{proof}

\begin{rem}\label{contractible}
When $G$ is the trivial group, $\Delta \bfGS$ is just $\Delta \bfS$. When $\varphi=\id_{C_2}$, then $\Delta \id_{C_2}\wr \bfS$ is precisely the \emph{hyperoctahedral crossed simplicial group} $\Delta \bfH$ defined in \cite[\S 3.1]{FL91}. When $\varphi$ is the trivial group homomorphism, then we recover \cite[Theorem~3.10]{FL91}.  Moreover, we see from the discussion of~\cite[Example~6]{FL91} that the simplicial set $G\wr \Sigma_{\bullet+1}$ associated to $\Delta \varphi \wr \Sigma_{\bullet+1}$ is contractible for any parity $\varphi$. 
\end{rem}

\begin{rem}
After a draft of our paper appeared in preprint form, we were informed that the crossed simplicial groups from \autoref{crossedforpar} have been considered independently in work in progress of Daniel Graves and Sarah Whitehouse. 
\end{rem}

\begin{prop}
Let $\Delta \bfG$ be a crossed simplicial group. Then, for any $g\in G_n$, and any $0\leq i \leq n$, there exists $0\leq j \leq n$ such that $g^\ast \sigma_i=\sigma_j$ and $g^\ast \delta_i=\delta_j$ simultaneously.
\end{prop}

\begin{proof}
Let  $\sigma_i \colon [n+1]\to [n]$. Note that since $\sigma_i \delta_i=\id$, we have that $g^\ast\sigma_i \circ (\sigma_i^\ast g)^\ast \delta_i
= g^\ast(\sigma_i \delta_i) =\id$, thus $g^\ast \sigma_i$ is surjective, thus it must be a degeneracy $\sigma_j \colon [n+1]\to [n]$ for some $0\leq j\leq n$. 

Now let $\delta_i\colon [n-1]\to [n]$ and let $g^\ast \delta_i=\delta_{i_1}\dots\delta_{i_r}\sigma_{i_1}\dots\sigma_{i_r}$ be its factorization into degeneracies followed by face maps.  Now note that $(g^{-1}g)^\ast=\id$, but also we can compute 
\begin{align*}
(g^{-1}g)^\ast \delta_i&= (g^{-1})^\ast(\delta_{i_1}\dots\delta_{i_r}\sigma_{i_1}\dots\sigma_{i_r})\\
&=(g^{-1})^\ast (\delta_{i_1}\dots\delta_{i_r}\sigma_{i_1}\dots\sigma_{i_{r-1}}) \big((\delta_{i_1}\dots\delta_{i_r}\sigma_{i_1}\dots\sigma_{i_{r-1}})^\ast(g^{-1})\big)^\ast(\sigma_{i_r}).
\end{align*}
But since we have already shown that $\big((\delta_{i_1}\dots\delta_{i_r}\sigma_{i_1}\dots\sigma_{i_{r-1}})^\ast(g^{-1})\big)^\ast(\sigma_{i_r})$ has to be a degeneracy, we have a contradiction to the entire composite being injective. Thus there are no degeneracies that can occur in the factorization of  $g^\ast \delta_i\colon[n-1]\to [n]$, so it must be equal to a single face map $\delta_k$ for some $0\leq k \leq n$. 

Lastly we show that $j=k$. Using the simplicial identities $\sigma_i \delta_i=\id$ and $\sigma_i \delta_{i+1}=\id$ for $\delta_i, \delta_{i+1}\colon [n]\to[n+1]$, we have that $(\sigma_i^\ast g)^\ast\delta_i$ and $(\sigma_i^\ast g)^\ast\delta_{i+1}$ are  equal to  $\delta_j$ or $\delta_{j+1}\colon[n]\to
[n+1]$. But since $(\sigma_i^\ast g)^\ast$ is a bijection, one of these expressions is $\delta_j$ and the other one then has to be $\delta_{j+1}$.

Suppose that $(\sigma_i^\ast g)^\ast\delta_i=\delta_j$ and $(\sigma_i^\ast g)^\ast\delta_{i+1}=\delta_{j+1}$, so that we have the following commuting diagrams in $\Delta \bfG$ since $\delta_i^\ast\sigma_i^\ast g=g$.
   \[ \xymatrix{
    [n] \ar[r]^-{\delta_i} \ar[d]_-g & [n+1]\ar[d]^-{\sigma_i^\ast g}\\
   [n] \ar[r]^-{\delta_j} &[n+1]
   } \ \ \ \ \ \ \ \ \ \ \ \  \xymatrix{
    [n] \ar[r]^-{\delta_{i+1}} \ar[d]_-g & [n+1]\ar[d]^-{\sigma_i^\ast g}\\
   [n] \ar[r]^-{\delta_{j+1}}& [n+1]
   }\]
Precompose the top map with $\delta_i\colon [n-1]\to [n]$. Then note that for $\delta_k=g^\ast \delta_i$ we must have that $\delta_j\delta_k=\delta_{j+1}\delta_k$ since $\delta_i\delta_i=\delta_{i+1}\delta_i$. Thus $k=j$. The argument is completely symmetrical in the case when $(\sigma_i^\ast g)^\ast\delta_i=\delta_{j+1}$ and $(\sigma_i^\ast g)^\ast\delta_{i+1}=\delta_j$.
\end{proof}

\begin{const}\label{const: functor V}
For any crossed simplicial group $\Delta \bfG$, for every $n$, there is a homomorphism $V_n\colon G_n\to \Sigma_{n+1}$, where the permutation $V_n(g)=\theta_g \in \Sigma_{n+1}$ that $g$ maps to is defined by $\theta_g^{-1}(i) =j$, where given $0\leq i \leq n$, the value $j$ is defined in any of the following equivalent ways.
\begin{enumerate}
\item The value $j$ is the one determined by $g^\ast \sigma_i=\sigma_j$.
\item The value $j$ is the one determined by  $g^\ast \delta_i=\delta_j$.
\item Letting  $[0]\xrightarrow{\phi_i} [n]$  denote the map  that sends $0$ to $i$, the  value $j$ is the one determined by $g^\ast \phi_i=\phi_j$.
\end{enumerate}
Equivalently, the permutation $V_n(g)$ is the one determined by the action by left composition of $G_n$ on the set $\Hom_{\Delta \bfG}([0], [n])=G_0\times \Hom_{\Delta}([0], [n])$.  These homomorphisms $V_n$ assemble into a functor $$V= \text{Hom}_{\Delta \bfG}([0],-)/ G_0 \colon \Delta \bfG \to \text{Fin}$$ 
where $\text{Fin}$ denotes the category of finite sets and maps of finite sets. 
\end{const}

\begin{rem}
When $\Delta \bf G$ is the cyclic category, this is the functor used by Nikolaus--Scholze~\cite[pp~ 293]{NS18} in their definition of topological Hochschild homology.\footnote{Note that in \cite{NS18}, the object $[1]_{\Lambda}$ corresponds to $[0]$ in the standard model for the cyclic category.}
\end{rem}

\begin{cor}[{\cite[Proposition 1.7]{FL91}}]\label{crossed simplicial groups with faces and degeneracies}
A crossed simplicial group is a simplicial set $\bfG_{\sbt}$ such that for each $n\geq 0$,  $G_n$ is a group together with a group homomorphism $V_n\colon G_n\rightarrow \Sigma_{n+1}$ denoted $g\mapsto \theta_g$ such that:
\begin{enumerate}
    \item $\delta^*_i(gg')=\delta^*_i(g)\delta^*_{\theta_g^{-1}(i)}(g')$, $\sigma^*_i(gg')=\sigma^*_i(g)\sigma^*_{\theta_g^{-1}(i)}(g')$;
    \item 
    the following set diagrams are commutative:
    \[
    \begin{tikzcd}
        {[n-1]} \ar{r}{\delta_{\theta_g^{-1}(i)}} \ar{d}[swap]{\theta_{\delta^*_i(g)}} & {[n]} \ar{d}{\theta_g} & {[n+1]} \ar{r}{\sigma_{\theta_g^{-1}(i)}} \ar{d}[swap]{\theta_{\sigma^*_i(g)}} & {[n]}\ar{d}{\theta_g}\\
        {[n-1]} \ar{r}[swap]{\delta_i} & {[n]} & {[n+1]} \ar{r}[swap]{\sigma_i} & {[n]}
    \end{tikzcd}
    \]
\end{enumerate}
\end{cor}

\begin{rem}\label{rem: map of crossed simplicial group}
    Notice from \autoref{crossed simplicial groups with faces and degeneracies} that a functor $f\colon \Delta \bfG \rightarrow \Delta \mathbf{G'}$ is specified by a sequence of group homomorphisms $f_n\colon G_n\rightarrow G_n'$ such that:
\begin{enumerate}
    \item $\theta_g=\theta_{f_n(g)}$, for all $n\geq 0$, $g\in G_n$;
    \item $f_{n-1}(\delta^*_i(g))=\delta^*_i(f_n(g))$, for all $n\geq 1$, $g\in G_n$, and $0\leq i \leq n$;
    \item $f_{n+1}(\sigma^*_i(g))=\sigma^*_i(f_n(g))$, for all $n\geq 0$, $g\in G_n$, and $0\leq i \leq n$.
\end{enumerate}
\end{rem}

\begin{rem}\label{rem: functor from group with parities to crossed simplicial groups}
\autoref{crossedforpar} is natural in $(G,\varphi)$ in the sense that it defines a functor
\begin{align*}
  \Delta -\wr \mathbf{\Sigma}\colon \Grparity & \longrightarrow \CSG\\
    (G, \varphi) & \longmapsto \Delta \bfGS
\end{align*}
from the category $\Grparity$ of groups with parities to the full subcategory $\CSG$ of small categories $\Cat$ spanned by crossed simplicial groups.
Indeed, given groups with parities $(G, \varphi)$ and $(G', \varphi')$, and a group homomorphism $f\colon G\rightarrow G'$ compatible with the parities, i.e. $\varphi'\circ f=\varphi$, we can define a functor $\Delta \bfGS \rightarrow \Delta \varphi' \wr \mathbf{\Sigma}$ by specifying group homomorphisms
\begin{align*}
    G \wr \Sigma_{n+1} & \longrightarrow G' \wr \Sigma_{n+1}\\
    (g_0, \ldots, g_n ; \gamma) & \longmapsto (f(g_0), \ldots, f(g_n); \gamma)
\end{align*}
that are compatible with the crossed simplicial structure following \autoref{rem: map of crossed simplicial group}.
\end{rem}

\begin{thm}[\protect{\cite[Theorem 3.6]{FL91}}]\label{classification of csg}
For every crossed simplicial group $\Delta \bfG$, there is a canonical functor 
\[ 
\lambda \colon \Delta \bfG \longrightarrow \Delta \bfH
\]
whose essential image is a sub-crossed simplicial group $\Delta \bfG^{\prime}$ of $\Delta \bfH$. The kernels $G''_n$ of the group homomorphisms $\lambda_n \colon G_n\to G_n^{\prime}\subset \Sigma_{n+1}\wr C_2$ form a simplicial group. 
\end{thm}
The subcrossed simplicial groups $\Delta \bfG^{\prime}$ of $\Delta \bfH$ were completely classified by \cite{Abo87} (cf. 
\cite[Proposition 1.5]{Kra87} and \cite[Theorem 3.6]{FL91}).

Using the classification of crossed simplicial groups from  \autoref{classification of csg}, we produce the following family of groups with parity. 

\begin{exm}\label{examples of groups with parity}
Let $\Delta\bfG$ be a crossed simplicial group. Then by \autoref{classification of csg}, there is a group homomorphism 
\[ \lambda_0 \colon G_0^{\op} \longrightarrow H_0=C_2,\] 
which we call the \emph{canonical parity} of the group $G_0^{\op}=\Aut_{\Delta \bfG}([0])$. 
Explicitly, note that an element $g\in G_0$ determines an element $\sigma_0^\ast g$ for the unique map $\sigma_0\colon [1]\to [0]$. The value $\lambda_0(g)$ is defined as 1 or -1, depending on whether $(\sigma_0^\ast g)^\ast(\delta_0)$ gives $\delta_0$ back or $\delta_1$, where $\delta_0, \delta_1$ are the two maps $[0]\to [1]$ skipping $0$ and $1$, respectively.
We call $(G_0^{\op},\lambda_0)$ the \emph{group with parity associated to the crossed simplicial group $\Delta \bfG$}. 
\end{exm}

\begin{rem}\label{rem: functor from crossed simplicial groups to groups with parities}
The definition of the canonical parity $(-)_0\colon \lambda_0\colon G_0^{\op}\rightarrow C_2$ of \autoref{examples of groups with parity} is natural in $\Delta \mathbf{G}$, in the sense that it defines a functor 
\begin{align*}
  (-)_0\colon\CSG &\longrightarrow \Grparity\\
  \Delta \mathbf{G} & \longmapsto (G_0^{\op}, \lambda_0)
\end{align*}  
from the full subcategory $\CSG$ of small categories $\Cat$, spanned by crossed simplicial groups, to the category of $\Grparity$ of groups with parities. 
Indeed, given a functor $f\colon \Delta \mathbf{G}\rightarrow \Delta \mathbf{G}'$, \autoref{rem: map of crossed simplicial group} shows that we can obtain a group homomorphism $f_0\colon G_0^{\op}\rightarrow (G_0')^{\op}$ compatible with the canonical parities in the sense that the diagram commutes:
\[
\begin{tikzcd}
    G_0^{\op} \ar{rr}{f_0}\ar{dr}[swap]{\lambda_0} & & (G_0')^{\op} \ar{dl}{\lambda_0'}\\
    & C_2.
\end{tikzcd}
\]
We observe in \autoref{adjunction} that the functor $(-)_0$ is adjoint to the functor $\Delta -\wr \mathbf{\Sigma}$ and we give an explicit description of the unit of this adjunction in \autoref{canonical map}.
\end{rem}

Let $\Delta \bfG$ be a crossed simplicial group. Consider the wreath product $G_0^{\op} \wr \Sigma_{n+1} $.
Let $g\in G_n^{\op}$, and let $[0]\xrightarrow{\phi_i} [n]$  be the map  that sends $0$ to $i$. Consider the corresponding canonical factorization: 
\[\xymatrix{
[0] \ar[d]_-{\phi_i^\ast g} \ar[r]^-{\phi_i} & [n]\ar[d]^-{g}\\
[0]\ar[r]_-{g^{\ast} \phi_i} &[n]
}\]
where this square completion agrees with the one in previous conventions because the vertical arrows are invertible. 
As before, define the permutation $\theta_g\in \Sigma_{n+1}$ by $\theta_g^{-1}(i) = (g^\ast \phi_i) (0)$, 
Define a map $L_n\colon G_n^{\op}\to  G_0^{\op} \wr \Sigma_{n+1}$
as follows:
$$L_n(g)=( \phi_0^\ast g,  \phi_1^\ast g, \dots,  \phi_n^\ast g; \theta_g).$$

\begin{prop}[{\protect{\cite[Proposition I.37]{DK15}}}]\label{homomorphismLn}
The above function $L_n\colon G_n^{\op}\to  G_0^{\op} \wr \Sigma_{n+1}$  is a group homomorphism.
\end{prop}

\begin{proof}
To check that this is a group homomorphism, consider the composites 
\[
\xymatrix{
[0] \ar[d]_-{\phi_i^\ast g} \ar[r]^-{\phi_i} & [n]\ar[d]^-g\\
[0]\ar[d]_-{(g^\ast\phi_i)^\ast h} \ar[r]_-{g^\ast \phi_i } &[n]\ar[d]^-h\\   
[0] \ar[r]_-{h^\ast(g^\ast\phi_i)} &[n]
}  
\xymatrix{
[0]\ar[dd]_-{= \ \ \ (\phi_i^\ast g)((g^\ast\phi_i)^\ast h)}  \ar[rr]^-{\phi_i} && [n]\ar[dd]^-{h\circ g}\\
&&\\
[0]\ar[rr]_-{(gh)^\ast \phi_i} &&[n]
}
\] 
\noindent Note that by definition $g^\ast\phi_i= \phi_{\theta_g^{-1}(i)}$, $h\circ g=g\cdot h$ in $G^{\op}$, and  $\theta_{gh}^{-1}(i)=\theta_h^{-1}(\theta_g^{-1}(i))=(\theta_g\theta_h)^{-1}(i)$, thus
\begin{align*}
 L_n(g)L_n(h) &= (\phi_0^\ast g,  \phi_1^\ast g, \dots,  \phi_n^\ast g; \theta_g) (  \phi_0^\ast h,  \phi_1^\ast h, \dots,  \phi_n^\ast h; \theta_h)\\
 &= \big( (\phi_0^\ast g)  (\phi_{\theta_g^{-1}(0)}^\ast h),\ (\phi_1^\ast g)  (\phi_{\theta_g^{-1}(1)}^\ast h),  \dots, (\phi_n^\ast g)  (\phi_{\theta_g^{-1}(n)}^\ast h); \theta_g \theta_h\big)\\
 &= L_n(gh).  \qedhere
 \end{align*}
\end{proof}

\begin{prop}\label{canonical map}
When $(G_0^{\op},\lambda_0)$ is the group with parity associated to a crossed simplicial group $\Delta \bfG$, the functor $\lambda \colon \Delta \bfG\to \Delta \bfH$ has a canonical factorization through a functor $\tilde{\lambda}\colon \Delta \bfG\to \Delta \lambda_0 \wr \bfS$.
\end{prop}

\begin{proof}
We have a group homomorphism $L_n\colon G_n^{\op}\to  G_0^{\op}\wr \Sigma_{n+1} $. 
As noted in \autoref{rem: map of crossed simplicial group}, to result in a functor 
$\Delta \bfG\to \Delta \lambda_0 \wr \bfS$ 
we only need to verify the three equations in the remark.
To show that $\theta_g=\theta_{L_n(g)}$ for all $g\in G_n$, notice that we have for all $0\leq i\leq n$:
\begin{align*}
   \phi_{\theta^{-1}_{L_n(g)}(i)} & =L_n(g)^*\phi_i  = \big(\phi_0^*g, \dots, \phi_n^*g; \theta_g\big)^* \phi_i 
     = \theta_g^*\phi_i 
     = \phi_{\theta^{-1}_g(i)}
\end{align*}
Thus $\theta^{-1}_{L_n(g)}(i)=\theta^{-1}_g(i)$ for all $0\leq i\leq n$, and thus $\theta_g=\theta_{L_n(g)}$.

 We now need to show $L_{n-1}(\delta^*_i(g))=\delta_i^*(L_n(g))$ for $n\geq 1$, $g\in G_n$ and $0\leq i \leq n$.
 By definition, we have on the one hand:
 \begin{align*}
     \delta^*_i(L_n(g))  = \delta^*_i(\phi_0^*g, \dots, \phi_n^*g; \theta_g)
     = \big( \phi_0^*g, \dots, \phi_{i-1}^*g, \phi_{i+1}^*g, \dots, \phi^*_ng; \delta^*_i \theta_g\big).
 \end{align*}
 On the other hand, we have:
 \[
 L_{n-1}(\delta_i^* g) = \big( \phi_0^*(\delta^*_i g), \dots, \phi^*_{n-1}(\delta^*_i g) ; \theta_{\delta^*_ig}\big)
 \]
 Notice that for $0\leq k \leq n-1$, we have:
 \[
 \delta_i\circ \phi_k = \begin{cases} 
    \phi_k & \text{if } k<i \\ 
    \phi_{k+1}& \text{if }k\geq i.
\end{cases}
 \]
 Therefore we obtain:
 \[
 \phi_k^*(\delta_i^*g)=(\delta_i\circ \phi_k)^*(g) = \begin{cases} 
    \phi_k^*g & \text{if } k<i \\ 
    \phi_{k+1}^*g& \text{if }k\geq i.
\end{cases}
 \]
 So we are left to prove that $\delta^*_i \theta_g=\theta_{\delta^*_ig}$.
 Recall that the permutation $\delta_i^* \theta_g$ in $\Delta \bfS$ is defined uniquely by sending the subset ${\delta^*_i}^{-1}(k)\subseteq [n]$ to the set $(\theta_g^*\delta_i)^{-1}(\theta_g(k))=\delta_{\theta_g^{-1}(i)}^{-1}(\theta_g(k))$ and preserving its order. Therefore we obtain:
 \[
 \delta_i^*\theta_g(k) =\begin{cases}
     \theta_g(k) & \text{if }k<i, \theta_g(k)<\theta_g(i)\\
     \theta_g(k)-1 & \text{if }k<i, \theta_g(k)\geq \theta_g(i)\\
     \theta_g(k+1) & \text{if } k\geq i, \theta_g(k+1)<\theta_g(i)\\
     \theta_g(k+1)-1 & \text {if }k\geq i, \theta_g(k+1)\geq \theta_g(i).
 \end{cases}
 \]
 Meanwhile, recall that $\theta_{\delta^*_i g}^{-1}(k)={(\delta^*_i g)}^* \phi_k(0)$. We have the commutative diagram 
 \[
 \begin{tikzcd}[column sep=large]
     {[0]} \ar{r}{{(\delta^*_ig)}^*\phi_k} \ar{d}[swap]{\phi_k^*\delta_i^*g} & {[n-1]} \ar{r}{\delta_{\theta^{-1}_g(i)}} \ar{d}{\delta_i^*g} & {[n]}\ar{d}{g}\\
     {[0]}\ar{r}{\phi_k} & {[n-1]}\ar{r}{\delta_i} & {[n]}
 \end{tikzcd}
 \]
 in $\Delta \bfG$. By uniqueness of the factorization in $\Delta \bfG$, we obtain that:
 \[
 \delta_{\theta_g^{-1}(i)} \circ {(\delta^*_i g)}^* \phi_k = \begin{cases}
     g^*\phi_k & \text{if }k<i\\
     g^*\phi_{k+1}& \text{if }k\geq i.
 \end{cases}
 \]
 As $\delta_{\theta_g^{-1}(i)}$ skips the entry $\theta_g^{-1}(i)$, we get then:
 \[
 {(\delta^*_i g)}^* \phi_k (0) =\begin{cases}
     g^*\phi_k(0) & \text{if }k<i, \theta_g(k)<\theta_g(i)\\
     g^*\phi_k(0)-1 & \text{if }k<i, \theta_g(k)\geq \theta_g(i)\\
    g^*\phi_{k+1}(0) & \text{if } k\geq i, \theta_g(k+1)<\theta_g(i)\\
     g^*\phi_{k+1}(0)-1 & \text {if }k\geq i, \theta_g(k+1)\geq \theta_g(i).
 \end{cases}
 \]
 Thus we have just shown $\delta^*_i \theta_g=\theta_{\delta^*_ig}$, concluding that $L_{n-1}(\delta^*_i(g))=\delta_i^*(L_n(g))$.

 We now need to show $L_{n+1}(\sigma^*_i(g))=\sigma^*_i(L_n(g))$, for all $n\geq 0$, $g\in G_n$, and $0\leq i \leq n$.
 Be definition, we have on one hand:
 \begin{align*}
     \sigma_i^*(L_n(g)) & = \sigma^*_i \big( \phi_0^*g, \dots, \phi_n^* g ; \theta_g\big)\\
     & = \big( \phi_0^*g, \dots, \phi_i^*g, \phi_i^*g, \dots, \phi_n^*g ; \tau \sigma^*_i\theta_g \big),
 \end{align*}
 where the permutation $\tau$ is defined by:
 \[
 \tau= \begin{cases}
     \id & \text{if }\lambda_0(\phi_i^*g)=1\\
     (i, i+1) & \text{if }\lambda_0(\phi_i^*g)=-1.
 \end{cases}
 \]
 On the other hand, we have:
 \[
 L_{n+1}(\sigma^*_ig) = \big( \phi_0^*(\sigma_i^* g), \dots, \phi_{n+1}^*(\sigma_i^* g) ; \theta_{\sigma^*_i g}\big). 
 \]
 Notice that for $0\leq k \leq n+1$, we have:
 \[
 \sigma_i\circ \phi_k= \begin{cases}
     \phi_k & \text{if }k\leq i\\
     \phi_{k-1} & \text{if }k>i.
 \end{cases}
 \]
 Therefore we obtain:
 \[
 \phi_k^*(\sigma^*_i g)=(\sigma_i \circ \phi_k)^*(g)= \begin{cases}
     \phi_k^*g & \text{if }k\leq i\\
     \phi_{k-1}^*g & \text{if }k>i.
 \end{cases}
 \]
 So we are left to prove that $\tau \sigma^*_i\theta_g=\theta_{\sigma^*_i g}$.
 Recall that the permutation $\sigma^*_i\theta_g$ in $\Delta \bfS$ is defined uniquely by sending the subset ${\sigma^*_i}^{-1}(k)\subseteq [n]$ to the set $(\theta_g^*\sigma_i)^{-1}(\theta_g(k))=\sigma^{-1}_{\theta^{-1}_g(i)}(\theta_g(k))$ and preserving its order. Therefore we obtain:
 \[
 \sigma_i^*\theta_g(k) =\begin{cases}
     \theta_g(k) & \text{if }k\leq i, \theta_g(k)\leq \theta_g(i)\\
     \theta_g(k)+1 & \text{if }k \leq i, \theta_g(k)> \theta_g(i)\\
     \theta_g(k-1) & \text{if } k> i, \theta_g(k-1)< \theta_g(i)\\
     \theta_g(k-1)+1 & \text {if }k> i, \theta_g(k-1)\geq \theta_g(i).
 \end{cases}
 \] 
So $\tau \sigma^*_i\theta_g$ satisfies these same formulas for $k\neq i, i+1$ while $i, i+1$ map to $\theta_g(i)$, $\theta_g(i)+1$ if $\tau=\id$ or to 
$\theta_g(i)+1, \theta_g(i)$ if 
$\tau=(i,i+1)$. 

On the other hand, recall that $\theta^{-1}_{\sigma^*_ig}(k)=(\sigma^*_i g)^*\phi_k(0)$.
 We have the commutative diagram in $\Delta \bfG$:
 \[
 \begin{tikzcd}[column sep=large]
     {[0]} \ar{r}{(\sigma^*_i g)^*\phi_k} \ar{d}[swap]{\phi_k^*\sigma_i^*g} & {[n+1]} \ar{r}{\sigma_{\theta^{-1}_g(i)}}\ar{d}{\sigma_i^*g} & {[n]}\ar{d}{g}\\
     {[0]} \ar{r}{\phi_k}& {[n+1]}\ar{r}{\sigma_i} & {[n]}.
 \end{tikzcd}
 \]
 Just as in the cases for cofaces, noting that $\sigma_i\phi_k=\phi_k$ for $k\leq i$ and $\sigma_i\phi_k=\phi_{k-1}$ for $k>i$,  since $\sigma_{\theta^{-1}_g(i)}$ repeats the entry $\theta^{-1}_g(i)$, we obtain the following formulas for $k\neq i, i+1$:
  \[
 {(\sigma^*_i g)}^* \phi_k (0) =\begin{cases}
     g^*\phi_k(0) & \text{if }k<i, \theta_g(k)< \theta_g(i)\\
     g^*\phi_k(0)+1 & \text{if }k<i, \theta_g(k)> \theta_g(i)\\
    g^*\phi_{k-1}(0) & \text{if } k>i+1, \theta_g(k-1)< \theta_g(i)\\
     g^*\phi_{k-1}(0)+1 & \text {if }k> i+1, \theta_g(k-1)> \theta_g(i).
 \end{cases}
 \]
Therefore for $k\neq i, i+1$ these formulas agree precisely with those above. Noting that for $k=i, i+1$, we have $\phi_k\sigma_i=\phi_i$, from the above diagram we already know that $(\sigma^*_i g)^* \phi_k (0)$ must be $\theta_g(i)$ or $\theta_g(i)+1$ and we need to argue that this depends on $\tau$.  We now argue for $k=i$; the case $k=i+1$ is similar. 
 Applying the functor $\lambda\colon \Delta \bfG\rightarrow \Delta \bfH$ to obtain the commutative diagram in $\Delta \bfH$:
\[
 \begin{tikzcd}[column sep=large]
     {[0]} \ar{r}{(\sigma^*_i g)^*\phi_i} \ar{d}[swap]{\lambda_0(\phi_i^*g)} & {[n+1]} \ar{r}{\sigma_{\theta^{-1}_g(i)}}\ar{d}{\lambda_{n+1}(\sigma_i^*g)} & {[n]}\ar{d}{\lambda_n(g)}\\
     {[0]} \ar{r}{\phi_i}& {[n+1]}\ar{r}{\sigma_i} & {[n]}
 \end{tikzcd}
 \]
By definition, the permutation showing up in $\lambda_n(g)\in C_2\wr \Sigma_{n+1}$ is $\theta_g$, and the $i$th element in the tuple of elements of $C_2$ is $\tau$. Therefore, using the composition rule in $\Delta \bfH$, we see that the permutation showing up in  $\lambda_{n+1}(\sigma_i^*g)$ is precisely $\tau \sigma^*_i \theta_g$. Again, using the composition rule in $\Delta \bfH$, we see that the value $(\sigma^*_i g)^*\phi_i(0)$ is
\[
(\sigma^*_i g)^*\phi_i(0) = \begin{cases}
    \theta_g(i) & \text{if } \tau=\id,\\
    \theta_g(i)+1 &  \text{if }\tau=(i, i+1),
\end{cases}
\]
so it precisely depends on the parity of $\lambda_0(\phi_i^*g)\in C_2$.
We conclude  $\tau \sigma^*_i\theta_g$ equals $\theta_{\sigma^*_i g}$ as desired.
\end{proof}

\begin{rem}\label{adjunction}
Following notations from \autoref{rem: functor from group with parities to crossed simplicial groups} and \autoref{rem: functor from crossed simplicial groups to groups with parities}, the map $\Delta \mathbf{G}\rightarrow \Delta\lambda_0\wr \mathbf{\Sigma}$ of \autoref{canonical map} is the unit of an adjunction:
 \[
 \begin{tikzcd}[column sep =large]
     (-)_0\colon\CSG \ar[bend left]{r}
     \ar[phantom, "\perp" description,xshift=0.1ex]{r}& \Grparity \colon {\Delta -\wr \mathbf{\Sigma}.} \ar[bend left]{l}
 \end{tikzcd}
 \]
\end{rem}

Lastly we introduce a variant of the category $\Delta \bfGS$ where we adjoin an initial object to it in order to equip it with a monoidal structure.  Our goal is to characterize twisted $G$-monoids in $\cC$ as monoidal functors from this monoidal category into $\cC$.

\begin{defin}[Whiskering]\label{def:whiskering}
 Given a category $\cC$, we define $\cC_+$ to be the category with objects $\{e\}\cup \mathrm{ob}(\cC)$ and morphisms
\[ \mathrm{Hom}_{\cC_+}(a,b)=\begin{cases}
\cC(a,b) & \text{ if } a,b\in \mathrm{ob}(\cC) \\ 
\id_{e} & \text{ if } a,b=e \\ 
*  & \text{ if } a=e,b\in \text{ob}(\cC)  \\
\emptyset  & \text{ if } a\in \text{ob}(\cC),b=e 
\end{cases}
\]
where composition is defined in the evident way. In the case of $\cC=\Delta \bfG$ we write $\langle 0\rangle\coloneqq e$ and $\nD \coloneqq [n-1]$ otherwise. 
\end{defin}

Note that $\Delta_+$ is a monoidal category with monoidal product given by the disjoint union and relabeling of elements $[n-1]\sqcup [m-1]=[n+m-1]$, and defining $e$ to act as the neutral element. Using our notation, this is written $\nD +\mD =\nplusmD$. On morphisms $f,g\in \Delta$, the symmetric monoidal structure is given by $f\sqcup g$. Note that the associativity and unitality are strict. As observed in \cite[Proposition~6.3]{Gra22}, this monoidal structure extends to a symmetric monoidal structure on $\Delta \bfH_+$, where the symmetry is given by the permutation that swaps the two blocks in the disjoint union\footnote{Since the associativity and unitality are strict, this is a symmetric strict monoidal category, also called a permutative category.}. Completely analogously, the category $\Delta \bfGS_+$ can be equipped with a symmetric monoidal structure. We record this as a lemma.

\begin{lem}\label{lemma:pro}
The category $\Delta \bfGS_+$ is a symmetric monoidal category with monoidal product defined on objects by
\(
    \nD +\mD=\nplusmD.
\)

\end{lem}

\section{Twisted \texorpdfstring{$G$}{TEXT}-actions}\label{sec:twisted-G-actions}
Rings with twisted $G$-action are a simultaneous generalization of rings with anti-involution and rings with $G$-action. In \autoref{sec:twisted-G-objects}, we give a formal definition of rings with twisted $G$-action in 1-categories and in \autoref{subsection: twisted G-rings as algebras} we show that these arise as algebras over a certain $G$-operad.  We then extend these constructions to the setting of symmetric monoidal $\infty$-categories in \autoref{sec:twisted-G-infinity}, which we will then use throughout the rest of the paper. 

\subsection{Twisted \texorpdfstring{$G$}{TEXT}-objects}\label{sec:twisted-G-objects} 
Given a finite group $G$, we write $BG$ for the classifying space of $G$. We follow \cite[Section I.3]{DK15} in defining objects with twisted $G$-action with respect to a group with parity in the sense of \autoref{gppar}. More generally, we have the following definition.
 
 \begin{defin}
     A category $\cC$ with parity is a functor $\cC\to BC_2$. We call morphisms that map to $1$ even and morphisms that map to $-1$ odd. A functor of categories with parity is a functor that is compatible with the maps to $BC_2$.
\end{defin}

If $(G, \varphi)$ is a group with parity, then $BG\xrightarrow{B\varphi} BC_2$ is a category with parity. If $\varphi$ is non-trivial then $B\varphi$ is a Grothendieck op-fibration. 

Suppose $\cC$ is a category with $C_2$-action, i.e. a functor $BC_{2}\to \mathrm{Cat}$ where $\mathrm{Cat}$ is the (large) category of (small) categories. Consider the associated Grothendieck construction $C_2\rtimes \cC$. Recall that this category has the same objects as $\cC$ and a morphism $C_0\to C_1$ is a pair $(f,g)$ of an element $g \in C_2$ and a
  morphism $f\colon gC_0 \to C_1$ in $\cC$. The composition is given by
$ (f_1,g_1) \circ (f_2,g_2) = (f_1 \circ g_1(f_2),\ g_1g_2)$. The category $C_2\rtimes \cC$ has a canonical parity $C_2\rtimes \cC\to BC_2$. Note that this is a Grothendieck op-fibration.

\begin{defin}\label{catpar}
    Let $(G, \varphi)$ be a group with parity and let $\cC$ be a category with $C_2$-action. A \emph{twisted $G$-object} in $\cC$ is a functor of categories with parity $BG\to C_2\rtimes \cC.$
\end{defin}

\begin{defin}
    Given a group with parity $(G, \varphi)$, a \textit{category with twisted $G$-action} is a twisted $G$-object in $\Cat$, the category of categories, with $C_2$-action given by $\cC\mapsto \cC^{\op}$.
    We denote by $\Cat^\varphi=\Fun_{BC_2}(BG, C_2\rtimes\Cat)$ the category of small categories with twisted $G$-actions. 
If $\varphi\colon C_2\to C_2$ is $\id_{C_2}$ we denote it by $\Cat^\tau$ and if $\varphi\colon C_2\to C_2$ is the trivial homomorphism, we denote it by $\Cat^{BC_2}$. We will use the same terminology in the $\infty$-categorical context later.
\end{defin}

Essentially, a category with twisted $G$-action is a category $\cC$ on which $g$ acts by a covariant functor $g\colon \cC\to\cC$ if $g$ is even and by a functor $g\colon \cC^\op\to \cC$ if $g$ is odd.

\begin{defin}\label{def:twisted-G-monoid}
Let $(G, \varphi)$  be a group with parity and let $\Mon_{\cC}$ be the category of monoids in a symmetric monoidal category $\cC$. Consider the anti-involution on $\Mon_{\cC}$ given by $M\mapsto M^{\op}$, where $M^{\mathrm{op}}$ is the monoid defined using opposite multiplication.  
A \emph{monoid with twisted $G$-action} $M$ in $\cC$ is an object in $\Mon_{\cC}$ with twisted $G$-action.
We denote by $\Mon^\varphi(\cC)$ the induced category of monoids with twisted $G$-action in $\cC$.
\end{defin}

In this section, we follow the $1$-categorical convention of referring to monoid objects in a symmetric monoidal $1$-category $\cC$ as monoids. In $\infty$-categorical language, we would only refer to such objects as monoids in $\cC$ if the symmetric monoidal structure on $\cC$ were Cartesian. For the Cartesian monoidal structure in the $\infty$-category of spaces, monoids are usually referred as $\mathbb{E}_1$-spaces.

Note that the data of a monoid with twisted $G$-action in a category $\cC$ is a monoid in $\cC$ which is also an object in $\cC$ with $G$-action, and where the action by an element $g$ is through a monoid homomorphism or anti-homomorphism, depending on whether $g$ is even or odd respectively, where the actions compose compatibly. In particular, a discrete ring with twisted $G$-action is a ring $R$ on which $g$ acts by ring homomorphisms if $g$ is even and by ring anti-homomorphisms  if $g$ is odd.

\begin{rem}
Rings with twisted $G$-action have also been considered by Koam and Pirashvili \cite{Koa18,KP18}, under the name oriented algebras. It would be interesting compare their constructions to some of the constructions appearing in this paper.
\end{rem}

\begin{exm}
   When $\varphi$ is the trivial homomorphism, a twisted $G$-monoid in abelian groups is a ring with $G$-action. When $\varphi=\id_{C_2}$ then a twisted $C_2$-monoid $R$ in abelian groups is a ring with anti-involution.
\end{exm}

\begin{exm}
 Let $q \colon C_4\rightarrow C_2$ be the quotient homomorphism. Then a monoid with a twisted $C_4$-action is a monoid $M$ together with a monoid homomorphism $t\colon M^\op\rightarrow M$, such that $t^4=\id$. \footnote{If we instead chose the trivial morphism $C_{4}\to C_{2}$ then a monoid with twisted $C_{4}$-action would be a monoid homomorphism $t\colon M\rightarrow M$, such that $t^4=\id$, so the notion of twisted $C_{4}$-action depends on the parity $C_{4}\to C_{2}$ although it is supressed from the terminology.}
 Of course, every monoid with anti-involution is a monoid with $C_4$-twisted action, but not every such monoid need to be from a monoid with anti-involution.
 For instance, the quaternions $\mathbb{H}$ can be given the structure of an $\mathbb{R}$-algebra with a twisted $C_4$-action that is not a twisted $C_2$-action by defining the action of the generator $t$ of $C_4$ on $(1,i, j,k)\mapsto (1, j, -i,-k)$ and extending $\mathbb{R}$-linearly. We can check that $t^4=\id$ but $t^2\neq \id$, and $t,t^3$ act by antihomomorphisms, whereas $t^2$ acts by a homomorphism.  
\end{exm}

\subsection{Rings with twisted \texorpdfstring{$G$}{TEXT}-action as algebras over an operad}\label{subsection: twisted G-rings as algebras}
Given an operad $\cO$ in spaces, recall that there is an associated category of operators $\cO^\otimes$ introduced in~\cite{MT78} and generalized in~\cite[2.1.1.7]{HA}. The objects of $\cO^\otimes$ are the finite sets $\nF=\{0, 1, \dots, n\}$ pointed at $0$, and the morphisms are given by
\[ \cO^\otimes(\mF, \nF) = \coprod_{\phi\in\Fin_\ast (\mF,\nF)} 
\prod_{1\leq j\leq n} \cO(|\phi^{-1}(j)|).  \]
For $(\phi,c_1,\dots, c_n)\colon \mF\to \nF$ and $(\psi,d_1, \dots, d_m)\colon \kF \to \mF$, composition is defined as
\[ (\phi,c_1,\dots, c_n)\circ (\psi,d_1, \dots, d_m) = \left(\phi\circ \psi, 
\prod_{1\leq j\leq n}\gamma\left(c_j;\prod_{\phi(i) = j} d_i\right)\sigma_j\right), \]
where $\gamma$ denotes the structural maps of the operad $\cO$.  The $d_i$ for which $\phi(i) =j$ 
are ordered by the natural order on their indices $i$
and $\sigma_j$ is that permutation of $|(\phi\circ\psi)^{-1}(j)|$ letters which converts 
the natural ordering of $(\phi\circ\psi)^{-1}(j)$ as a subset of $\{1,\dots,k\}$ to
its ordering obtained by regarding it as $\coprod_{\phi(i)=j}\psi^{-1}(i)$, so ordered
that elements of $\psi^{-1}(i)$ precede elements of $\psi^{-1}(i')$ if $i< i'$ and 
each $\psi^{-1}(i)$ has its natural ordering as a subset of $\{1,\dots,k\}$. 

We recall the definition of the symmetric monoidal envelope of an operad from \cite[2.2.4.1]{HA}. See~\cite[\S3]{Law21} for further discussion. 
\begin{defn}
    The symmetric monoidal envelope $\Env(\cO)$, sometimes denoted $\Env(\cO^\otimes)$ or $\cO^{\otimes}_{\textup{act}}$, of an operad $\cO$ is the subcategory of $\cO^\otimes$ for which the morphisms $(\alpha, \{\preccurlyeq_i\}_{1\leq i \leq n})\colon\mF \rightarrow \nF$ require $\alpha$ to be active, namely $\alpha^{-1}(0)=\{0\}$. Thus $\alpha$ is really the data of a map in $\Fin$ instead of $\Fin_*$. 
    Its symmetric monoidal structure is determined on objects by combining $\nF$
 and $\mF$ to $\nF+\,\mF$.    
\end{defn}

It is shown in~\cite[Lemma~4.2]{MT78} that an algebra $X$ over the operad $\cO$ in spaces is precisely a functor $\cO^\otimes\to \mathrm{Top}$ for which  $\nF \mapsto X^{\times n}$, where $\mathrm{Top}$ is the category of based topological spaces.
In fact, this applies more generally to an algebra over an operad $\cO$ in a general symmetric monoidal category $\cC$ and in $\infty$-categories this is built into the definition, cf.~\cite[2.1.2.7]{HA}. 
Moreover, there is a symmetric monoidal product on $\Env(\cO)$ such that strong symmetric monoidal functors $\Env(\cO)\to \cC$ correspond precisely to $\cO$-algebras in $\cC$, see \cite[2.2.4.9]{HA}.

\begin{exm}
    The \emph{associative operad} $\Assoc$ is defined by $\Assoc(n)=\Sigma_n$, where each permutation of $\Assoc(n)$ can be instead thought of as an ordering $\preccurlyeq_n$ on the set $\{1,\dots, n\}$. With this interpretation, as decribed in \cite[4.1.1.4]{HA}, the category of operators $\Assoc^\otimes$ has objects the same as $\Fin_*$, and given a pair of objects $\mF$, $\nF$ in $\Fin_*$, a morphism $\mF \rightarrow \nF$ in $\Assoc^\otimes$ consists of a map $\alpha\colon \mF \rightarrow \nF$ of finite pointed sets together with a linear ordering $\preccurlyeq_i$ on each inverse image $\alpha^{-1}(i) \subseteq \mF$, for $i=1, \ldots, n$.
 Given a pair of morphisms
  $\Big(\beta, \{ \preccurlyeq'_j\}_{1\leq j \leq p}\Big)\colon\nF \rightarrow \pF$
 and $\Big(\alpha, \{\preccurlyeq_i\}_{1\leq i \leq n}\Big)\colon\mF \rightarrow \nF$, 
the composition rule in the category of operators then gives the map of finite pointed sets $\beta \circ \alpha\colon \mF \rightarrow \pF$ together with the linear ordering $\{\preccurlyeq_j''\}_{1\leq j \leq p}$ defined as follows. For each $j=1, \ldots, p$, given $a,b \in (\beta\circ \alpha)^{-1}(j)\subseteq \mF$, we have $a \preccurlyeq_j'' b$ if and only if $\alpha(a) \preccurlyeq_j' \alpha (b)$ and $a \preccurlyeq_i b$ if $\alpha(a)=\alpha(b)=i$.

\end{exm}

We highlight the following proposition, which we leave as an exercise for now, since it will follow as a special case of \autoref{iso of categories}, see~\cite{PR02} for related work. Recall that $\Delta \bfS_+$ is a symmetric monoidal category by \autoref{lemma:pro}.

\begin{prop}
    There is an isomorphism of symmetric monoidal categories $\Env(\Assoc)\cong \Delta \bfS_+.$
\end{prop}

Given an operad $\cO$ in $G$-spaces, we review the definition of the operad $\cO\rtimes G$ constructed in \cite[Definition 2.1]{SW03}. 

\begin{defin}
Let  $\cO$ be an operad in $G$-spaces with structure map $\gamma_{\cO}$. The twisted $G$-operad $\cO\rtimes G$  is defined as $\cO\rtimes G(k)=\cO(k)\times G^k$, with $\Sigma_k$ acting diagonally, and with structure map
\[
\gamma \colon (\cO\rtimes G) (k) \times (\cO\rtimes G) (n_1) \times \cdots \times (\cO\rtimes G) (n_k) \rightarrow (\cO\rtimes G) (n_1+\cdots+n_k)
\]
given by 
\[
\big((\phi, \underline{g}), (\psi_1, \underline{h}^1), \ldots, (\psi_k,\underline{h}^k)\big) \mapsto (\gamma_{\cO}(\phi, g_1\psi_1, \ldots, g_k\psi_k), g_1\underline{h}^1, \ldots, g_k\underline{h}^k)
\]
where $\underline{g}=(g_1, \ldots, g_k)\in G^k$, $\underline{h}^i=(h^i_1, \ldots, h^i_{n_i})\in G^{n_i}$ and $g_i\underline{h}^i=(g_ih^i_1, \ldots, g_ih^i_{n_i})$.
The unit in $\cO\rtimes G$ is formed by the units of $\cO$ and $G$.
\end{defin}

We will primarily consider the following special case. 

\begin{exm}\label{example: twisted operad}
Let $(G,\varphi)$ be a group with parity. Consider the action of $C_2$ on $\Assoc$ given by sending an ordering to its reverse ordering as in \cite[Remark 4.1.1.7]{HA}.\footnote{Note that this is different from the $C_2$-action given on $\Assoc(n)=\Sigma_n$ by sending a permutation to its inverse. The action that sends an ordering to its reverse ordering, when viewed as permutations, is not a group homomorphism, but only an action on the underlying set $\Sigma_n$. For example, clearly the identity does not go to the identity.} Note that the structure map of the operad $\Assoc$ is $C_2$-equivariant with respect to this action. Let $\Assoc^{\varphi}$ denote the operad $\Assoc\rtimes G$ where $G$ acts on $\Assoc$ by restriction along the group homomorphism $\varphi\colon G\to C_2$. We write $\Assoc^{\tau}$ when $\varphi=\id_{C_2} :C_2\to C_2$.
\end{exm}
\begin{rem}
The operad $\Assoc^{\tau}$ was previously studied by Kro~\cite{Kro05}. 
\end{rem}

\begin{exm}\label{example: algebra over twisted operad are twisted rings}
    An algebra over $\Assoc^\varphi$ in $\Set$ is precisely a (discrete) twisted $G$-monoid. We spell out the details in the category $\Set$, though the arguments generalize to an arbitrary symmetric monoidal category tensored over $\Set$. The map $\theta_0$ specifies the unit of $R$. The map $\theta_2\colon \Assoc(2)\times G^2 \times R^2\to R$, when specified to
    \[
    \theta_2((1<2), (e, e), (\--, \--))\colon R^2\to R
    \] 
    gives the multiplication, where $e$ denotes the unit element of $G$. This is an associative and unital operation as in the case of $\Assoc$-algebras. But now, additionally, the map $\theta_1\colon \Assoc(1)\times G\times R\to R$ specifies maps $\theta_1(g,-)\colon R\to R$ for each $g\in G$. We claim that these maps are a monoid homomorphism or monoid anti-homomorphism depending on the parity of $\varphi(g)$.
To see this, recall first that we have a commuting diagram
\[
\begin{tikzcd}
    \Assoc(2)\times G^2 \times \Assoc(1)\times G \times \Assoc(1)\times G \times R^2 \ar{r}{\gamma\times \id_{R^2}} \ar{d}{\cong} & \Assoc(2)\times G^2 \times R^2 \ar{dd}{\theta_2}\\
    \Assoc(2)\times G^2 \times \Assoc(1)\times R\times  G \times \Assoc(1)\times R \times G \ar{d}{\id_{ \Assoc(2)\times G^2}\times \theta_1\times\theta_1} &\\
    \Assoc(2)\times G^2 \times R^2\ar{r} \ar{r}{\theta_2}&R
\end{tikzcd}
    \]
    \noindent which gives us in particular that
\begin{equation}\label{theta1prod}
        \theta_2((1<2), (g,g), (a,b))=\theta_2((1<2),(e,e),(\theta_1(g, a),\theta_1(g, b)))\,.
\end{equation}

We also have a commuting diagram 
\[\begin{tikzcd}
    \Assoc(1)\times G \times  \Assoc(2)\times G^2 \times R^2\ar{rrr}{\id_{ \Assoc(1)\times G}\times \theta_2} \ar{d}{\gamma\times \id_{R^2}} &&&  \Assoc(1)\times G \times R \ar{d}{\theta_1}\\
    \Assoc(2)\times G^2  \times R^2\ar{rrr}{\theta_2} &&& R
\end{tikzcd}
\]
\noindent from which we get that 
\begin{equation}\label{theta2}
    \theta_1(g, \theta_2(1<2, e,e), a,b)=\begin{cases}
        (\theta_2(2<1, g,g), a,b), & \text{if } \varphi(g)=-1 \\
        \theta_2((1<2, g,g), a,b), & \text{if } \varphi(g)=1,
    \end{cases}
\end{equation}
\noindent using that 
\[
\gamma((1,g); (1<2); e,e))=(\gamma_{\Assoc}(1; g\cdot (1<2)); g,g)\,,
\] 
where the action of $g$ is via $\varphi\colon G\to C_2$ composed with the $C_2$-action that reverses order. Thus, from \autoref{theta1prod} and \autoref{theta2} we see that the map $\theta_1(g, -)$ is indeed an automorphism of $R$ when $g$ is even and an anti-automorphism of $R$ when $g$ is odd. Lastly, from the compatibility diagram of $\theta_1$ with $\gamma$ we also have a relation 
$$\theta_1(g, \theta_1(h,a))=\theta_1(\gamma(g;h),a)=\theta_1(gh, a),$$ which shows that the actions $\theta_1(g,-)$ compose appropriately. Thus, we observe that $R$ has the structure of a twisted $G$-monoid. The additional associativity axioms for a twisted $G$-monoid are encoded by incorporating the action of $\Assoc(n)\times G^n$ for $n>2$ as well. 
\end{exm}

Let $\cC$ be a symmetric monoidal category tensored over $\Set$. For any group $G$ with parity $\varphi$, an algebra $R$ over $\Assoc^\varphi$ in $\cC$ is precisely a twisted $G$-monoid in $\cC$. In particular, an algebra over $\Assoc^{\tau}$ is precisely a monoid with anti-involution.
We summarize this in a proposition.

\begin{prop}\label{TwistedGringoperad}
    Let $\cC$ be a 
    symmetric monoidal category tensored over $\Set$. Let  $\varphi\colon G\rightarrow C_2$ be a group homomorphism. Then we have an isomorphism of categories 
    \[ \Alg_{\Assoc^{\varphi}}(\cC)\overset{\simeq}{\longrightarrow} \Mon^{\varphi}(\cC).\]
\end{prop}

We return to the analysis of the category $\Env(\Assoc^{\varphi})$.

\begin{exm}\label{compEnv}
    We give an explicit example of how composition works in $\Env(\Assoc^\varphi)$. Consider the following composable morphisms in $\Fin$.
\[
\begin{tikzpicture}[baseline= (a).base]
\node[scale=0.7] (a) at (1,1){\begin{tikzcd}[row sep=small]
	{1} && {\psi} &&& {\phi} \\
	{2} &&& {1} &&& {1} \\
	{3} &&& {2} &&& {2} \\
	{4} &&& {3} &&& {3} \\
	{5} &&& {4} \\
	{6}
	\arrow[color=blue, no head, from=1-1, to=4-4, thick]
	\arrow[color={rgb,255:red,92;green,214;blue,92}, no head, from=2-1, to=2-4, thick]
	\arrow[color={rgb,255:red,92;green,214;blue,92}, no head, from=4-1, to=2-4, thick]
	\arrow[color={rgb,255:red,92;green,214;blue,92}, no head, from=5-1, to=2-4, thick]
	\arrow[color={rgb,255:red,214;green,92;blue,214}, no head, from=3-1, to=5-4, thick]
	\arrow[color={rgb,255:red,214;green,92;blue,214}, no head, from=6-1, to=5-4, thick]
	\arrow[color={rgb,255:red,92;green,214;blue,214}, no head, from=2-4, to=4-7, thick]
	\arrow[color={rgb,255:red,92;green,214;blue,214}, no head, from=3-4, to=4-7, thick]
	\arrow[color={rgb,255:red,92;green,214;blue,214}, no head, from=4-4, to=4-7, thick]
	\arrow[color={rgb,255:red,214;green,153;blue,92}, no head, from=5-4, to=2-7, thick]
\end{tikzcd}};  
\end{tikzpicture}
\]
Consider the following morphisms in $\Env(\Assoc^\varphi)$, where the upperscripts on the elements of $G$ denote their parities:
$$(\psi;\  {\color{rgb,255:red,92;green,214;blue,92}(5<2<4;\ g_{11}^+, g_{12}^-, g_{13}^+)},\  \emptyset,\ {\color{blue}(1;\ g_{31}^+)},\ {\color{rgb,255:red,214;green,92;blue,214} (6<3;\
 g_{41}^-, g_{42}^+)}), \ \text{and}$$ 
$$(\phi;\ {\color{rgb,255:red,214;green,153;blue,92} (4; \ h_{11}^-)},\ \emptyset,\ {\color{rgb,255:red,92;green,214;blue,214} (3<1<2;\ h_{31}^-, h_{32}^-, h_{33}^+)}).$$
Then we can compute the composite of these morphisms to be 
$$(\phi\circ \psi;\ (3<6;\ h_{11}g_{41}, h_{11}g_{42}),\ \emptyset,\ (1<4<2<5;\ h_{33}g_{31}, h_{31}g_{11}, h_{31}g_{12}, h_{31}g_{13})),$$
where the parities of the elements of $G$ in the previous row are obtained by multiplying the parities of the elements in the products.
\end{exm}

Recall the symmetric monoidal structure on $\bfGS_+$ from \autoref{lemma:pro}.

\begin{thm}\label{iso of categories}
There is an isomorphism of symmetric monoidal categories 
\[
F\colon\Env(\Assoc^\varphi)\xrightarrow{\cong} \Delta \bfGS_+ \,.
\]
\end{thm}
\begin{proof}
\vspace{-0.25cm}
On objects, we define the functor $F$ 
by the obvious bijection that sends $\nF$ to $\nD$. 
On morphisms, $F$ is defined by sending a nontrivial morphism
\[\big( (\alpha, \{\preccurlyeq_i\}_{1\leq i\leq n}) \colon \mF \to \nF , \ \ell_i\colon \alpha^{-1}(i)\to G\big)\] 
where $\alpha$ is an active map of finite pointed sets, $\preccurlyeq_i$ is a total order on $\alpha^{-1}(i)$, and $\ell_i$ is a map  sets, to the morphism $(g,\phi)$
in $\Delta \bfGS_+$ determined as follows. The map $\phi$ in $\Delta$ is the order preserving map $\mF \to \nF$ determined by the sizes of the inverse images of $\alpha$. 
The element $g=(\gamma, g_1,\dots ,g_m)$ in $\Sigma_m\rtimes G^m$ is defined by letting $\gamma$ be the block permutation of $\{1, \dots, m\}$, which in each inverse image block $\alpha^{-1}(i)$ is given by the permutation determined by the given linear ordering on $\alpha^{-1}(i)$. For the labels in $G$, we set set $g_j=\ell_i(j)$ for each $j\in \alpha^{-1}(i)$. We illustrate this definition in \autoref{Fcomputations} by computing $F$ on the morphisms from \autoref{compEnv}. 
 
This correspondence between morphisms is bijective, so it remains to check that $F$ is indeed a monoidal functor. It follows from a careful check of definitions that $F$ does respect composition, though it is tedious to write out explicitly. We choose to illustrate this in \autoref{Fcomputations}. It is immediate to verify that $F$ is strong symmetric monoidal.
\end{proof}

\begin{exm}\label{Fcomputations}
We show explicitly that the functor $F$ respects the composition of the morphisms from \autoref{compEnv}.
Consider the morphisms in $\Env(\Assoc^\varphi)$ from \autoref{compEnv}. First we factor $\psi=\gamma_1\circ \sigma_1$ and $\phi=\gamma_2\circ \sigma_2$ into an order preserving map preceded by a permutation which is determined by the linear orderings given on the preimages, as described in the definition of $F$:
\[ \begin{tikzpicture}[baseline= (a).base]
\node[scale=0.6] (a) at (1,1){\begin{tikzcd}[row sep=small]
	&& {\sigma_1} && {\gamma_1} \\
	{1} &&& 1 &&&&&&&& {\sigma_2} && {\gamma_2} \\
	{2} &&& {2} &&& {1} &&& 1 &&& {1} &&& {1} \\
	{3} &&& {3} &&& {2} &&& {2} &&& 2 &&& 2 \\
	{4} &&& {4} &&& {3} &&& {3} &&& {3} &&& 3 \\
	{5} &&& {5} &&& {4} &&& {4} &&& {4} \\
	{6} &&& {6}
	\arrow[color={rgb,255:red,92;green,214;blue,92}, no head, from=2-4, to=3-7, thick]
	\arrow[color={rgb,255:red,92;green,214;blue,92}, no head, from=3-4, to=3-7, thick]
	\arrow[color={rgb,255:red,92;green,214;blue,92}, no head, from=4-4, to=3-7, thick]
	\arrow[color={rgb,255:red,92;green,214;blue,92}, no head, from=6-1, to=2-4, thick]
	\arrow[color={rgb,255:red,92;green,214;blue,92}, no head, from=5-1, to=4-4, thick]
	\arrow[color={rgb,255:red,92;green,214;blue,92}, no head, from=3-1, to=3-4, thick]
	\arrow[color={rgb,255:red,57;green,67;blue,249}, no head, from=2-1, to=5-4, thick]
	\arrow[color={rgb,255:red,57;green,67;blue,249}, no head, from=5-4, to=5-7, thick]
	\arrow[color={rgb,255:red,214;green,92;blue,214}, no head, from=4-1, to=7-4, thick]
	\arrow[color={rgb,255:red,214;green,92;blue,214}, no head, from=7-1, to=6-4, thick]
	\arrow[color={rgb,255:red,214;green,92;blue,214}, no head, from=6-4, to=6-7, thick]
	\arrow[color={rgb,255:red,214;green,92;blue,214}, no head, from=7-4, to=6-7, thick]
	\arrow[color={rgb,255:red,214;green,153;blue,92}, no head, from=3-13, to=3-16, thick]
	\arrow[color={rgb,255:red,214;green,153;blue,92}, no head, from=6-10, to=3-13, thick]
	\arrow[color={rgb,255:red,92;green,214;blue,214}, no head, from=4-13, to=5-16, thick]
	\arrow[color={rgb,255:red,92;green,214;blue,214}, no head, from=5-13, to=5-16, thick]
	\arrow[color={rgb,255:red,92;green,214;blue,214}, no head, from=6-13, to=5-16, thick]
	\arrow[color={rgb,255:red,92;green,214;blue,214}, no head, from=5-10, to=4-13, thick]
	\arrow[color={rgb,255:red,92;green,214;blue,214}, no head, from=3-10, to=5-13, thick]
	\arrow[color={rgb,255:red,92;green,214;blue,214}, no head, from=4-10, to=6-13, thick]
\end{tikzcd}};  
\end{tikzpicture}\]
Thus we get the following images under the functor $F$ in $\Delta \bfGS_+$, when expressed as their unique factorizations as an element of $\Sigma_6 \rtimes G^6$, resp. $\Sigma_4 \rtimes G^4$, followed by a map in $\Delta$:
 $$F(\psi;\  {\color{rgb,255:red,92;green,214;blue,92}(5<2<4;\ g_{11}^+, g_{12}^-, g_{13}^+)},\  \emptyset,\ {\color{blue}(1;\ g_{31}^+)},\ {\color{rgb,255:red,214;green,92;blue,214} (6<3;\
 g_{41}^-, g_{42}^+)})$$ 
 $$=\gamma_1\circ (\sigma_1;\ {\color{rgb,255:red,92;green,214;blue,92} g_{13}^+, g_{11}^+,  g_{12}^-}, {\color{blue} g_{31}^+}, {\color{rgb,255:red,214;green,92;blue,214} g_{42}^+, g_{41}^-})$$  $$\text{and}$$ 
$$F(\phi;\ {\color{rgb,255:red,214;green,153;blue,92} (4; \ h_{11}^-)},\ \emptyset,\ {\color{rgb,255:red,92;green,214;blue,214} (3<1<2;\ h_{31}^-, h_{32}^-, h_{33}^+)})=\gamma_2\circ (\sigma_2;\ {\color{rgb,255:red,214;green,153;blue,92}  h_{11}^-},{\color{rgb,255:red,92;green,214;blue,214} h_{33}^+, h_{31}^-, h_{32}^-}).$$
Similarly, 
$$F(\phi\circ \psi;\ {\color{rgb,255:red,214;green,92;blue,92}(3<6;\ h_{11}g_{41}, h_{11}g_{42})},\ \emptyset,\ {\color{rgb,255:red,92;green,92;blue,214} (1<4<2<5;\ h_{33}g_{31}, h_{31}g_{11}, h_{31}g_{12}, h_{31}g_{13})})$$
$$= \gamma_3\circ (\sigma_3;\ {\color{rgb,255:red,214;green,92;blue,92} h_{11}g_{41}, h_{11}g_{42}}, {\color{rgb,255:red,92;green,92;blue,214} h_{33}g_{31},  h_{31}g_{12}, h_{31}g_{11}, h_{31}g_{13}}),$$
where the factorization of $\phi\circ \psi$ into a permutation $\sigma_3\in \Sigma_6$ determined by the linear orderings on preimages, and an order preserving map $\gamma_3$ is depicted below. 
\[ \begin{tikzpicture}[baseline= (a).base]
\node[scale=0.6] (a) at (1,1){\begin{tikzcd}[row sep=small]
&& {\sigma_3} && {\gamma_3} \\
	{1} &&& {1} \\
	{2} &&& {2} &&& {1} \\
	{3} &&& {3 } &&& {2} \\
	{4} &&& {4} &&& 3 \\
	{5} &&& {5} \\
	{6} &&& {6}
	\arrow[color={rgb,255:red,214;green,92;blue,92}, no head, from=2-4, to=3-7, thick]
	\arrow[color={rgb,255:red,214;green,92;blue,92}, no head, from=3-4, to=3-7, thick]
	\arrow[color={rgb,255:red,92;green,92;blue,214}, no head, from=4-4, to=5-7, thick]
	\arrow[color={rgb,255:red,92;green,92;blue,214}, no head, from=5-4, to=5-7, thick]
	\arrow[color={rgb,255:red,92;green,92;blue,214}, no head, from=6-4, to=5-7, thick]
	\arrow[color={rgb,255:red,92;green,92;blue,214}, no head, from=7-4, to=5-7, thick]
	\arrow[color={rgb,255:red,214;green,92;blue,92}, no head, from=4-1, to=2-4, thick]
	\arrow[color={rgb,255:red,92;green,92;blue,214}, no head, from=2-1, to=4-4, thick]
	\arrow[color={rgb,255:red,92;green,92;blue,214}, no head, from=5-1, to=5-4, thick]
	\arrow[color={rgb,255:red,92;green,92;blue,214}, no head, from=3-1, to=6-4, thick]
	\arrow[color={rgb,255:red,92;green,92;blue,214}, no head, from=6-1, to=7-4, thick]
	\arrow[color={rgb,255:red,214;green,92;blue,92}, no head, from=7-1, to=3-4, thick]
\end{tikzcd}};  
\end{tikzpicture}
\]
 We need to show that the following equality of morphisms holds in $\Delta \bfGS_+$:
$$\gamma_2 \circ(\sigma_2;\   h_{11}^-, h_{33}^+, h_{31}^-, h_{32}^-) \circ \gamma_1\circ (\sigma_1;\   g_{13}^+,  g_{11}^+,  g_{12}^-, g_{31}^+,  g_{42}^+, g_{41}^-) $$ $$=\gamma_3\circ (\sigma_3;\ h_{11}g_{41}, h_{11}g_{42},  h_{33}g_{31},  h_{31}g_{12}, h_{31}g_{11}, h_{31}g_{33}).$$

Expressing the map $(\sigma_2;\   h_{11}^-, h_{33}^+, h_{31}^-, h_{32}^-) \circ \gamma_1$  in terms of its unique factorization as an element of $\Sigma_6 \rtimes G^6$ composed with an order preserving map in $\Delta$ allows us to write the inner composite in terms of its unique factorization
$$(\sigma_2;\   h_{11}^-, h_{33}^+, h_{31}^-, h_{32}^-)^\ast\gamma_1 \circ \gamma_1^\ast(\sigma_2;\   h_{11}^-, h_{33}^+, h_{31}^-, h_{32}^-). $$ 
It is immediate to see that 
$$\gamma_2\circ (\sigma_2;\   h_{11}^-, h_{33}^+, h_{31}^-, h_{32}^-)^\ast\gamma_1= \gamma_3,$$
so we are left with reconciling the elements in $\Sigma_6 \rtimes G^6$.

Let $\sigma_4$ be the permutation defined in \autoref{fig: example with sigma 4},
\begin{figure}
    \centering
\begin{tikzpicture}[baseline= (a).base]
\node[scale=0.60] (a) at (1,1){
\begin{tikzcd}[row sep=small]
	&&&&& \textcolor{rgb,255:red,214;green,153;blue,92}{1} \\
	&&&& \textcolor{rgb,255:red,214;green,153;blue,92}{2} \\
	&&& \textcolor{rgb,255:red,92;green,92;blue,214}{3} &&&&&&&& \textcolor{rgb,255:red,214;green,153;blue,92}{1} \\
	&& \textcolor{rgb,255:red,92;green,214;blue,92}{4} &&&&&&&& \textcolor{rgb,255:red,92;green,92;blue,214}{2} \\
	& \textcolor{rgb,255:red,92;green,214;blue,92}{5} &&&&&&&& \textcolor{rgb,255:red,92;green,214;blue,92}{3} \\
	\textcolor{rgb,255:red,92;green,214;blue,92}{6} &&&&&&&& \textcolor{rgb,255:red,214;green,92;blue,214}{4} \\
	&&&&& \textcolor{rgb,255:red,92;green,214;blue,92}{1} &&&&&& {\sigma_2} \\
	&&&& \textcolor{rgb,255:red,92;green,214;blue,92}{2} \\
	{\sigma_4} &&& \textcolor{rgb,255:red,92;green,214;blue,92}{3} &&&&&&&& \textcolor{rgb,255:red,92;green,214;blue,92}{1} \\
	&& \textcolor{rgb,255:red,92;green,92;blue,214}{4} &&&&&&&& \textcolor{rgb,255:red,214;green,92;blue,214}{2} \\
	& \textcolor{rgb,255:red,214;green,153;blue,92}{5} &&&&&&&& \textcolor{rgb,255:red,92;green,92;blue,214}{3} \\
	\textcolor{rgb,255:red,214;green,153;blue,92}{6} &&&&&&&& \textcolor{rgb,255:red,214;green,153;blue,92}{4} \\
	&&&& {\gamma_1}
	\arrow[draw={rgb,255:red,214;green,153;blue,92}, no head, from=1-6, to=3-12, thick]
	\arrow[draw={rgb,255:red,214;green,153;blue,92}, no head, from=1-6, to=12-1, thick]
	\arrow[draw={rgb,255:red,214;green,153;blue,92}, no head, from=2-5, to=3-12, thick]
	\arrow[draw={rgb,255:red,214;green,153;blue,92}, no head, from=2-5, to=11-2, thick]
	\arrow[color={rgb,255:red,92;green,92;blue,214}, no head, from=3-4, to=10-3, thick]
	\arrow[draw={rgb,255:red,92;green,214;blue,92}, no head, from=4-3, to=5-10, thick]
	\arrow[draw={rgb,255:red,92;green,214;blue,92}, no head, from=4-3, to=9-4, thick]
	\arrow[draw={rgb,255:red,92;green,92;blue,214}, no head, from=4-11, to=3-4, thick]
	\arrow[draw={rgb,255:red,92;green,214;blue,92}, no head, from=5-2, to=5-10, thick]
	\arrow[draw={rgb,255:red,92;green,214;blue,92}, no head, from=5-2, to=8-5, thick]
	\arrow[draw={rgb,255:red,92;green,214;blue,92}, no head, from=6-1, to=5-10, thick]
	\arrow[draw={rgb,255:red,92;green,214;blue,92}, no head, from=6-1, to=7-6, thick]
	\arrow[draw={rgb,255:red,92;green,214;blue,92}, no head, from=7-6, to=9-12, thick]
	\arrow[draw={rgb,255:red,92;green,214;blue,92}, no head, from=8-5, to=9-12, thick]
	\arrow[draw={rgb,255:red,92;green,214;blue,92}, no head, from=9-4, to=9-12, thick]
	\arrow[draw={rgb,255:red,92;green,214;blue,92}, no head, from=9-12, to=5-10, thick]
	\arrow[draw={rgb,255:red,92;green,92;blue,214}, no head, from=10-3, to=11-10, thick]
	\arrow[color={rgb,255:red,214;green,92;blue,214}, no head, from=10-11, to=6-9, thick]
	\arrow[draw={rgb,255:red,214;green,153;blue,92}, no head, from=11-2, to=12-9, thick]
	\arrow[draw={rgb,255:red,92;green,92;blue,214}, no head, from=11-10, to=4-11, thick]
	\arrow[draw={rgb,255:red,214;green,153;blue,92}, no head, from=12-1, to=12-9, thick]
	\arrow[draw={rgb,255:red,214;green,153;blue,92}, no head, from=12-9, to=3-12, thick]
\end{tikzcd}
};  
\end{tikzpicture}
    \caption{The permutation $\sigma_4$ induced by $\gamma_1$ and $\sigma_2$ in $\Delta \varphi\wr \bfS$}
    \label{fig: example with sigma 4}
\end{figure}
so that $$\gamma_1^\ast(\sigma_2;\   h_{11}, h_{33}, h_{31}, h_{32})= (\sigma_4, h_{11}, h_{11}, h_{31}, h_{32}, h_{32}, h_{32}). $$
Now we finally check that indeed in $\Sigma_6 \rtimes G^6$ we have
$$(\sigma_4; h_{11}, h_{11}, h_{31}, h_{32}, h_{32}, h_{32})\circ (\sigma_1; g_{13}, g_{11}, g_{12}, g_{31}, g_{42}, g_{41})$$
$$=(\sigma_3;\ h_{11}g_{41}, h_{11}g_{42},  h_{33}g_{31},  h_{31}g_{12}, h_{31}g_{11}, h_{31}g_{33}),$$ since $\sigma_4\sigma_1=\sigma_3$, and the multiplication rule permutes the $g$'s first by $\sigma_4$ before multiplying them by the $h$'s. 
\end{exm}

\begin{rem2}
As a consequence of \autoref{lemma:pro}, we observe that given a group with parity $(G,\varphi)$ the associated category $\Delta \bfGS_+$ is in fact a PROP \cite{Mac65}. In light of this, one can recast \autoref{iso of categories} and \autoref{TwistedGringoperad} as the statement that $\Delta \bfGS_+$ is the PROP for rings with twisted $G$-action. In this sense, \autoref{iso of categories} and \autoref{TwistedGringoperad} generalize results of \cite[\S 3]{Pir02} and \cite[\S 3]{Gra20} in the cases where $\Delta  \bfGS_+=\Delta \bfS_+$ and $\Delta \bfGS_+=\Delta \bfH_+$  respectively. 
\end{rem2}

\subsection{Twisted \texorpdfstring{$G$}{TEXT}-objects in \texorpdfstring{$\infty$}{TEXT}-categories}\label{sec:twisted-G-infinity}

We can extend the definitions of \autoref{sec:twisted-G-objects} to $\infty$-categories. Notably, an \textit{$\infty$-category $\cC$ with parity} is a functor $\cC\to BC_2$. 
Given $\infty$-categories with parity $\cC\to BC_2$ and $\cD\to BC_2$, a functor $F\colon \cC\rightarrow \cD$ compatible with the maps to $BC_2$
is called a functor of categories with parity.
Throughout this section, we write $(G, \varphi)$ for a group with parity. 

\begin{defin}
Let $\cC$ be a (large) $\infty$-category with a $C_2$-action. By unstraightening, the induced functor $BC_2\rightarrow \CAT_\infty$ classifies a cocartesian fibration $C_2\rtimes \cC\rightarrow BC_2$. 
For  a group with parity $(G, \varphi)$, a \emph{twisted $G$-object in $\cC$} is a functor of $\infty$-categories with parity $BG\rightarrow C_2\rtimes \cC$.
\end{defin}

\begin{rem2}
Note that when the parity $\varphi\colon G\rightarrow C_2$ is non-trivial, a functor of categories with parity $BG\rightarrow C_2\rtimes \cC$ sends cocartesian lifts to cocartesian lifts.
\end{rem2}

Given a symmetric monoidal $\infty$-category $\cC$, the $\infty$-category $\Alg(\cC)$ of $\mathbb{E}_1$-algebras in $\cC$ is endowed with a $C_2$-action given by $M\mapsto M^\op$, where $M^\op$ is the  algebra obtained by precomposing the multiplication map with the swap map \cite[2.13.4]{hinich}. This leads to the following definition.

\begin{defin}[The $\infty$-category of rings with twisted $G$-action]\label{def: category of twisted G-rings}
Let $\cC$ be a cocomplete symmetric monoidal $\infty$-category.
An \textit{algebra with twisted $G$-action in $\cC$} is a twisted $G$-object in $\Alg(\cC)$. 
We denote by $\Alg^\varphi(\cC)$ the  $\infty$-category $\Fun_{/BC_2}(BG, C_2\rtimes \Alg(\cC))$ of twisted $G$-algebras.
\end{defin}

\begin{rem2}
In this section, we follow Lurie's convention in~\cite{HA}
and refer to an algebra over an operad in a symmetric monoidal $\infty$-category $\cC$ as an algebra and reserve the term monoid for an algebra over the operad $\Assoc$ in a symmetric monoidal $\infty$-category $\cC$ equipped with a Cartesian monoidal structure. 
\end{rem2}

\begin{rem3}\label{algebras description}
As in \autoref{example: algebra over twisted operad are twisted rings} and \autoref{TwistedGringoperad}, the $\infty$-category $\Alg^\varphi(\cC)$ is equivalent to the $\infty$-category $\Alg_{\Assoc^\varphi}(\cC)$ of algebras in $\cC$ over the operad $\Assoc^\varphi$ defined in \autoref{example: twisted operad}.
By~\cite[Proposition~2.2.4.9]{HA} together with \autoref{iso of categories} and \autoref{TwistedGringoperad}, we obtain the following different perspectives to identify this $\infty$-category:
\begin{align*}
    \Alg^\varphi(\cC) & = \Fun_{/BC_2}(BG, C_2\rtimes \Alg(\cC))\\
    & \simeq \Alg_{\Assoc^\varphi}(\cC)\\
    & \simeq \Fun^\otimes (\Env(\Assoc^\varphi), \cC)\\
    & \simeq \Fun^\otimes({\Delta \varphi\wr \Sigma}_+, \cC) \,.
\end{align*}
Here $\Fun_{/BC_2}(BG,C_2\ltimes \Alg(\cC))$ is the $\infty$-category of functors $BG\to C_2\ltimes \Alg(\cC)$ over $BC_2$, we write $\Alg_{\Assoc^\varphi}(\cC)$ for the $\infty$-category of $\infty$-operad maps from $\Assoc^\varphi\to \cC$, and we write $\Fun^{\otimes}$ for the $\infty$-category of symmetric monoidal functors. 
\end{rem3}
\begin{defin}
When $\cC$ is the symmetric monoidal $\infty$-category of spectra $\Sp$, then we call $\Alg^\varphi(\Sp)$ the \textit{$\infty$-category of ring spectra with twisted $G$-action} and we call the objects in $\Alg^\varphi(\Sp)$ \emph{ring spectra with twisted $G$-action.}
When $\cC$ is the symmetric monoidal $\infty$-category of spaces, then we call $\Alg^\varphi(\cS)$ the \textit{$\infty$-category of $\mathbb{E}_1$-spaces with twisted $G$-action} and we call the objects in $\Alg^\varphi(\cS)$ \textit{$\mathbb{E}_1$-spaces  with twisted $G$-action}.
\end{defin}

\begin{rem}
Let $\cM$ be a combinatorial symmetric monoidal model category, with class of weak equivalences denoted $\cW$ and let $\cC=\cM[\cW^{-1}]$ be its underlying $\infty$-category in the sense of \cite[Definition~1.3.4.15,~Remark~1.3.4.16]{HA}.
Then there is a natural functor
\[
     \Alg^\varphi(\cM) \left [\cW^{-1}_{\Alg}\right ]\longrightarrow \Alg^{\varphi} \left( \cC\right)
\]
that is injective on objects cf. \cite[Theorem~1.3.4.20]{HA}, where $\cW_{\Alg}$ is the class of weak equivalences $\cM$ that are homomorphisms of monoids with twisted $G$-action. Essentially, any monoid with twisted $G$-action in a symmetric monoidal model category can be viewed as an algebra with twisted $G$-action in the corresponding underlying symmetric monoidal $\infty$-category.
This allows us to produce examples of rings with twisted $G$-action using explicit model categorical constructions.
\end{rem}

Recall that the (large) $\infty$-category $\Cat_\infty$ of small  $\infty$-categories is endowed with a unique non-trivial equivalence $\Cat_\infty\rightarrow\Cat_\infty$ given by $\cC\mapsto \cC^\op$, which defines a $C_2$-action on $\Cat_\infty$ (see \cite[6.3]{toen}).
The induced functor on $BC_2\rightarrow \CAT_\infty$ classifies a cocartesian fibration that we denote $C_2\rtimes \Cat_\infty\rightarrow BC_2$.

\begin{defin}
A \textit{(small) $\infty$-category with twisted $G$-action} is a twisted $G$-object in $\Cat_\infty$. We denote by $\Cat_\infty^\varphi$ the induced (large) $\infty$-category $\Fun_{/BC_2}(BG, C_2\rtimes \Cat_\infty)$ of small twisted $G$-$\infty$-categories. 
\end{defin}

Unwinding the definition, an $\infty$-category with twisted $G$-action consists of an $\infty$-category $\cC$ equipped with a functor $g\colon \cC\to \cC$ for each even $g\in G$ and a functor 
$g\colon \cC^{\op}\to \cC$ for each odd $g\in G$ such that $g\circ g^{\prime}=g g^{\prime}$ for each pair $g,g^{\prime}\in G$. We refer to $\cC$ as the underlying $\infty$-category.

\begin{rem}
Given $\infty$-categories with twisted $G$-action $\cC$ and $\cD$, a twisted $G$-equivariant functor $F\colon \cC\to \cD$ is a morphism in $\Cat_\infty^\varphi=\Fun_{/BC_2}(BG, C_2\rtimes \Cat_\infty)$.
Informally, given a functor $F\colon \cC\rightarrow \cD$ in $C_2\rtimes \Cat_\infty$, the functor $gF\colon \cC\rightarrow \cD$ is obtained by the following composition in $C_2\rtimes \Cat_\infty$:
\[
\begin{tikzcd}
    \cC \ar{r}{g^{-1}} & \cC \ar{r}{F}  & \cD \ar{r}{g}  & \cD.
\end{tikzcd}
\]
Explicitly, as maps in $\Cat_\infty$, if $F\colon \cC\to \cD$ is an even morphism, this 
is given by 
\[
\begin{tikzcd}
    \cC \ar{r}{g^{-1}} & \cC \ar{r}{F}  & \cD \ar{r}{g}  & \cD, 
\end{tikzcd}
\text{ or } \begin{tikzcd}
    \cC \ar{r}{(g^{-1})^{\op}} & \cC^{\op} \ar{r}{F^{\op}}  & \cD^{\op} \ar{r}{g}  & \cD,
\end{tikzcd}\]
depending on whether $g$ is even or odd, and if $F\colon \cC^\op\to \cD$ is an odd morphism, this 
is given by 
\[
\begin{tikzcd}
    \cC^{\op} \ar{r}{(g^{-1})^{\op}} & \cC^{\op} \ar{r}{F}  & \cD \ar{r}{g}  & \cD \,,
\end{tikzcd}
\text{ or }
\begin{tikzcd}
    \cC^{\op} \ar{r}{g^{-1}} & \cC \ar{r}{F^{\op}}  & \cD^{\op} \ar{r}{g}  & \cD \,,
\end{tikzcd}
\]
depending on whether $g$ is even or odd.
\end{rem}

\begin{exm}\label{ex: twisted G-categories with one object}
Recall that an $\mathbb{E}_1$-space $M$ defines an $\infty$-category $\bB M$ with one object and space of morphisms given by $M$. 
Explicitly, one can construct $\bB M$ as the homotopy coherent nerve of the simplicial category with one object and with hom-simplicial set $M$ seen as a Kan complex.
The assignment defines a functor $\bB\colon \Alg(\cS)\rightarrow \Cat_\infty$.
As $\bB(M^\op)\simeq (\bB M)^\op$, the functor is $C_2$-invariant and defines thus a functor of categories with parity $\bB \colon C_2\rtimes \Alg(\cS)\rightarrow C_2\rtimes \Cat_\infty$ by unstraightening. Let $(G,\varphi)$ be a group with parity. If $M$ is an $\mathbb{E}_1$-space with twisted $G$-action, then $\bB M$ is an $\infty$-category with twisted $G$-action via the induced assignment:
\[
\bB^{\varphi}\colon \Alg^\varphi(\cS)=\Fun_{/BC_2}(BG, C_2\rtimes \Alg(\cS))\rightarrow \Fun_{/BC_2}(BG, C_2\rtimes \Cat_\infty)=\Cat_\infty^\varphi.
\]
The upshot of this discussion is that one can view an $\infty$-category with twisted $G$-action as an ``$\mathbb{E}_1$-space with twisted $G$-action with many objects". 
\end{exm}

\begin{rem}
    The $C_2$-action on $\Cat_\infty$ restricted to $\infty$-groupoids is trivial.
    Therefore, an $\infty$-groupoid with twisted $G$-action is simply an $\infty$-groupoid with $G$-action. This defines a fully faithful embedding $\cS^{BG}\hookrightarrow \Cat_\infty^\varphi$.
\end{rem}

\begin{prop}
    The $\infty$-category $\Cat_\infty^\varphi$ is presentable, and the usual maximal $\infty$-groupoid and $\infty$-groupoid completion functors lift to $\infty$-categories with twisted $G$-action:
    \[
    \begin{tikzcd}[column sep=huge]
        \Cat_\infty^\varphi \ar[bend left=50, "(-)^\mathrm{gp}" description]{r} \ar[shift left=4, phantom, "\perp" description]{r}  &  \ar[shift left=4, phantom, "\perp" description]{l} \cS^{BG}.\ar[hook']{l}\ar[bend left=50, "(-)^\simeq" description, leftarrow]{l}
    \end{tikzcd}
    \]
\end{prop}

\begin{proof}
    The $\infty$-category $\Cat_\infty^\varphi$ is presentable by \cite[5.5.3.3, 5.5.3.6, 5.5.3.10]{HTT}.
    Because $\cS^{BG}\hookrightarrow \Cat_\infty^\varphi$ preserves all colimits and limits, it has a left and right adjoint by the adjoint functor theorem \cite[5.5.2.9]{HTT}.
    Moreover, recall that the maximal $\infty$-groupoid and free $\infty$-groupoid functors are invariant under the $C_2$-action on $\Cat_\infty$, i.e. $(\cC^\op)^\simeq\simeq \cC^\simeq$ and $(\cC^\op)^\mathrm{gp}\simeq  \cC^\mathrm{gp}$, for any small $\infty$-category $\cC$.
    This shows that we obtain by unstraightening the adjunctions: 
    \[
    \begin{tikzcd}[column sep=huge]
        C_2\rtimes  \Cat_\infty \ar[bend left=35, "(-)^\mathrm{gp}" description]{r} \ar[shift left=4, phantom, "\perp" description]{r}  &  \ar[shift left=4, phantom, "\perp" description]{l} C_2\rtimes \cS\simeq \cS.\ar[hook']{l}\ar[bend left=35, "(-)^\simeq" description, leftarrow]{l}
    \end{tikzcd}
    \]
    The forgetful functors $U\colon\Cat_\infty^\varphi\rightarrow C_2\rtimes  \Cat_\infty$ and $U\colon\cS^{BG}\rightarrow \cS$ induced by choosing the identity $*\rightarrow BG$ fit then in a diagram:
    \[
    \begin{tikzcd}
       \cS^{BG} \ar[hook]{r} \ar{d}[swap]{U} & \Cat_\infty^\varphi \ar{d}{U}\\
\cS \ar[hook]{r} & C_2\rtimes  \Cat_\infty.
    \end{tikzcd}
    \] 
The commutativity between the forgetful functors $U$ and the inclusions automatically induces the desired lifting of $(-)^\simeq$ by unicity of adjoint functors.
For the $\infty$-groupoid completion, one easily checks that if $\cC$ has a twisted $G$-action, then $\cC^\gp$ has a $G$-action with the correct desired universal property.
\end{proof}

\section{Topological homology of rings with twisted \texorpdfstring{$G$}{TEXT}-action}\label{sec:thgdef}
We are now prepared to present the main constructions of this paper. Recall that these constructions can be considered as combinatorial models for the norm for compact Lie groups in analogy with \cite{ABGHLM18} and~\cite{AKGH21} (at least after Borel completion). We also advertise that these constructions can be considered in terms of homotopical functor homology, see \autoref{sec:functor-homology}. We first construct the $\Delta \bfG$-bar construction in \autoref{sec:bar}. We define analogues of topological Hochschild homology $\mathrm{T}\mathbf{G}$, topological negative cyclic homology $\mathrm{T}\mathbf{G}^{-}$, and topological periodic cyclic homology $\mathrm{T}\mathbf{G}^{\textup{per}}$ for an arbitrary self-dual crossed simplicial group in \autoref{sec:homology-selfdual}. For an arbitrary crossed simplicial group, we also define a homotopical analogue of the homology of crossed simplicial groups in \autoref{sec:homology}. 

\subsection{The $\Delta \bfG$-bar construction}\label{sec:bar}

We first present the covariant bar construction, which will be used to define topological $\Delta \bfG$-homology of twisted $G$-rings in \autoref{sec:homology-selfdual} and \autoref{sec:homology}. This relies on our results from \autoref{sec:css} and \autoref{sec:twisted-G-actions}. 

\begin{defin}[The covariant $\Delta \bfG$-bar construction]\label{bar-construction}
Let $\Delta \bfG$ be a crossed simplicial group.
Let $\cC$ be a symmetric monoidal $\infty$-category.
Let $(G_0, \lambda_0)$ be the group with parity associated to $\Delta \bfG$. 
Let $R$ be an algebra with twisted $G_0$-action in $\cC$. 
The covariant $\Delta \bfG$-bar construction of $R$ is the functor $B_{\bfG}^{\bullet}(R)\colon \Delta \bfG\rightarrow \cC$ defined as the following composite:
\[
\begin{tikzcd}
    \Delta \bfG \ar{r}{\widetilde{\lambda}} & [-1em] \Delta \lambda_0\wr \bfS \ar{r} & [-1em]  \Env(\Assoc^{\lambda_0}) \ar{r}{\Env(R)} & [0.5em] \Env(\cC) \ar{r}{\otimes} & \cC.
\end{tikzcd}
\]
Here $\widetilde{\lambda}$ is defined in \autoref{canonical map}. 
The unlabeled functor is induced by the canonical inclusion of $\Delta \lambda_0 \wr \bfS $ in ${\Delta \lambda_0 \wr \bfS}_+$ 
followed by the identification  ${\Delta \lambda_0\wr \bfS }_+\cong \Env(\Assoc^{\lambda_0})$ of \autoref{iso of categories}. 
This construction is functorial in $R$ and thus determines a functor $B^\bullet_{\bfG}\colon \Alg^{\lambda_0}(\cC)\to \cC^{\Delta \bfG}$. 
\end{defin}

\autoref{bar-construction} makes precise the following diagram in the $\infty$-category $\cC$:
\[
\begin{tikzcd}
 \cdots \ar[shift left=4.5, leftarrow]{r} \ar[shift left=1.5, leftarrow]{r} \ar[shift right=1.5, leftarrow]{r} \ar[shift right= 4.5, leftarrow]{r} & \ar[shift left=3, leftarrow]{l} \ar[leftarrow]{l} \ar[shift right=3, leftarrow]{l} R^{\otimes 3}\ar[out=120, in=60, loop, looseness=6]{}{{G}_2}\ar[shift left=3, leftarrow]{r} \ar[leftarrow]{r} \ar[shift right=3, leftarrow]{r} & \ar[shift left=1.5, leftarrow]{l} \ar[shift right=1.5, leftarrow]{l} R^{\otimes 2}\ar[out=120, in=60, loop, looseness=6]{}{{G}_1}\ar[shift left=1.5, leftarrow]{r}\ar[shift right=1.5, leftarrow]{r} & R. \ar[leftarrow]{l}\ar[out=120, in=60, loop, looseness=6]{}{{G}_0}
\end{tikzcd}
\]

\begin{rem}
Even though $ B_{\bfG}^{\bullet}(R)$ is a cosimplicial object, we still refer to it as a bar construction to be consistent with the philosophy of Koszul duality: the bar construction is a receptacle for algebras whereas the co-bar construction is a receptacle for co-algebras.
\end{rem}

\begin{exm}\label{ex: covariant Delta G construction}
Let $k$ be field and let $\cC$ be the (nerve of the) category of $k$-vector spaces. 
Then explicitly the covariant $\Delta \bfG$-bar construction of a $k$-algebra $R$ with twisted $G_0$-action is a cosimplicial object $[n]\mapsto R^{\otimes n+1}$, with cofaces
\begin{align*}
     R^{\otimes n+1} & \stackrel{\delta^i}\longrightarrow R^{\otimes n+2}\\
    r_0\otimes \dots \otimes r_n & \longmapsto r_0 \otimes \dots \otimes r_{i-1} \otimes 1_R \otimes r_i \otimes \dots \otimes  r_n
\end{align*}
and codegeneracies 
\begin{align*}
 R^{\otimes n+1} & \stackrel{\sigma^i}\longrightarrow R^{\otimes n}\\
 r_0\otimes \dots \otimes  r_n & \longmapsto r_0\otimes \dots \otimes r_{i-1}\otimes r_ir_{i+1} \otimes r_{i+2} \otimes \dots \otimes r_n
\end{align*}
together with the following additional maps $R^{\otimes n+1}\rightarrow R^{\otimes n+1}$, for each element $g\in G_n$, compatible with the cofaces and codegeneracies:
\begin{align*}
    R^{\otimes n+1} & \stackrel{g}\longrightarrow R^{\otimes n+1}\\ 
    r_0\otimes \dots \otimes r_n & \longmapsto (\phi^*_{\theta_g^{-1}(0)}g \cdot r_{\theta_g^{-1}(0)})\otimes \dots \otimes (\phi^*_{\theta^{-1}_g(n)}g \cdot r_{\theta^{-1}_g(n)}),
\end{align*}
where the permutation $\theta_g\in \Sigma_{n+1}$ is defined as in \autoref{crossed simplicial groups with faces and degeneracies}.
\end{exm}

\subsection{Topological $\Delta \bfG$-homology: self-dual case}\label{sec:homology-selfdual}
We first define the topological homology of rings with twisted $G_0$ action for crossed simplicial group $\Delta \bfG $ equipped with a duality functor $\sdual: \Delta \bfG ^{\op} \overset{\simeq }{\to} \Delta \bfG $. 

\begin{defn}[{\cite[1.3]{Dun89}}]
Let $\Delta \bfG $ be a crossed simplicial group. Suppose there exists 
 an equivalence $\sdual\colon \Delta \bfG^{\op}\overset{\simeq}{\longrightarrow} \Delta \bfG$ of categories satisfying $\sdual\circ \sdual^{\op}\cong \id_{\Delta \bfG}$. Then we say $(\Delta \bfG,\sdual)$ is a \emph{self-dual crossed simplicial group}. 
\end{defn}

\begin{defin}[The contravariant $\Delta \bfG$-bar construction]
Given a self-dual crossed 
simplicial group $(\Delta \bfG,\sdual)$, we define the \textit{contravariant $\Delta \bfG$-bar construction} of an algebra $R$ with twisted $G_0$-action in $\cC$ to be following composite:
\[
\begin{tikzcd}
   \Delta \bfG^{\op} \ar{r}{\simeq}[swap]{\sdual} & \Delta \bfG \ar{r}{B_{\bfG}^\bullet(R)} & [2em] \cC.
\end{tikzcd}
\]
We shall denote it $B^{\bfG}_\bullet(R)\colon \Delta\bfG^\op\to \cC$. The construction is entirely natural in $R$ and thus determines a functor $B_\bullet^{\bfG}\colon \Alg^{\lambda_0}(\cC)\to \cC^{\Delta \bfG^\op}$.
\end{defin}

\begin{exm}[{\cite[1.4]{Dun89}}]\label{self dual structure for cyclic, dihedral and quaternionic}
   The cyclic category $\Delta \bfC$, the dihedral category $\Delta \bfD$, and the quaternionic category $\Delta \bfQ$ are self-dual crossed simplicial groups. To define their self-dual equivalence $\sdual\colon \Delta \bfG^\op\xrightarrow{\cong} \Delta \bfG$, we make use of the following presentations of the groups:
   \[
   C_{n+1}= \langle x_n \mid x_n^{n+1}=1\rangle, \quad D_{2(n+1)}=\langle x_n, y_n \mid x_n^{n+1}=y_n^2=1, \, y_nx_ny_n^{-1}=x_n^{-1} \rangle 
   \]
   \[
   Q_{4(n+1)}= \langle x_n, y_n \mid x_n^{n+1}=y_n^2, \, y_nx_ny_n^{-1}=x_n^{-1} \rangle. 
   \]
   Notice that in $Q_{4(n+1)}$, the relations imply that $y_n^4=1$.
   Then the functor $\sdual$ is the identity on objects, and on morphisms we define it as follows:
   \begin{align*}
   \sdual\left([n-1]\xrightarrow{\delta_i} [n]\right) &=\begin{cases}
       [n]\xrightarrow{\sigma_i} [n-1] & \text{if }0\leq i \leq n-1\\
       [n]\xrightarrow{\sigma_0\circ x_n^{-1}} [n-1] & \text{if }i=n,
   \end{cases}\\
   \sdual\left( [n+1]\xrightarrow{\sigma_i}[n]\right) &=[n]\xrightarrow{\delta_{i+1}}[n+1],\\
   \sdual\left( [n]\xrightarrow{x_n^ry_n^s}[n]\right) &=\begin{cases}
    [n]\xrightarrow{x_n^{-r}y_n^s}[n] & \text{if }s=0,2  \\
     [n]\xrightarrow{x_n^{r+1}y_n^s}[n] & \text{if }s=1,3.
   \end{cases}
   \end{align*}
   Of course, $s=0$ always in the case of $\Delta \bfC$, and $s=0,1$ in the case of $\Delta \bfD$.
\end{exm}

\begin{exm}
  The symmetric category $\Delta \bfS$ is not a self-dual crossed simplicial group~\cite[1.4]{Dun89}. The reflexive category $\Delta \bfR$ is also not a self-dual crossed simplicial group. 
\end{exm}

\begin{defin}[Topological $\Delta \bfG$-homology: self-dual case]\label{hom of twisted G self-dual}
Let $(\Delta \bfG, \sdual)$ be a self-dual crossed simplicial group and $(G_0,\lambda_0)$ its associated group with parity.
Let $\cC$ be a symmetric monoidal $\infty$-category with geometric realizations. Given an algebra with twisted $G_0$-action in $\cC$, we define: 
\[
\HG(R/\cC)=\colim \Big( \begin{tikzcd}
 \Delta^{\op} \ar[hook]{r}{\iota} & \Delta \bfG^{\op}  \ar{r}{B_{\bullet}^{\bfG}(R)} & [1em] \cC   
\end{tikzcd} \Big).
\]
When $\cC=\Sp$, we write $\THG(R)\coloneqq \HG(R/\Sp)$, and refer to it as the \textit{topological $\Delta \bfG$-homology of ring spectra with twisted $G_0$-action}. 
\end{defin}

\begin{exm}\label{ex: contravariant Delta G bar construction}
As in \autoref{ex: covariant Delta G construction}, let $\cC$ be the (nerve of the) category of $k$-vector spaces.
Let $\Delta\bfG$ be either $\Delta \bfD$ or $\Delta \bfQ$.
Let $R$ be a $k$-algebra with twisted $G_0$-action induced by the anti-homomorphism $t\colon R^\op\to R$.
Then explicitly, the contravariant $\Delta \bfG$-bar construction of $R$ is a simplicial vector space, with degeneracies
\begin{align*}
    R^{\otimes n+1} & \stackrel{\sigma_i}\longrightarrow R^{\otimes n+2}\\
    r_0\otimes \cdots \otimes r_n &\longmapsto r_0 \otimes \dots \otimes r_{i} \otimes 1_R \otimes r_{i+1} \otimes \dots \otimes  r_n
\end{align*}
and faces
\begin{align*}
    R^{\otimes n+1} & \stackrel{\delta_i}\longrightarrow R^{\otimes n}\\
    r_0\otimes \cdots \otimes r_n &\longmapsto \begin{cases}
    r_0\otimes \cdots \otimes r_ir_{i+1} \otimes \cdots \otimes r_n & \text{if }0\leq i \leq n-1\\
    t^2(r_n)r_0\otimes r_1\otimes \cdots \otimes r_{n-1} & \text{if }i=n
    \end{cases}
\end{align*}
together with additional maps $R^{\otimes n+1}\to R^{\otimes n+1}$, generated by $x_n,y_n\in G_n$ as denoted in \autoref{self dual structure for cyclic, dihedral and quaternionic}:
\begin{align*}
    R^{\otimes n+1} & \stackrel{x_n}\longrightarrow R^{\otimes n+1}\\
    r_0\otimes \cdots \otimes r_n & \longmapsto t^2(r_n)\otimes r_0 \otimes \cdots \otimes r_{n-1}
\end{align*}
and 
\begin{align*}
    R^{\otimes n+1} & \stackrel{y_n}\longrightarrow R^{\otimes n+1}\\
    r_0\otimes \cdots \otimes r_n & \longmapsto t(r_0)\otimes t^3(r_n)\otimes \cdots \otimes t^3(r_1).
\end{align*}
For $\Delta \bfD$, this is the \textit{dihedral bar construction} on a $k$-algebra $R$ with anti-involution \cite[\S2]{Dun89}. For $\Delta \bfQ$, this is called the \textit{quaternionic bar construction} on the $k$-algebra $R$ with twisted $C_4$-action \cite[\S2]{Dun89}.
\end{exm}

We now explain how to equip $\HG(\--/\cC)$ with a 
$|\bfG_{\sbt}|$-action. 
Let $\iota\colon \Delta^{\op}\rightarrow \Delta \bfG^{\op}$ be the canonical inclusion. 
Let $\cC$ be any presentable $\infty$-category.
Then the restriction functor $\iota^*\colon \cC^{\Delta \bfG^{\op}}\rightarrow \cC^{\Delta^\op}$ admits a left adjoint $\iota_!$, the free $\Delta \bfG^{\op}$-object: 
\[
\begin{tikzcd}
   \cC^{\Delta^{\op}} \ar[shift left=2, bend left, "\iota_!" description]{r} & [3em]\cC^{\Delta \bfG^{\op}}. \ar[shift left=2, bend left, "\iota^*" description]{l}\ar[phantom, "\perp" description]{l}
\end{tikzcd}
\]
Given a simplicial object $X\colon \Delta^\op\rightarrow \cC$, the free $\Delta \bfG^\op$-object $\iota_!X$ is the left Kan extension of $X$ along $\iota\colon \Delta^\op\hookrightarrow \Delta \bfG^\op$.
Unwinding the definitions, we can check that for all $n\geq 0$, we obtain an equivalence in $\cC$ 
\[
\iota_!X([n]_{\Delta \bfG})\simeq G_n\otimes X_n,
\] 
where the tensor product is the tensoring of $\cC$ over spaces. This is the $\infty$-categorical analogue of \cite[4.3]{FL91}.

\begin{notn}\label{not:realization}
    Given a $\Delta \bfG^\op$-object $X\colon \Delta \bfG^\op\to \cC$, we write $|X|$ for the geometric realization (i.e.\ colimit) of the underlying simplicial object $\iota^\ast X$. 
\end{notn}

\begin{rem}
    Although $\iota_!(X)$ and $\bfG_{\sbt}\otimes X$ have the same $n$-simplices, they are different $\Delta\bfG^\op$-objects, even different simplicial objects.  We see however below that they have the same geometric realization.
\end{rem}

Since $\iota^*$ preserves all colimits and is conservative, since $\iota\colon \Delta^\op\hookrightarrow \Delta \bfG^\op$ is a wide subcategory, we obtain by Barr--Beck--Lurie monadicity  \cite[4.7.3.5]{HA} that $\iota^*$ is monadic.
Therefore we recover \cite[4.4]{FL91} in the $\infty$-categorical context as we obtain an equivalence of $\infty$-categories $\Alg_{\iota^*\iota_!}(\cC^{ \Delta^\op})\simeq \cC^{\Delta \bfG^\op}$.
We show now how this leads to the following generalization of \cite[5.3]{FL91} and \cite[B.5]{NS18}.

\begin{prop}\label{G action on DeltaG objects}
    Let $\cC$ be a presentable $\infty$-category.
    Let $\Delta \bfG$ be a crossed simplicial group.
We obtain the following. 
    \begin{enumerate}
        \item There is a natural functor
    \[
    \cC^{\Delta \bfG^\op} \longrightarrow \cC^{B|\bfG_{\sbt}|}
    \]
    from $\Delta \bfG^\op$-objects in $\cC$ to $|\bfG_{\sbt}|$-equivariant objects in $\cC$. The underlying object in $\cC$ is given by
    \[
    \big(\Delta \bfG^\op \stackrel{M}\rightarrow \cC \big) \longmapsto |M|.
    \]
    \item We have $|\iota_!(X)| \simeq |\bfG_{\sbt}|\otimes  | X | $ for any simplicial object $X$ in $\cC$, and this equivalence is $|\bfG_{\sbt}|$-equivariant if $X=\iota^* M$ for some $\Delta \bfG^\op$-object $M$ in $\cC$.
    
    \item\label{test} Given a $\Delta \bfG^\op$-object $M$ in $\cC$, then $\colim_{\Delta \bfG^\op} M$ is equivalent to the $|\bfG_{\sbt}|$-homotopy orbits of $|M|$.
    \end{enumerate}
\end{prop}

\begin{proof}
    If $\cC=\cS$, the $\infty$-category of spaces, this was proved in \cite[5.1, 5.3, 5.9]{FL91}. 
    More generally, for any simplicial set $K$, recall that $\cS^K \otimes \cC \simeq \cC^K$, where here $\otimes$ denotes tensor product in $\operatorname{Pr}^L$. This is due to the universal property of the presheaf category:
        \[
   \cS^{K} \otimes  \cC  \simeq  \Fun^L(\cS^{K}, \cC^{\op})^{\op}
    \simeq \cC^{K} \,. 
    \]
Here the first equivalence follows from \cite[4.8.1.17]{HA} and the second equivalence follows from~\cite[5.1.5.6]{HTT}. 
    Explicitly, given an object $C\in \cC$ and a functor $F\colon K\rightarrow \cC$, the assignment is specified on objects by sending $F\otimes C$ to $F_C\colon K\rightarrow \cC$, where $F_C(x)=C\otimes F(x)$ for any vertex $x$ in $K$, where the last tensor product is the tensoring of $\cC$ in spaces. 
    
   Unpacking this, the first statement follows from the case of spaces, as we have a commutative diagram:
    \[
    \begin{tikzcd}
        \cC \otimes \cS^{\Delta \bfG^\op} \ar{r}{\id_{\cC}\otimes (\colim \circ \iota^*_{\cS})} \ar{d}{\simeq}& [4em]\cC \otimes \cS^{B|\bfG_{\sbt}|} \ar{d}{\simeq}\\
        \cC^{\Delta \bfG^\op} \ar{r}{\colim \circ \iota^*_{\cC}} & \cC^{B|\bfG_{\sbt}|}
    \end{tikzcd}
    \]
    as the tensor product commutes with colimits and is compatible with left adjoints.
    Moreover, any simplicial object $X$ in $\cC$ is of the form $F_C$ where $C\in \cC$ and $F\colon \Delta^\op\rightarrow \cS$. We then obtain:
    \begin{align*}
|\iota_! (X)| & \simeq \colim_{\Delta^\op}\iota^*\iota_!(F_C)\\
 & \simeq C\otimes \colim_{\Delta^\op}\iota^*\iota_!(F)\\
 & \simeq C\otimes (|\bfG_{\sbt}| \times \colim_{\Delta^\op} F) \\
 & \simeq |\bfG_{\sbt}|\otimes  \colim_{\Delta^\op} X \,.
    \end{align*}
    The equivariant statement follows  as well from the string of equivalences above.
    The last statement follows again from the compatibility of homotopy orbits and tensoring over spaces.
\end{proof}

\begin{cor}\label{cor: THG exists for self-dual}
Given a self-dual crossed simplicial group $\Delta \bfG$  and a ring spectrum $R$ with a twisted $G_0$-action, then $\THG(R)$ is a spectrum with left $|\bfG_{\sbt}|$-action.
\end{cor}

\begin{exm}
When $\Delta \bfG$ is the cyclic category $\Delta \bfC$, we can identify $\TH\bfC(R)$ with the underlying spectrum with $S^1$-action of $\THH(R)$. 
In this case, our construction of $\THH(R)$ agrees with the one from Nikolaus--Scholze where $\THH(R)$ was defined as the homotopy colimit in $\Sp$ of the composition:
\[
\begin{tikzcd}
    \Delta^\op\ar[hook]{r} & \Delta \bfC^\op\simeq \Delta \bfC \ar{r} & \Env(\Assoc) \ar{r}{\Env(R)} & [2em]\Env(\Sp) \ar{r}{\otimes} & \Sp,
\end{tikzcd}
\]
see \cite[III.2.3]{NS18}.
Similarly, when $\Delta \bfG$ is the dihedral category $\Delta \bfD$, we can identify $\TH\bfD(R)$ with the underlying spectrum with left $O(2)$-action of $\THR(R)$, when $R$ is a ring spectrum with anti-involution. 
\end{exm}

\begin{exm}
Another important special case is 
\[
\THQ(R) \coloneqq \TH\bfQ(R)
\]
associated to the quaternionic category $\Delta \bfQ$. Let $q\colon  C_4\to C_2$ be the canonical quotient from the cyclic group of order 4. The input is then a ring spectrum $R$ with twisted $C_4$-action with respect to the group with parity $(C_4,q)$, i.e.\ a ring spectrum $R$ together with a ring homomorphism $t\colon R^\op\to R$ such that $t^4\simeq \id$.
We call this theory \emph{quaternionic topological Hochschild homology}.
\end{exm}

\begin{rem}
    The canonical map of crossed simplicial groups $\lambda\colon \Delta \bfQ\to \Delta \bfD$ induces a map of topological groups $\lambda\colon \Pin(2)\to O(2)$, and thus any spectrum $E$ with $O(2)$-action can be viewed as a spectrum $\lambda^*E$ with same underlying spectrum $E$ but with action given by $\lambda\colon \Pin(2)\to O(2)$.
    Moreover, the quotient homomorphism $q\colon C_4\to C_2$, allows one to view any ring spectrum $R$ with anti-involution $t\colon R^\op\to R$ as a ring spectrum $q^*R$ with twisted $C_4$-action, where as a ring $q^*R\simeq R$ and as $t^2\simeq \id$ we also get $t^4\simeq \id$.
    In this case, we obtain an equivalence of spectra with left $\Pin(2)$-action:
    \[
    \THQ(q^*R)\simeq \lambda^*\mathrm{THR}(R)
    \]
    for any ring spectrum $R$ with anti-involution.
    This follows from the commutativity of the diagram:
    \[
    \begin{tikzcd}
     \Delta \bfQ^\op \ar{d}[swap]{\lambda^\op} \ar[description, phantom]{r}{\simeq} & [-2em] \Delta \bfQ \ar{d}{\lambda}\ar{r}{B^{\bullet}_\bfQ(q^*R)} & [3em]\Sp \ar[equals]{d}\\
     \Delta \bfD^\op \ar[description, phantom]{r}{\simeq} &  \Delta \bfD \ar{r}{B^{\bullet}_\bfD(R)} & \Sp.
    \end{tikzcd}
    \]
\end{rem}

The next lemma shows that from a given self-dual crossed simplicial group, we can produce many others by taking the product levelwise with some fixed discrete group. 

\begin{lem}\label{extension by constant has duality}
If $(\Delta \bfG, \sdual)$ is a self-dual crossed simplicial group and $\mathrm{H}$ is any discrete group, then the crossed simplicial group $\Delta \bfGH $ from \autoref{extension by constant} is a self-dual crossed simplicial group with duality 
\[ 
\overline{\sdual} \colon (\Delta\bfGH)^{\op}\to  \Delta \bfGH
\]
extended uniquely from $\sdual\colon \Delta \bfG^{\op}\rightarrow \Delta \bfG$, by assigning $\overline{\sdual}(h)=h^{-1}$ for all $h\in \mathrm{H}$.
\end{lem}

\begin{proof}
Recall that as any crossed simplicial group, a morphism in $\Delta \bfGH$ can be uniquely written as $\phi\circ (g,h)$ where $\phi\in \Delta([n], [m])$, $g\in G^\op_n$ for some $n\geq 0$ and $h\in H$.
As we need a functor $\overline{\sdual}\colon (\Delta \bfGH )^{\op}\rightarrow \Delta \bfGH$, we must have:
\[
\overline{\sdual}(\phi\circ (g,h))\coloneqq\overline{\sdual}(g,h)\circ \overline{\sdual}(\phi)=(\sdual(g), h^{-1})\circ \sdual(\phi). 
\]
The identity $\overline{\sdual}^2=\id$ is immediate.
We are therefore only left to check that $\overline{\sdual}$ is a functor. As $\overline{\sdual}(\id_{[n]})=\sdual(\id_{[n]})=\id_{[n]}$, we only need to check composition. 
Let $\phi_1\colon [n]\rightarrow [m]$ and $\phi_2\colon [m]\rightarrow [k]$ be maps in $\Delta$, $g_1\in G^\op_n$ and $g_2\in G^\op_m$, and $h_1, h_2\in H$. We need to verify the equality:
\[
\overline{\sdual} \big( (\phi_2 \circ (g_2, h_2))\circ (\phi_1 \circ (g_1, h_1)) \big) = \overline{\sdual} \big( \phi_1 \circ (g_1, h_1) \big) \circ \overline{\sdual}\big(\phi_2 \circ (g_2, h_2)\big).
\]
Recall that we can compose in $\Delta \bfGH$ the following way:
\begin{align*}
    \phi_2 \circ (g_2, h_2) \circ \phi_1 \circ (g_1, h_1) & = \phi_2 \circ (g_2, h_2)^*(\phi_1) \circ \phi_1^*(g_2,h_2) \circ (g_1, h_1)\\
    & = \phi_2 \circ g_2^*\phi_1 \circ (\phi_1^*g_2, h_2)\cdot (g_1, h_1)\\
    & = (\phi_2 \circ g_2^*\phi_1) \circ ( (\phi_1^*g_2)g_1, h_2 h_1).
\end{align*}
We then obtain the following string of equalities using that $\sdual\colon \Delta\bfG^\op\rightarrow \Delta \bfG$ is a functor: 
\begin{align*}
    \overline{\sdual} \big( (\phi_2 \circ (g_2, h_2))\circ (\phi_1  &\circ (g_1, h_1)) \big)   = \overline{\sdual}\big( (\phi_2 \circ g_2^*\phi_1) \circ ( (\phi_1^*g_2)g_1, h_2 h_1)\big)\\
    & = \big( \sdual((\phi_1^*g_2)g_1), (h_2h_1)^{-1}\big) \circ \sdual(\phi_2\circ g_2^*\phi_1)\\
    & = \big(\sdual(g_1)\sdual(\phi_1^*g_2), h_1^{-1}h_2^{-1} \big) \circ \sdual(g_2^*\phi_1)\circ \sdual(\phi_2)\\
    & = (\sdual(g_1), h_1^{-1})\circ (\sdual(\phi_1^*g_2), h_2^{-1}) \circ \sdual(g_2^*\phi_1) \circ \sdual(\phi_2).
\end{align*}
Now using the unique factorization of the crossed simplicial group $\Delta \bfGH$, we obtain that the middle term in the last equality can be written as:
\begin{align*}
(\sdual(\phi_1^*g_2), h_2^{-1}) \circ \sdual(g_2^*\phi_1) & = \big(\sdual(\phi_1^*g_2), h_2^{-1} \big)^*\sdual(g_2^*\phi_1) \circ \sdual\big( g_2^*\phi_1\big)^*(\sdual(\phi_1^*g_2), h_2^{-1})\\
& = \sdual(\phi_1^*g_2)^*\sdual(g_2^*\phi_1) \circ \big( \sdual(g_2^*\phi_1)^*\sdual(\phi_1^*g_2), h_2^{-1}\big)\\
& = \sdual(\phi_1) \circ (\sdual(g_2), h_2^{-1}).
\end{align*}
The last equality follows from applying the functor $\sdual\colon \Delta \bfG^\op\rightarrow \Delta \bfG$ to the commutative diagram:
\[
\begin{tikzcd}
    {[n]} \ar{r}{g_2^*\phi_1} \ar{d}[swap]{\phi_1^*g_2} & {[m]} \ar{d}{g_2}\\
    {[n]} \ar{r}[swap]{\phi_1} & {[m]}
\end{tikzcd}
\]
\vspace{-0.1cm}
so that we obtain the commutative diagram:
\[
\begin{tikzcd}
    {[m]} \ar{r}{\sdual(\phi_1)} \ar{d}[swap]{\sdual(g_2)} & {[n]} \ar{d}{\sdual(\phi_1^*g_2)}\\
    {[m]} \ar{r}[swap]{\sdual(g_2^*\phi_1)} & {[n]}.
\end{tikzcd}
\]
Uniqueness of the factorization in $\Delta \bfG$ implies
\[
\sdual(\phi_1^*g_2)^*\sdual(g_2^*\phi_1)=\sdual(\phi_1), \quad \text{and }\quad \sdual(g_2^*\phi_1)^*\sdual(\phi_1^*g_2)=\sdual(g_2).
\]
Therefore, we have obtained the desired equality:
\begin{align*}
     \overline{\sdual} \big( (\phi_2 \circ (g_2, h_2))\circ (\phi_1  &\circ (g_1, h_1)) \big) \\
     & = (\sdual(g_1), h_1^{-1})\circ (\sdual(\phi_1^*g_2), h_2^{-1}) \circ \sdual(g_2^*\phi_1) \circ \sdual(\phi_2)\\
     & =(\sdual(g_1), h_1^{-1})\circ \sdual(\phi_1) \circ (\sdual(g_2), h_2^{-1}) \circ \sdual(\phi_2)\\
     & = \overline{\sdual} \big( \phi_1 \circ (g_1, h_1) \big) \circ \overline{\sdual}\big(\phi_2 \circ (g_2, h_2)\big). \qedhere
\end{align*}
\end{proof}

\begin{exm}
In the case of the crossed simplicial group 
$\Delta \bfC\times\mathrm{G}$ defined in \autoref{extension by constant}, this provides a combinatorial model for $G$-equivariant topological Hochschild homology for any discrete group $G$ as spectrum with left $S^1\times G$-action.  
\end{exm}

\begin{exm}
We also produce $G$-equivariant Real topological Hochschild homology from the crossed simplicial group $\Delta \bfD\times G$ as in \autoref{extension by constant}.
\end{exm}

\begin{defin}\label{positive, negative THH}
Given a self-dual crossed simplicial group $\Delta \bfG$ with associated group with parity $(G_0,\lambda_0)$ and a twisted $G_0$-ring $R$, we define
\begin{align*}
    \rmT\bfG^{+}(R)\coloneqq & \THG(R)_{h|\bfG_{\sbt}|} \\
 \rmT\bfG^{-}(R)\coloneqq & \THG(R)^{h|\bfG_{\sbt}|},\\
	\rmT\bfG^{\textup{per}}(R)\coloneqq & \THG(R)^{t|\bfG_{\sbt}|}\,.
\end{align*}
\end{defin}

\begin{rem}
If $k$ is a commutative ring and $A$ is a $k$-algebra, classical cyclic homology $\operatorname{HC}_*(A/k)$ of $A$ corresponds to the homotopy orbits
$\mathrm{HH}(A/k)_{hS^1}$,
where $\mathrm{HH}(A/k)$ denotes Hochschild homology of $A$ over $k$,  and  negative cyclic homology $\operatorname{HC}_*^{-}(A/k)$ corresponds to the homotopy fixed points $\operatorname{HH}(A/k)^{hS^1}$ (cf. \cite{Hoy15}). Topological cyclic homology, however is not the topological analogue of cyclic homology. Instead, it is an invariant that requires the stable homotopy category and not just the derived category of abelian groups, see \cite[Remark~III.1.9]{NS18} and \cite[\S~12]{Law21}. We therefore refer to cyclic homology as \emph{positive cyclic homology} instead and denote it  $\operatorname{HC}_*^+(A)$. Our definition of topological positive cyclic homology of rings with twisted $G$-action is a topological analogue of this construction. We therefore write, for example,
\begin{align*} 
\rmT\bfC^{+}(R)\coloneqq \THH(R)_{hS^1}, \\
\rmT\bfC^{-}(R)\coloneqq \THH(R)^{hS^1}, \\
\rmT\bfC^{\per}(R)\coloneqq \THH(R)^{tS^1}
\end{align*}  
and reserve $\TC(R)$ for the equalizer of~\cite[Thm. III.1.10]{NS18}. We also write $\rmT\bfC^{\per}$ following~\cite{Lod92}, instead of $\TP$ as in~\cite{Hes18}.
\end{rem}

\begin{rem}
In the case of the dihedral category, one usually considers $\mathrm{TCR}^{-}$ and  $\mathrm{TPR}$ as $C_2$-equivariant objects and these do not correspond to $\rmT\bfD^{-}$ and $\rmT\bfD^{\per}$. 
Instead, our notions match the algebraic notions \cite{Lod87}. For example, $\TCR^{-}(R)=\THR(R)^{h_{C_2}S^1}$ whereas $\mathrm{T}\mathbf{D}^{-}(R)=\THR(R)^{hO(2)}$ (cf.~\cite{Ung16,QS21}).
\end{rem}

\subsection{Topological $\Delta \bf G$-homology: general case}\label{sec:homology}
Classically, the algebraic analogue of the topological $\Delta \bfG$ homology of rings with twisted $G$-action (called the homology of crossed simplicial groups in~\cite{FL91}) was defined using the derived functor of the functor tensor product. We consider an $\infty$-categorical generalization of  this construction and use it to define the homology of ring spectra with twisted $G$-action. 

\subsubsection{Homotopical functor homology}\label{sec:functor-homology}
One may also view the constructions in this paper as a form of derived functor homology. 
\begin{const}[Derived functor tensor product]\label{functor-tensor}
Suppose $\cD$ is a cocomplete symmetric monoidal $\infty$-category. Given functors $M \colon\cC^{\op} \longrightarrow \cD$ and $N\colon \cC\longrightarrow \cD$ where $\cC$ is a small $\infty$-category we define the (derived) functor tensor product to be the coend
\[
M\otimes_{\cC}N \coloneqq\int^{c\in \cC}M(c)\otimes N(c).
\] 
\end{const}

\begin{exm}
    If $\cD$ is the underlying $\infty$-category of chain complexes $\mathsf{Ch}_k$ of $k$-vector spaces, where $k$ is a field, and $\cC=N\C$ is the nerve of a small category $\C$, then given functors $M\colon \C^\op\rightarrow \mathsf{Ch}_k$ and $N\colon \C\rightarrow \mathsf{Ch}_k$, we obtain that the induced tensor product $M\otimes_{\cC} N$ in $\mathsf{Ch}_k$ corresponds to the derived tensor product $M\otimes_{\C}^\mathbb{L} N$. In particular, we get $H_*(M\otimes_{\cC} N)\cong\operatorname{Tor}_*^{\C}(M,N)$.
\end{exm}

\begin{notation}
Given an $\infty$-category $\cD$ and a small $\infty$-category $\cC$, we write 
\[ \underline{(-)}\colon \cD \to \Fun(\cC,\cD)\]
for the functor that sends $d\in \cD$ to the constant functor $\underline{d}$.
\end{notation}

\begin{rem}
    If $\cD$ is a presentable monoidal $\infty$-category with unit $\mathbbm{1}$ and $\cC$ is a small $\infty$-category, then for any functors $M\colon \cC^\op\rightarrow \cD$ and $N\colon \cC\rightarrow \cD$, we obtain $M\otimes_{\cC} \underline{\mathbbm{1}}\simeq \colim_{\cC^\op} M$ and $\underline{\mathbbm{1}}\otimes_{\cC} N\simeq \colim_{\cC} N$.
    For instance, if $X_{\sbt}$ is a simplicial space, we recover the familiar formula $|X_{\sbt} | \simeq X_{\sbt} \times^\mathbb{L}_\Delta \underline{\ast} \simeq X_{\sbt}\times_\Delta \Delta^\bullet$, where we write $\times^{\mathbb{L}}$ to emphasize that this construction is the \emph{derived} functor tensor product. From now on all functor tensor products are implicitly derived. 
\end{rem}

\begin{defin}[Homotopical functor homology]
Given a cocomplete symmetric monoidal $\infty$-category $\cD$, a small $\infty$-category $\cC$, and functors $G\colon\cC\to\cD$ and $F\colon\cC^{\op}\to \cD$, we define the \emph{homotopical functor homology} to be 
\[ 
\mathrm{Tor}_*^{\cC}(G,F)=\pi_*(G\otimes_{\cC}F).
\]
\end{defin}

\subsubsection{Positive topological \texorpdfstring{$\Delta \bfG$}{TEXT}-homology}

We can now present our definition of positive topological $\Delta \bfG$-homology as a special case of \autoref{functor-tensor}. 
 
\begin{defin}\label{topological positive}
Let $\Delta \bfG$ be a crossed simplicial group with associated group with parity $(G_0,\varphi)$.
Let $\cC$ be a cocomplete symmetric monoidal $\infty$-category. Given a ring with twisted $G_0$-action $R$ in $\cC$, we define 
\[
\mathrm{H}\mathbf{G}^{+}(R/\cC)=\colim_{\Delta \bfG} B^\bullet_{\bfG}(R) \,.
\]
When $\cC=\Sp$, we simply write $\mathrm{T}\mathbf{G}^+(R)\coloneqq \mathrm{H}\mathbf{G}^+(R/\Sp)$.
\end{defin}

\begin{rem}
We can identify 
\[ 
\pi_*\HG^+(R/\cC)\cong \Tor_*^{\Delta G}(\underline{\mathbbm{1}},B^{\bullet}_{\bfG}R)\,.
\]
This makes it clear that there are bivariant analogues of all of the constructions in this paper as in~\cite[\S~5.5]{Lod87}.
\end{rem}

\begin{rem}\label{rem: compatibility of TG and THG}
  \autoref{topological positive} is compatible with \autoref{positive, negative THH}.
    Indeed, recall that if $(\Delta \bfG, \sdual)$ is a self-dual crossed simplicial group, then $\rmT\bfG^+(R)$ was defined as the $|\bfG_{\sbt}|$-homotopy orbits of $\THG(R)=|B_{\bullet}^{\bfG}(R)|$. 
    By \autoref{G action on DeltaG objects}, these homotopy orbits are equivalent to $\colim_{\Delta \bfG^{\op}}B_{\bullet}^{\bfG}(R)\simeq \colim_{\Delta \bfG}B_{\bfG}^\bullet(R)$.
\end{rem}

We highlight some special cases of \autoref{topological positive} to introduce new terminology. In these cases we remove the superscript $+$ because these are not self-dual crossed simplicial groups and therefore no ambiguity should arise.

\begin{exm}\label{twisted symmetric example}
Given a parity $\varphi\colon G\to C_2$, when $\Delta \bfG=\Delta \varphi\wr \bfS$ is the twisted symmetric crossed simplicial group, and $A$ is a monoid in $\cC$ with twisted $G$-action, then we write:
\[
\mathrm{H}\varphi(A/\cC)\coloneqq \mathrm{H}\varphi\wr \bfS^+(A/\cC)
\]
for the construction from \autoref{topological positive} and refer to it as \emph{topological twisted symmetric  homology}. We further write $\mathrm{T}\varphi(A):=\mathrm{H}\varphi(A/\Sp)$.
\end{exm}
We highlight two special cases of this construction below.
 \begin{exm}\label{hyperoctahedral}
 When $\Delta \bfG=\Delta\bfH$ is the hyperoctahedral crossed simplicial group and $A$ is an $\bE_1$ algebra with anti-involution, then we write 
 \[
 \mathrm{H}\mathrm{O}(A/\cC)\coloneqq \mathrm{H} \mathbf{H}^{+}(A/\cC)
 \] 
and refer to this construction as \emph{topological hyperoctahedral homology}. When $\cC=\Sp$, we further write 
\[ \TO(A)\coloneqq  \mathrm{H}\mathrm{O}(A/\Sp)\,.\]
In the algebraic context, this was previously studied by~\cite{Fie,Gra22}. 
 \end{exm}
 
 \begin{exm}\label{symmetric}
 When $\Delta \bfG$ is the symmetric crossed simplicial group $\Delta \bfS$ and $A$ is a $\bE_1$ algebra in $\cC$ then we write 
\[ 
   \mathrm{H}\Sigma (A/\cC)\coloneqq\mathrm{H}\bfS^{+}(A/\cC)
\] 
for the construction from \autoref{topological positive} and refer to it as \emph{topological symmetric homology}. When $\cC=\Sp$, we write 
\[
\TS (A)\coloneqq\mathrm{H}\mathrm{S}(A/\Sp)\,.
\]
This is an interesting invariant especially in light of work of Berest--Ramados~\cite{BR23} relating symmetric homology to representation homology. In the algebraic context, this was previously studied by Fiedorowicz~\cite{Fie} and Ault~\cite{Aul10}.
\end{exm}

\begin{exm}
When $\Delta\bfG=\Delta \mathbf{B}$ is the braid crossed simplicial group and $A$ is a $\bE_1$ algebra in $\cC$ with anti-involution, then we write 
\[ \mathrm{H}\mathrm{B}(A/\cC)\coloneqq \mathrm{H}\mathbf{B}^{+}(A/\cC)\]
for the construction from \autoref{topological positive} and refer to it as \emph{topological braid homology}. If $\cC=\Sp$, we write 
\[ 
\TB(A):=\mathrm{H}\mathrm{B}(A/\Sp) \,.
\]  
\end{exm}


\begin{exm}
When $\Delta \bfG=\Delta \bfR$ is the reflexive category and $A$ is an $\bE_1$-algebra in $\cC$ with anti-involution then we define the reflexive bar construction $B^{\textup{ref}}_{\bullet}R$ to be the composite of the functor 
\[
\Delta \bfR^{\op}\to \Delta \mathbf{D}^{\op}
\]
with the dihedral bar construction. 
We write 
\[ 
\mathrm{H}\mathcal{R}(A/\cC) \coloneqq \colim_{\Delta \bfR^{\op}}B^{\textup{ref}}_{\bullet}R 
\]
for topological reflexive homology of $A$. Here use $\mathcal{R}$ to distinguish from topological restriction homology, which is commonly denoted $\mathrm{TR}$. When $\cC=\Sp$, we write 
\[\TR(A)\coloneqq\TR(A/\Sp)\,.\] 
This is an interesting invariant in light of recent work of~\cite{Gra22b} and~\cite{LR24} relating the algebraic analogue of reflexive homology to involutive homology of~\cite{Bra14,FVG18} and the equivariant Loday construction respectively. (Note that there are two variants of the dihedral bar construction, see~\cite{KLS88} for example, which leads to two variants of reflexive homology as considered in~\cite{Gra22b}, but we prefer not to go into detail on this distinction here.)
\end{exm}

\begin{rem}\label{rem: canonical map to top twisted symmetric homology}
    Given a crossed simplicial group $\Delta\bfG$, and its canonical parity $\lambda_0\colon G\to C_2$, the functor $\widetilde{\lambda}\colon \Delta\bfG\to \Delta\lambda_0\wr \bfS$ defined in \autoref{canonical map} induces a functor:
    \[
    \widetilde{\lambda}^*\colon \Fun({\Delta \lambda_0\wr \bfS}, \Sp)\longrightarrow\Fun({\Delta \bfG}, \Sp).
    \]
    Given $R$ a ring spectrum with twisted $G_0$-action, notice that $\widetilde{\lambda}^*(B_{\lambda_0\wr \bfS}^\bullet(
R))=B_{\bfG}^\bullet(R)$ and thus we obtain a natural map $\colim_{\Delta \bfG}B_{\bfG}^\bullet(R)\to \colim_{\Delta \lambda_0\wr \bfS}B_{\lambda_0\wr \bfS}^\bullet(
R)$, i.e., we obtain a natural map $\mathrm{T}\bfG^+(R)\to \mathrm{T}\lambda_0(R)$.
In particular, if $\Delta \bfG$ is self-dual, we also get $\THG(R)\to \mathrm{T}\lambda_0(R)$ by pre-composing with the quotient $\THG(R)\to \mathrm{T}\bfG^+(R)$.
\end{rem}

For the ease of the reader, we provide a table of the constructions in this paper along with the notations we use in \autoref{fig: new table}.

\begin{figure}
\footnotesize
\[
\begin{array}{c|c|c|c|c|c|c|c}
    \text{CSG} & \Delta \bfG & G_n & |\mathbf{G}_{\sbt}| & \THG  & \mathrm{T}\bfG^{+} &  \mathrm{T}\bfG^{-}& \mathrm{T}\bfG^{\per} \\ \hline 
\text{cyclic} & \Delta \bfC & C_{n+1} & S^1 & \THH  & \TC^{+}  & \TC^{-} & \TC^{\per} \\ 
\text{equivariant cyclic} & \Delta \bfCG & C_{n+1}\times G & S^1 \times G & \THHG & \TCG^{+} &  \TCG^{-} &  \TCG^{\per} \\ 
\text{dihedral} & \Delta \bfD & D_{2(n+1)} & \rmO(2) & \THR & \TD^{+} &  \TD^{-}& \TD^{\per}\\ 
\text{quaternionic} & \Delta \bfQ & Q_{4(n+1)} &  \Pin(2)  &  \THQ  &  \TQ^{+} & \TQ^{-} & \TQ^{\per}\\ 
\text{reflexive} & \Delta \bfR & C_2 & C_2 &   &     \mathrm{T}\mathcal{R} &   &  \\
\text{braid} & \Delta \bfB & B_{n+1} & * &     & \TB &     &  \\
\text{symmetric} & \Delta \bfS     & \Sigma_{n+1} & * &     & \TS &    &  \\
\text{twisted symmetric} & \Delta \varphi\wr \bfS     & G\wr \Sigma_{n+1} & * &   &   \mathrm{T}\varphi &      &  \\
\text{hyperoctahedral} & \Delta \bfH     & C_2\wr \Sigma_{n+1} & * &     & \TO &       &   
\end{array}
\]
\caption{Various crossed simplicial groups and their topological homology}\label{fig: new table}
\end{figure}


\begin{rem2}
Given a symmetric monoidal $\infty$-category $\cC$, such as the $\infty$-category of spectra, the operad for monoids with twisted $G$-actions from \autoref{TwistedGringoperad} also provides the notion of coalgebras with twisted $G$-action as $\infty$-operad maps $\Assoc^\varphi\to \cC^{\op}$.
In fact, from  \cite{PerEnrichment}, we get that the $\infty$-category of rings with twisted $G$-action is enriched in the $\infty$-category of coalgebras with twisted $G$-action.
It should be possible to provide generalizations of our constructions \autoref{hom of twisted G self-dual} and \autoref{topological positive} for coalgebras with twisted $G$-action, generalizing the construction of topological coHochschild homology of \cite{HScoTHH}.
Just as in \cite{BPcoTHH}, there should be a duality between the topological homology of rings and  coalgebras with twisted $G$-action, under some finiteness conditions.
The resulting homology on coalgebras would provide trace methods and shadow functors in this setting, generalizing work of \cite{KP23} and \cite{AGHKK23}. 
Our higher categorical approach in this current paper is required in the coalgebraic setting, since there are no good model categories for coalgebras \cite{PScoalgebras, PerDK}, or are limited if they do exist \cite{PerDoldKan, Percomod}.
\end{rem2}

\section{Topological quaternionic homology of loop spaces 
}\label{computations}\label{def: the circles}
In this section, we provide some computations of topological homology of rings with twisted $G$-action in the case of self-dual crossed simplicial groups. Specifically, in \autoref{thm: main computation} we compute $\mathrm{TH}\mathbf{G}$ of spherical group rings in the case of the cyclic category, the dihedral category, and the quaternionic category. This recovers the previously known result in the case of the cyclic category and the dihedral category, and it is entirely new in the case of the quaternionic category. To do this, we first discuss induced twisted $G$-actions in \autoref{sec:induced-twisted-actions} and we present a construction of unstable topological twisted $G$-homology in some special cases in \autoref{sec:unstable}. 

One perspective on the constructions in this paper is that they are combinatorial models for norms for compact Lie groups, as mentioned in the introduction. For example, we view topological Hochschild homology as the norm from the trivial group to $S^1$~\cite{ABGHLM18}, Real topological Hochschild homology as the norm from the cyclic group of order two to $O(2)$~\cite{AKGH21}, and quaternionic topological Hochschild homology can conjecturally be viewed as the norm from the cyclic group of order four to $\Pin(2)$. This perspective informs some of the computations in this section. 

To start, it will be helpful to understand $O(2)$ as a left $C_2$-space and $\Pin(2)$ as a left $C_4$-space. Note that we give preference to left actions so, for example, given a subgroup $H<G$ and a left $H$-set $T$, we can consider $G$ as a $(H,G)$-biset and define the coinduction
\[ \Map^{H}(G,T)\]
as a space with left $G$-action, coming from the right $G$-action on $G$, and induction 
\[ \Ind_{H}^{G}T=G\times_{H}T 
\]
equipped with a left $G$-action coming from the left $G$-action on $G$. Each of these constructions works in spaces with continuous actions by topological groups as well. 

We can view the topological groups $O(2)$ and $\Pin(2)$ as two copies of the circle as follows.
View $O(2)\subseteq \mathrm{GL}_2(\mathbb{R})$.
As usual, let
\[
S^1\coloneqq \left\lbrace \begin{bmatrix}
\cos(\theta) & -\sin(\theta) \\
\sin(\theta) & \cos(\theta)
\end{bmatrix} \mid \theta\in\mathbb{R} \right\rbrace
\]
which represents the rotations in $\mathbb{R}^2$, and let 
\[
S^1\rho \coloneqq \left\lbrace \begin{bmatrix}
\cos(\theta) & \sin(\theta) \\
\sin(\theta) & -\cos(\theta)
\end{bmatrix} \mid \theta\in\mathbb{R} \right\rbrace
\]
where here $\rho=\begin{bmatrix}
    1 & 0\\
    0& -1
\end{bmatrix}$ {represents the reflection across the $x$-axis. (Regarding $\mathbb{R}^{2}$ as a $C_{2}$-space with $C_{2}$-action by $\rho$ identifies it with the regular representation, hence the name $\rho$.)
Then note that as a topological group} \[O(2)\cong S^1\cup S^1\rho.\]
The element $\rho$ generates a (non-normal) subgroup of $O(2)$ of order $2$. Right multiplication by $\rho$ on $O(2)$ defines a right $C_2$-action that {interchanges the two copies of the circle}:
\[
\cdot \rho\colon S^1\cup S^1\rho \longrightarrow S^1\rho \cup S^1
\]
while left multiplication of $\rho$ defines a left $C_2$-action that {flips each copy of the circle across the $x$-axis and then interchanges the copies:}
\[
\rho\cdot\colon S^1\cup S^1\rho \longrightarrow \rho S^1 \cup \rho S^1 \rho. 
\]
We could have used the model $S^1\cup \rho S^1$ and as left $C_2$-spaces they are equivalent. Note that there is a $C_2$-equivariant equivalence $C_2\times S^1\simeq C_2\times S^{\sigma}$ where $S^{\sigma}$ is the circle with $C_2$-action given by flipping across the $x$-axis.
Consequently, notice that 
\[
O(2)\simeq \Ind_e^{C_2}(S^1)
\] 
as left $C_2$-spaces. 

View $\Pin(2)\simeq S^1\cup S^1 j\subseteq \mathbb{H}$ where here $S^1=\{a+bi \mid a^2+b^2=1\}$ and $S^1j=\{aj+bk\mid a^2+b^2=1\}$.
The element $j$ generates a subgroup of order $4$. Right multiplication by $j$ induces a right $C_4$-action on $\Pin(2)$ that interchanges the copies of the circle, and rotates the second copy of the circle by $180^\circ$: 
\[
\cdot j\colon S^1\cup S^1j\longrightarrow S^1j\cup -S^1.
\]
Left multiplication by $j$ induces a left $C_4$-action on $\Pin(2)$ that first {interchanges the two copies of the circle, while flipping the first one across the $x$-axis and the second one across the $y$-axis, i.e., the second copy is flipped across the $x$-axis and then rotated by $180^\circ$:}
\[
j\cdot \colon S^1\cup S^1j\rightarrow jS^1\cup jS^1j.
\]
We could have used $S^1\cup jS^1$ as a model for $\Pin(2)$. Of course, the sets $jS^1$ and $S^1j$ are the same, but the difference is that when we identify $S^1\cup jS^1$ with two copies of $S^1$ we identify the point $a+bi$ in the first copy with the point $aj-bk$ in the second copy, whereas when we identify $S^1\cup S^1j$ with two copies of $S^1$ we identify the point $a+bi$ in the first copy with the point $aj+bk$ in the second copy, so there is no implicit flipping of one circle when we interchange copies of the two circles. With these interpretations, we can check that $S^1\cup S^1j\simeq S^1\cup jS^1$ are equivalent as left $C_4$-spaces.
Consequently, we notice that 
\[ 
\Pin(2)\simeq \Ind_{C_2}^{C_4}(S^1)\] 
as left $C_4$-spaces.

We shall introduce categorical models of $S^1$, $O(2)$ and $\Pin(2)$. We first need some general categorical constructions.

\subsection{Induced twisted \texorpdfstring{$G$}{TEXT}-actions}\label{sec:induced-twisted-actions}
If $H$ is a subgroup of $G$, given a parity on $G$, {we can restrict it to a parity $H\to C_2$ and we can restrict a twisted $G$-action on a category to a twisted $H$-action.} Therefore, just as in the untwisted case, there should be an induced (and coinduced) twisted $G$-action on a category with twisted $H$-action. We focus on some specific examples.

\begin{const}\label{const: induced twisted C_4-action}
Let $q\colon C_4\rightarrow C_2$ be the quotient homomorphism. Given a category $\cC$ with twisted $C_4$-action $t\colon \cC^\op\to \cC$, we obtain that $\cC$ has also a (non-twisted) $C_2$-action via $t^2\colon \cC\to\cC$.
More formally, viewing $C_2$ as a subgroup of $C_4$ defines an injective homomorphism $\textup{inc} \colon C_2\hookrightarrow C_4$. 
Note that $q\circ \textup{inc}$ is the zero map, so we obtain a commuting diagram 
\[
\begin{tikzcd}
    BC_2 \ar{r}{\textup{inc}}\ar{dr}{0} & BC_4 \ar{d}{q}\ar{r}{\cC} & C_2\rtimes \Cat, \ar{dl}\\
& BC_2 & 
\end{tikzcd}
\]
{which} defines a functor $\Res_{C_2}^{q}\coloneqq\inc^* \colon \Cat^{q}\to \Cat^{BC_2}$ between presentable categories that preserves limits, so it must have a left adjoint that we denote \[\Ind_{C_2}^{q}\colon \Cat^{BC_2}\rightarrow \Cat^{q}.\]
We can check that if $\cC$ is a category with a $C_2$-action $\alpha\colon \cC\rightarrow \cC$, i.e. $\alpha^2=\id$, then $\Ind_{C_2}^{q}(\cC)=\cC \coprod \cC^\op$ with twisted $C_4$-action given by:
\[
\begin{tikzcd}
    (\cC\coprod \cC^\op)^\op\simeq \cC^\op\coprod \cC \ar{r}{\text{swap}} & \cC\coprod\cC^\op \ar{r}{\alpha\coprod \id} & \cC\coprod \cC^\op.
\end{tikzcd}
\]
Given $\cD$ a category with twisted $C_4$-action, there is indeed a natural correspondence between twisted $C_4$-equivariant functors $\cC\coprod\cC^\op\rightarrow \cD$ and (non-twisted) $C_2$-equivariant functors $\cC\to \textup{Res}_{C_2}^q(\cD)$. Indeed, given a twisted $C_4$-equivariant functor $F\colon \cC\coprod\cC^\op\rightarrow \cD$, its restriction on the first copy of $\cC$ provides the desired equivariant functor. Conversely, given a $C_2$-equivariant functor $G\colon \cC\rightarrow \cD$, define a new functor $G\coprod tG\colon \cC\coprod \cC^\op\to \cD$ where $tG$ is the composition $G^\op\colon \cC^\op\rightarrow \cD^\op$ with the twisted $C_4$-action $t\colon \cD^\op\rightarrow \cD$.

\noindent
Similarly, we can forget the twisted $C_2$-action on any category and this defines a functor $\Res_{e}^{\tau} \colon \Cat^{\tau}\to \Cat$, with left adjoint $\Ind_{e}^{\tau}\colon \Cat\rightarrow \Cat^{\tau}$ where if $\cC$ is a category, we get $\Ind_{e}^{\tau}(\cC)=\cC\coprod \cC^\op$ with twisted $C_2$-action given by swapping the copies. 
\end{const}

\begin{const}\label{const: variant wit flip induced twisted C_4 actions}
    We can vary \autoref{const: induced twisted C_4-action} as follows. Suppose $\cC$ has both a twisted $C_2$-action $\tau\colon \cC^\op\rightarrow \cC$ and (non-twisted) $C_2$-action $\alpha\colon \cC\rightarrow \cC$, i.e. suppose $\cC$ is a twisted $C_2\times C_2$-category with parity  $C_2\times C_2\to C_2$ given by projection {onto the first factor}. Then $\cC\coprod \cC$ is also a twisted $C_4$-category 
    via:
    \[
    \begin{tikzcd}
(\cC \coprod \cC)^\op \simeq \cC^\op\coprod \cC^\op \ar{r}{\tau\coprod \tau} & \cC\coprod \cC \ar{r}{\text{swap}} & \cC \coprod \cC \ar{r}{\alpha\coprod \id} & \cC \coprod \cC.
    \end{tikzcd}
    \]
Comparing with \autoref{const: induced twisted C_4-action}, forgetting the twisted $C_2$-structure on $\cC$, we had a twisted $C_4$-structure on $\cC\coprod \cC^\op$ via $\Ind_{C_2}^{q}$.
It turns out we obtain an equivalence:
\[
\cC\coprod\cC^\op\simeq \cC\coprod  \cC
\]
of categories with twisted $C_4$-action,  via the functor $\id\coprod\tau\colon \cC\coprod \cC^\op\to \cC\coprod\cC$.

Similarly, suppose $\cC$ has only a twisted $C_2$-action, then we can equip $\cC \coprod \cC$ with a twisted $C_2$-structure via:
 \[
    \begin{tikzcd}
(\cC \coprod \cC)^\op \simeq \cC^\op\coprod \cC^\op \ar{r}{\tau\coprod \tau} & \cC\coprod \cC \ar{r}{\text{swap}} & \cC \coprod \cC,
    \end{tikzcd}
    \]
    and comparing with $\Ind_{e}^{\tau}(\cC)=\cC\coprod \cC^\op$ we obtain an equivalence:
    \[
    \cC\coprod \cC^\op\simeq \cC\coprod \cC
    \]
    of categories with twisted $C_2$-action.
\end{const}

{We note that we wrote these constructions in ordinary categories, but they work analogously in $\infty$-categories. We will be applying them to the nerves of  ordinary categories $\bT_n$ introduced in \autoref{Tndef} below, and using the corresponding adjunctions in $\infty$-categories. }

\subsection{The unstable topological \texorpdfstring{$\Delta \bfG$}{TEXT}-homology of an \texorpdfstring{$\infty$}{TEXT}-category with twisted \texorpdfstring{$G$}{TEXT}-action}\label{sec:unstable}
We first introduce categories $\bT_n$ and the unstable topological Hochschild homology construction first appearing in work of Nikolaus in~\cite{HS19} (see also McCandless~\cite{McC21}). 

\begin{defin}\label{Tndef}
   Define $\bT_n$ to be the free category generated on the cyclic graph with $(n+1)$-vertices labeled $\{0, 1, \dots,  n\}$.
\[
   \begin{tikzpicture}
\def \n {0}
\def \radius {1.3cm}
\def \margin {15} 

\foreach \s in {0,...,\n}
{
  \node[draw, circle] at ({360/(\n+1) * (\s)}:\radius) {$\s$};
  \draw[->, >=latex] ({360/(\n+1) * (\s)+\margin}:\radius) 
    arc ({360/(\n+1) * (\s)+\margin}:{360/(\n+1) * (\s+1)-\margin}:\radius);
}
\end{tikzpicture}
\hspace{1cm}
\begin{tikzpicture}
\def \n {1}
\def \radius {1.3cm}
\def \margin {15} 

\foreach \s in {0,...,\n}
{
  \node[draw, circle] at ({360/(\n+1) * (\s)}:\radius) {$\s$};
  \draw[->, >=latex] ({360/(\n+1) * (\s)+\margin}:\radius) 
    arc ({360/(\n+1) * (\s)+\margin}:{360/(\n+1) * (\s+1)-\margin}:\radius);
}
\end{tikzpicture}
\hspace{1cm}
\begin{tikzpicture}
\def \n {2}
\def \radius {1.3cm}
\def \margin {15} 

\foreach \s in {0,...,\n}
{
  \node[draw, circle] at ({360/(\n+1) * (\s)}:\radius) {$\s$};
  \draw[->, >=latex] ({360/(\n+1) * (\s)+\margin}:\radius) 
    arc ({360/(\n+1) * (\s)+\margin}:{360/(\n+1) * (\s+1)-\margin}:\radius);
}
\end{tikzpicture}
\]
\[
\begin{tikzpicture}
\def \n {3}
\def \radius {1.5cm}
\def \margin {14} 

\foreach \s in {0,...,\n}
{
  \node[draw, circle] at ({360/(\n+1) * (\s)}:\radius) {$\s$};
  \draw[->, >=latex] ({360/(\n+1) * (\s)+\margin}:\radius) 
    arc ({360/(\n+1) * (\s)+\margin}:{360/(\n+1) * (\s+1)-\margin}:\radius);
}
\end{tikzpicture}
\hspace{1cm}
\begin{tikzpicture}
\def \n {4}
\def \radius {1.5cm}
\def \margin {14} 

\foreach \s in {0,...,\n}
{
  \node[draw, circle] at ({360/(\n+1) * (\s)}:\radius) {$\s$};
  \draw[->, >=latex] ({360/(\n+1) * (\s)+\margin}:\radius) 
    arc ({360/(\n+1) * (\s)+\margin}:{360/(\n+1) * (\s+1)-\margin}:\radius);
}
\end{tikzpicture}
\]
These categories assemble into a cocyclic category.
The cofaces $\delta_i\colon [n]\rightarrow [n+1]$ induce functors $\TT_n\rightarrow \TT_{n+1}$ which are defined on objects by $j\mapsto \delta_i(j)$, send the generating morphisms $j\rightarrow j+1$ in $\TT_n$ to the corresponding generating morphism $\delta_i(j)\rightarrow \delta_i(j+1)$ in $\TT_{n+1}$ for all $j\neq i$, and send $i\rightarrow i+1$ in $\TT_n$ to the composition $i\rightarrow i+1\rightarrow i+2$ in $\TT_{n+1}$.
The codegeneracies $\sigma_i\colon [n+1]\rightarrow [n]$ induce functors $\TT_{n+1}\rightarrow \TT_n$ that are defined on objects by $j\mapsto \sigma_i(j)$, send the generating morphisms $j\rightarrow j+1$ in $\TT_{n+1}$ to the corresponding generating morphism $\sigma_i(j)\rightarrow \sigma_i(j+1)$ for all $j\neq i$, and send $i\rightarrow i+1$ in $\TT_{n+1}$ to $\id_i$ in $\TT_n$.
We define an automorphism $r\colon \TT_n\rightarrow \TT_n$ given on objects by $r(j)= j+1 \pmod{n+1}$, that sends the generating morphism $j\rightarrow j+1$ to its corresponding generating morphism $r(j)\rightarrow r(j+1)$. 
This defines a functor 
\begin{align*}
    \TT_{\sbt}\colon \Delta \bfC & \longrightarrow
 \Cat \\
{[n]} & \longmapsto \TT_n.
\end{align*}
\end{defin}

\begin{defin}
    Let $\cC$ be an $\infty$-category. The \textit{unstable topological Hochschild homology of $\cC$} is the space $\uTHH(\cC)$ defined as the homotopy colimit of:
    \[
    \begin{tikzcd}
        \Delta^\op \ar[hook]{r} & \Delta \bfC^\op \ar{r}{\TT_{\sbt}} & \Cat^\op \ar{r}{\Hom_{\Cat_\infty}(-, \cC)} & [4em] \cS.
    \end{tikzcd}
    \]
    Here $\Hom_{\Cat_\infty}(\TT_n, \cC)$ denotes the hom-space of functors $\TT_n\rightarrow \cC$, and is the $n$-th level $\uTHH_n(\cC)$ of the cyclic space.
\end{defin}

\begin{const}
Each $\mathbb{T}_n$ is a twisted $C_2$-category, i.e., there is a functor $\tau\colon \mathbb{T}_n^\op\rightarrow \mathbb{T}_n$ that is its own inverse defined on objects $i\in\{0,\dots n\}$ by:
\[
\tau(i)=\begin{cases}
  -i  \pmod{n+1},  & \text{for $n$ even}\\
  -i-1 \pmod{n+1}, & \text{for $n$ odd}.
\end{cases}
\]
The functor $\tau$ sends a generating morphism  $i\rightarrow i+1$ of $\bT_n$ to the corresponding generating morphism $\tau(i+1)\rightarrow \tau(i)$.
One can verify that given this twisted $C_2$-structure, the cofaces and codegeneracies above are twisted $C_2$-equivariant, as well as the rotations. Therefore, we obtain that $\TT_{\sbt}$ can be considered as a cocyclic twisted $C_2$-category.
\end{const}

\begin{defin}
    Define $\OO_n=\TT_n\coprod \TT_n$. 
    Using \autoref{const: variant wit flip induced twisted C_4 actions}, because $\TT_n$ is a category with twisted $C_2$-action, we get that $\OO_n$ is a twisted $C_2$-category that is equivalent to $\Ind_e^{\tau}(\TT_n)$ of \autoref{const: induced twisted C_4-action} as a category with twisted $C_2$-actions.
    We obtain that $\OO_{\sbt}$ forms a cosimplicial object in categories with twisted $C_2$-actions.
    Moreover, the category $\OO_n$ is endowed with automorphism from $D_{2(n+1)}=\langle r, s \mid r^{n+1}=s^2=1, \, srs^{-1}=r^{-1} \rangle$
    where $s$ acts by swapping the copies of $\TT_n$, while $r$ acts as the usual generating cyclic automorphism $r$ on the first copy of $\TT_n$, and as $r^{-1}$ on the second copy of $\TT_n$. 
    One can verify that these automorphisms are compatible with the cofaces, codegeneracies and the twisted $C_2$-action, so that we obtain a functor:
\begin{align*}
    \OO_{\sbt}\colon \Delta \bfD & \longrightarrow
 \Cat^{\tau} \\
{[n]} & \longmapsto \OO_n \,.
\end{align*}
\end{defin}

\begin{defin}
    Let $\cC$ be an $\infty$-category with twisted $C_2$-action. The \textit{unstable Real topological Hochschild homology of $\cC$} is the space $\uTHR(\cC)$ defined as the homotopy colimit of:
    \[
    \begin{tikzcd}
        \Delta^\op \ar[hook]{r} & \Delta \bfD^\op \ar{r}{\OO_{\sbt}} & (\Cat^{\tau})^\op \ar{r}{\Hom_{\Cat_\infty^{\tau}}(-, \cC)} & [5em] \cS.
    \end{tikzcd}
    \]
    Here $\Hom_{\Cat_\infty^{\tau}}(\OO_n, \cC)$ 
    denotes the hom-space of twisted $C_2$-equivariant functors $\OO_n\rightarrow \cC$, and is the $n$-th level $\uTHR_n(\cC)$ of the dihedral space.
\end{defin}

\begin{const}\label{const: sd_2 for cocyclic}
    Let $\cC$ be any category and $X$ be a cosimplicial object in $\cC$.
    The B\"okstedt--Hsiang--Madsen edgewise subdivision $\sd_2\colon \Delta\rightarrow \Delta$ defines a new cosimplicial object $\sd_2(X)$ where $\sd_2(X)^n=X^{2n+1}$ and given $f$ in $\Delta$ we have $\sd_2(f)=f\coprod f$.
    If $X$ was a cocyclic object, then $\sd_2(X)$ is a cosimplicial object in $\cC^{BC_2}$ with $C_2$-action on $X^{2n+1}$ given by $r^{n+1}$ where $r\in C_{2n+2}$ is the generator of the automorphism acting on $X^{2n+1}$. See~\cite[\S1]{BHM93} for details. 
\end{const}

Throughout the rest of the section, let $q\colon C_4\to C_2$ be the quotient homomorphism.

\begin{defin}
    Define $\PP_n=\TT_{2n+1}\coprod \TT_{2n+1}$. 
    Using \autoref{const: sd_2 for cocyclic}, we identify $\TT_{2n+1}=\sd_2(\TT_{n})$, and is thus a category with $C_2$-action, informally given by rotation of $180^\circ$.
    Using \autoref{const: variant wit flip induced twisted C_4 actions}, because $\TT_n$ is a twisted $C_2$-category, so is $\sd_2(\TT_n)$, we get that $\PP_n$ is a category with twisted $C_4$-action that is equivalent to $\Ind_{C_2}^{q}(\sd_2(\TT_n))$ of \autoref{const: induced twisted C_4-action} as a category with twisted $C_4$-actions.
    We thus formally obtain that $\PP_{\sbt}$ forms a cosimplicial object in $\Cat^{q}$, the category of small categories with twisted $C_4$-action.
    Moreover, the category $\PP_n$ is endowed with automorphisms of $Q_{4(n+1)}=\langle x,y \mid x^{n+1}=y^2, yxy^{-1}=x^{-1} \rangle$ where $y\colon \PP_n\rightarrow \PP_n$ is defined as:
    \[
    \begin{tikzcd}
        \TT_{2n+1}\coprod \TT_{2n+1} \ar{r}{\text{swap}} & \TT_{2n+1}\coprod \TT_{2n+1} \ar{r}{\alpha \coprod \id} & \TT_{2n+1}\coprod \TT_{2n+1}
    \end{tikzcd}
    \]
    where $\alpha\colon \TT_{2n+1}\rightarrow \TT_{2n+1}$ is the generating $C_2$-action on $\sd_2(\TT_n)$, while $x\colon \PP_n\to \PP_n$ is defined by $r\colon \TT_{2n+1}\rightarrow \TT_{2n+1}$ on the first copy, and $r^{-1}\colon \TT_{2n+1}\to \TT_{2n+1}$ on the second copy.
    One can verify that these automorphisms are compatible with the cofaces, codegeneracies, and the twisted $C_4$-actions, so that we obtain a functor:
    \begin{align*}
    \PP_{\sbt}\colon \Delta \bfQ & \longrightarrow
 \Cat^{q} \\
{[n]} & \longmapsto \PP_n.
\end{align*}
\end{defin}

\begin{defin}
    Let $\cC$ be an $\infty$-category with twisted $C_4$-action.
    The \textit{unstable quaternionic topological Hochschild homology of $\cC$} is the space $\uTHQ(\cC)$ defined as the homotopy colimit of:
    \[
    \begin{tikzcd}
        \Delta^\op \ar[hook]{r} & \Delta \bfQ^\op \ar{r}{\PP_{\sbt}} & (\Cat^{q})^\op \ar{r}{\Hom_{\Cat_\infty^{q}}(-, \cC)} & [5em] \cS.
    \end{tikzcd}
    \]
    Here $\Hom_{\Cat_\infty^{q}}(\PP_n, \cC)$ denotes the hom-space of twisted $C_4$-equivariant functors $\PP_n\rightarrow \cC$, and is the $n$-th level $\uTHQ_n(\cC)$ of the quaternionic space.
\end{defin}

If $\cC$ is an ordinary category, we denote by $\cC^{\gp}$ the $\infty$-groupoid completion of its nerve, i.e. $\cC^{\gp}\simeq B\cC$, see \cite[\S 1.2.5]{HTT}.

\begin{prop}\label{prop: classifying space is a circle}
Let $n\ge 0$. There is an equivalence of spaces $B\TT_n\simeq  S^1$. There is an equivalence $B\OO_n\simeq O(2)$ as left $C_2$-spaces.
There is an equivalence $B\PP_n\simeq \Pin(2)$ as left $C_4$-spaces.
\end{prop}

\begin{proof}
    Notice that the category $\TT_0$ is equivalent to the category with one object and morphism set the non-negative integers $\mathbb{N}$. Therefore we obtain that $B\TT_0\simeq B\mathbb{N}\simeq S^1$.
    The codegeneracies define a functor $f\colon \TT_n\rightarrow \TT_0$. Note that $\TT_0$ has a unique object $0$ so the comma category $f\downarrow 0$ has precisely objects $(0,0,z)$ where $z$ is a morphism in $\TT_0$. We observe that this  comma category has a terminal object $(0,0,\id_{0})$, since for any object $(0,0,z)$ in $f\downarrow 0$ there is precisely one commuting diagram 
    \[
    \begin{tikzcd}
 0  \ar[d,"z",swap] \ar[r,"z"] &  0 \ar[d,"\id_{0}"]  \\ 
    0 \ar[r,"\id_0",swap] &  0  
    \end{tikzcd}
    \]
    or in other words 
    \[ \Hom_{f\downarrow 0}((0,0,z),(0,0,\id_{0}))=* \,.
    \]
    Therefore, $B(f\downarrow 0)$ is contractible and thus by Quillen's Theorem A, we get that $f$ induces an equivalence on classifying spaces $B\TT_n\simeq B\TT_0\simeq S^1$.
    Moreover, the $C_2$-action on $\sd_2(\TT_n)=\TT_{2n+1}$ induces precisely the $C_2$-action on $S^1\simeq B\TT_{2n+1}$ given by rotations of $180^\circ$.
    
    Since the right adjoints must commute in the diagram below:
    \[
  \begin{tikzcd}[row sep= large, column sep=large]
        \Cat_\infty \ar[bend left=30, "\Ind_{e}^{\tau}" description]{r} \ar[phantom, "\perp" description]{r} \ar[phantom, "\vperp" description]{d} \ar[bend left, leftarrow]{d}&  \Cat_\infty^{\tau}\ar[bend left=30, hook']{l}\ar[phantom, "\vperp" description]{d} \ar[bend left, leftarrow]{d}\\
        [3em]\cS \ar[bend left=30, "\Ind_e^{C_2}" description]{r} \ar[bend left, leftarrow]{u}{(-)^\gp}\ar[phantom, "\perp" description, xshift=0.5ex]{r}  &   \ar[bend left, leftarrow]{u}{(-)^\gp} \cS^{BC_2}\ar[bend left=30, hook']{l}
    \end{tikzcd}
\]    
we obtain equivalences of $C_2$-spaces:
\[
B\OO_n\simeq (\OO_n)^\gp \simeq (\Ind_e^{\tau}(\TT_n))^\gp\simeq \Ind_e^{C_2}((\TT_n)^\gp)\simeq \Ind_e^{C_2}(S^1)\simeq O(2).
\]
Similarly, if we denote $q\colon C_4\rightarrow C_2$ the quotient homomorphism, the following diagram commutes: 
 \[
  \begin{tikzcd}[row sep= large, column sep=large]
        \Cat_\infty^{BC_2} \ar[bend left=30, "\Ind_{C_2}^{q}" description]{r} \ar[phantom, "\perp" description, xshift=-0.5ex]{r} \ar[phantom, "\vperp" description]{d} \ar[bend left, leftarrow]{d}&  \Cat_\infty^{q}\ar[bend left=30, "\Res_{C_2}^q" description]{l}\ar[phantom, "\vperp" description]{d} \ar[bend left, leftarrow]{d}\\
        [3em]\cS^{BC_2} \ar[bend left=30, "\Ind_{C_2}^{C_4}" description]{r} \ar[bend left, leftarrow]{u}{(-)^\gp}\ar[phantom, "\perp" description, xshift=-0.5ex]{r}  &   \ar[bend left, leftarrow]{u}{(-)^\gp} \cS^{BC_4}.\ar[bend left=30, "\Res_{C_2}^{C_4}" description]{l}
    \end{tikzcd}
\]   
Therefore we obtain equivalences of $C_4$-spaces: 
\[
B\PP_n\simeq (\PP_n)^\gp \simeq (\Ind_{C_2}^{q}(\TT_{2n+1}))^\gp\simeq \Ind_{C_2}^{C_4}((\TT_{2n+1})^\gp)\simeq \Ind_{C_2}^{C_4}(S^1)\simeq \Pin(2). \qedhere
\]
\end{proof}

\begin{prop}
If $\cC$ is an $\infty$-category with twisted $C_2$-action, then there is an equivalence of spaces \[\Hom_{\Cat_{\infty}^{\tau}}(\OO_n, \cC)\simeq \Hom_{\Cat_\infty}(\TT_n, \cC).\]
Let $q\colon C_4\to C_2$ be the quotient homomorphism.
If $\cC$ is an $\infty$-category with twisted $C_4$-action, then there is an equivalence of spaces 
\[\Hom_{\Cat_\infty^{q}}(\PP_n, \cC)\simeq \Hom_{\Cat_\infty^{BC_2}}(\TT_{2n+1}, \Res_{C_2}^q(\cC)).\]
Here $\Res_{C_2}^q$ was defined in \autoref{const: induced twisted C_4-action}.
\end{prop}

\begin{proof}
    The proof follows directly from the definitions of $\OO_n$ and $\PP_n$ as well as the adjunctions established in \autoref{const: induced twisted C_4-action}.
\end{proof}

\begin{notation}
Let $\Delta \bfG$  be either  $\Delta \bfC$, $\Delta \bfD$ or $\Delta \bfQ$.
Let $\cC$ be an $\infty$-category with twisted $G_0$-action.
We write:
\[
\uTHG(\cC)=\begin{cases}
    \uTHH(\cC) & \text{if }\Delta\bfG=\Delta\bfC,\\
    \uTHR(\cC) & \text{if }\Delta \bfG=\Delta \bfD,\\
    \uTHQ(\cC)& \text{if }\Delta \bfG=\Delta \bfQ.
\end{cases}
\]
We shall also write $G_n=\langle x, y \rangle $ the generating elements as in \autoref{self dual structure for cyclic, dihedral and quaternionic}. 
If $G_n=C_{n+1}$, then $y=1$. If $G_n=D_{2(n+1)}$, then $y^2=1$. If $G_n=Q_{4(n+1)}$, then $y^2=x^{n+1}$ and $y^4=1$.
\end{notation}

\begin{prop}\label{prop: uTHG_n(C) description}
Let $n \geq 0$.
Let $\Delta \bfG$  be either  $\Delta \bfC$, $\Delta \bfD$ or $\Delta \bfQ$. 
Let $\cC$ be an $\infty$-category with twisted $G_0$-action. In the case of $\Delta\bfD$ and $\Delta\bfQ$, we denote by $t  \colon \cC^\op\rightarrow \cC$ the twisted $G_0$-action. In the case of $\Delta\bfC$, we denote by $t$ the identity functor on $\cC$.
We obtain an equivalence of spaces with left $G_n$-action:
\[
\uTHG_n(\cC)\simeq \colim_{(C_0, \cdots, C_n)\in \cC^{\simeq}} \prod_{i=0}^n \Hom_\cC(C_i, C_{i+1})
\]
where $C_{n+1}=t^2C_0$. 
Here $G_n=\langle  x, y\rangle $ acts on the right hand side as follows. The automorphism $x$ is defined as \[(f_0, \cdots, f_n) \mapsto 
(t^2f_n, f_0, \cdots, f_{n-1}) 
\]
where $f_i\colon C_i\to C_{i+1}$ in $\cC$, for $0\leq i \leq n$.
The automorphism  $y$ in the case of $\Delta \bfD$ or $\Delta\bfQ$ is defined as
\[
(f_0, \cdots, f_n) \mapsto (tf_0, t^3f_n, \cdots, t^3f_{1}).
\]
\end{prop}

\begin{proof}
We just prove the quaternionic case since the other cases are similar. Notice that we have an equivalence of spaces:
\[
\Hom_{\Cat_\infty}((\TT_n)^\simeq, \cC)\simeq (\cC^{\simeq})^{\times n+1}\simeq \Hom_{\Cat_\infty^{BC_2}}((\TT_{2n+1})^\simeq, \cC).
\]
Denote by $X$ the space $(\cC^{\simeq})^{\times n+1}$. It is endowed with a left $Q_{4(n+1)}$-action.
Moreover, the inclusions $(\PP_n)^\simeq \subseteq\PP_n$ are twisted $C_4$-equivariant and thus induce a map of spaces:
\[
\uTHG_n(\cC)\longrightarrow X,
\]
which is $Q_{4(n+1)}$-equivariant.

By straightening-unstraightening \cite[2.2.1.2]{HTT}, forgetting the equivariance, the above map is classified by a functor:
    \begin{align*}
        F\colon X &  \longrightarrow \cS\\
        (C_0,\ldots,  C_n)  & \longmapsto \prod_{i=0}^n \Hom_\cC(C_i, C_{i+1})
    \end{align*}
    where $C_{n+1}=t^2C_0$.
    Therefore $\uTHG_n(\cC)\simeq \colim_X F$.
    The $Q_{4(n+1)}$-action is induced by the $\Delta\bfG$ structure on $\uTHG_n(\cC)$ and we can verify  that this identification is indeed $Q_{4(n+1)}$-equivariant.
\end{proof}

Recall from \autoref{ex: twisted G-categories with one object} that, given a parity $\lambda_0\colon G_0\to C_2$, if $M$ is an $\mathbb{E}_1$-space with twisted $G_0$-action, we denote by $\bB^{\lambda_0}M$ the $\infty$-category with twisted $G_0$-action with a single object, and hom-space given by $M$.

\begin{cor}\label{cor: uTHG(BM) is THG in spaces of M}
Let $\Delta \bfG$  be either  $\Delta \bfC$, $\Delta \bfD$ or $\Delta \bfQ$. 
Let $\lambda_0\colon G_0\to C_2$ be its canonical parity.
Let $M$ be an $\mathbb{E}_1$-space with a twisted $G_0$-action. 
There is an equivalence $\uTHG(\bB^{\lambda_0}M)\simeq \HG(M/\cS)$ of spaces with left $|\mathbf{G}_{\sbt}|$-action. 
\end{cor}

\begin{proof}
As $\bB^{\lambda_0}M$ has only one object, and hom-space $M$, we obtain from \autoref{prop: uTHG_n(C) description} an equivalence of spaces with $G_n$-actions:
\[
\uTHG_n(\bB
^{\lambda_0}M)\simeq M^{\times n+1}.
\]
Here the $G_n$-action on the right is precisely the one induced by the $\Delta \bfG$-bar construction in $\cS$ from \autoref{ex: contravariant Delta G bar construction}. 
In fact the above equivalence induces an equivalence of $\Delta \bfG^\op$-object in spaces and the result follows.
\end{proof}

\begin{const}
Let $X$ be a space. We denote the free loop space by 
\[\mathcal{L}X\coloneqq\Map(S^1, X)\,.\] 
It is endowed with an $S^1$-action and in fact can be viewed as the norm of $X$ from $e$ to $S^1$ in unpointed spaces.

\noindent
If $X$ is endowed with a $C_2$-action, notice that we have equivalences of spaces
\[
\Map^{C_2}(O(2), X)\simeq \Map^{C_2}(\Ind_e^{C_2}(S^1), X)\simeq \Map(S^1, X)=\mathcal{L}X.
\]
Here $\Map^{C_2}(O(2), X)$ denotes the space of $C_2$-equivariant maps $O(2)\rightarrow X$. If $X$ is a space with left $C_2$-action, we can view $\mathcal{L}X$ as a space with left $O(2)$-action, and in fact can be regarded as the norm of $X$ from $C_2$ to $O(2)$ in unpointed spaces. Note that restricting the $O(2)$-action to a $C_2$-action identifies the free loop space with the signed free loop space $\mathcal{L}^{\sigma}X$, of $X$. 

\noindent
If $X$ has a $C_4$-action, notice we have equivalences of spaces: 
\[
\Map^{C_4}(\Pin(2), X)\simeq \Map^{C_4}(\Ind_{C_2}^{C_4}(S^1), X)\simeq  \Map^{C_2}(S^1, X)\simeq \mathcal{L}^{\tau}X. 
\]
Here we omit the restriction functor $\cS^{BC_4}\to \cS^{BC_2}$ from the notation. 
Here $\mathcal{L}^{\tau}X$ denotes the twisted free loop space \cite[Definition~A.4]{CP19} defined as:
\[\mathcal{L}^{\tau}X\simeq \{ \gamma\colon[0,1]\rightarrow X \mid t^2(\gamma(0))=\gamma(1)\}\subseteq \Map([0,1], X)
\]
where $t\colon X\rightarrow X$ is the $C_4$-action (note that we prefer to write $\cL^{\tau}$ instead of $\cL_{t^2}$). By our analysis above, we see that $\mathcal{L}^{\tau}X$ is endowed with a left $\Pin(2)$-action, and in fact can be regarded as the norm of $X$ from $C_4$ to $\Pin(2)$ in unpointed spaces.
\end{const}

\begin{cor}\label{cor: uTHG of space is free loop space}
Let $\Delta \bfG$  be either  $\Delta \bfC$, $\Delta \bfD$ or $\Delta \bfQ$. 
    Let $X$ be a space with a left $G_0$-action.
    Then we obtain an equivalence of spaces with left $|\mathbf{G}_{\sbt}|$-action:
    \[
    \uTHG(X)\simeq \begin{cases}
        \mathcal{L}X & \text{if }\Delta\bfG=\Delta\bfC, \Delta\bfD,\\
        \mathcal{L}^{\tau}X & \text{if }\Delta\bfG=\Delta\bfQ.
    \end{cases}
    \]
\end{cor}

\begin{proof}
We prove only the case $\Delta\bfG=\Delta\bfQ$ as the others are similar.
Since $X$ can be viewed as an $\infty$-groupoid with $C_4$-action, we get that twisted $C_4$-equivariant functors $\PP_n\to X$ corresponds to $C_4$-equivariant maps from the $\infty$-groupoid completion of $\PP_n$ to $X$. In fact, by \autoref{prop: classifying space is a circle}, we have:
\begin{align*}
    \uTHQ_n(X) & \simeq \Map_{\Cat_\infty^{q}}(\PP_n, X) \\
    & \simeq \Map^{C_4}(B\PP_n, X) \\
    &\simeq \Map^{C_4}(\Pin(2), X)\\
    & \simeq \mathcal{L}^{\tau}X.
\end{align*}
As a simplicial space, $\uTHQ(X)$ is constant at $\mathcal{L}^{\tau}X$, but not constant as a quaternionic space. In fact, the resulting left $\Pin(2)$-action is precisely the one coinduced through right multiplication on $\Pin(2)$ in $\Map^{C_4}(\Pin(2), X)$. 
The result follows.
\end{proof}

\begin{const}[{\cite[2.5, 3.1]{SW03}}]\label{loop-space-twisted-G-action}
Let $G$ be a finite group with parity $\varphi\colon G\to C_2$.
Let $X$ be a pointed space with left $G$-action.
The loop space $\Omega X$ is an $\mathbb{E}_1$-space with twisted $G$-action given as follows. If $g\in G$ is even, given a loop $\lambda\colon [0,1]\rightarrow X$, define $g\lambda\in \Omega X$ as:
\begin{align*}
  {[0,1]} & \longrightarrow X\\
  t & \longmapsto g\lambda(t), 
\end{align*}
while if $g\in G$ is odd, the loop $g\lambda$ is defined as:
\begin{align*}
    {[0,1]} & \longrightarrow X\\
    t & \longmapsto g\lambda(1-t).
\end{align*}
We denote $\Omega^{\varphi} X$ the loop space $\Omega X$ with the above twisted $G$-action.
Classically, group-like $\mathbb{E}_1$-spaces are equivalent to connected pointed spaces via the loop space construction. Salvatore--Wahl show that using the construction above, the loop space lifts to a functor of $\infty$-categories:
\[
\Omega^{\varphi}\colon (\cS^*_0)^{BG} \longrightarrow \Alg^\varphi(\cS_{\mathrm{gl}}),
\]
from pointed connected spaces with $G$-action to $\mathbb{E}_1$-spaces that are group-like and endowed with twisted $G$-action, that is an equivalence on the homotopy categories.
The usual left adjoint of $\Omega$ is the two-sided bar construction seen as a delooping, and can be regarded as a functor $B^{\varphi}\colon \Alg^\varphi(\cS_{\mathrm{gl}}) \to (\cS^*_0)^{BG}$.
In particular, for any pointed connected space $X$ with $G$-action, we obtain an equivalence $B^{\varphi}\Omega^{\varphi} X\simeq X$ of spaces with $G$-actions. 
\end{const}


\begin{prop}\label{prop: THG of group-like monoids}
    Let $\Delta \bfG$  be either  $\Delta \bfC$, $\Delta \bfD$ or $\Delta \bfQ$.
    Denote $\lambda_0\colon G\rightarrow C_2$ its canonical parity.
Let $M$ be a group-like $\mathbb{E}_1$-space with a twisted $G_0$-action.
Then we obtain an equivalence of spaces with
left $|\mathbf{G}_{\sbt}|$-action:
\[
\HG(M/\cS)\simeq \uTHG(B^{\lambda_0}M).
\]
\end{prop}

\begin{proof}
    By \autoref{cor: uTHG(BM) is THG in spaces of M}, we have the equivalence:
    $\HG(M/\cS)\simeq \uTHG(\bB^{\lambda_0}M)$.
    As $M$ is group-like, then $\bB M$ is an $\infty$-groupoid and is thus equivalent to the space $BM$. This equivalence is compatible with the induced (non-twisted) $G$-actions, and thus we get $\bB^{\lambda_0}M\simeq B^{\lambda_0}M$.
\end{proof}

\begin{cor}\label{cor: main computation but for spaces THG of loop space}
    Let $\Delta \bfG$  be either  $\Delta \bfC$, $\Delta \bfD$ or $\Delta \bfQ$.
    Denote $\lambda_0\colon G_0\rightarrow C_2$ its canonical parity.
    Let $X$ be a connected pointed space with left $G_0$-action.
    There is an equivalence of spaces with left $|\mathbf{G}_{\sbt}|$-action:
    \[
    \HG(\Omega^{\lambda_0}X/\cS)\simeq \begin{cases}
        \mathcal{L}X & \text{if }\Delta\bfG=\Delta\bfC, \Delta\bfD,\\
        \mathcal{L}^{\tau}X & \text{if }\Delta\bfG=\Delta\bfQ.
    \end{cases}
    \]
\end{cor}

\begin{proof}
    Combine \autoref{cor: uTHG of space is free loop space} with \autoref{prop: THG of group-like monoids} as we have $B^{\lambda_0}\Omega^{\lambda_0} X\simeq X$.
\end{proof}

\subsection{Topological \texorpdfstring{$\Delta \bfG$}{TEXT}-homology of monoid rings with twisted \texorpdfstring{$G$}{TEXT}-action}
Recall that the tensoring of spectra over spaces
\begin{align*}
    \cS\times \Sp & \longrightarrow \Sp\\
    (X, E) & \longmapsto E\wedge X_+
\end{align*}
 is strong symmetric monoidal in the sense that for any spaces $X,Y$ and spectra $E, F$ we have the natural equivalences of spectra:
\[
(E\wedge F) \wedge (X\times Y)_+\simeq (E\wedge X_+)\wedge (F\wedge Y_+),\quad \quad \mathbb{S}\wedge S^0\simeq \mathbb{S}.
\]
Thus for any $\infty$-operad $\mathscr{O}$, we obtain a functor:
\begin{align*}
\Alg_\mathscr{O}(\cS)\times \Alg_\mathscr{O}(\Sp)\simeq \Alg_\mathscr{O}(\cS\times \Sp)& \longrightarrow \Alg_\mathscr{O}(\Sp)  \,. 
\end{align*}
\begin{defin}
Let $(G, \varphi)$ be a group with parity.
The tensoring of spectra over spaces defines a functor 
\begin{align*}
\Alg^\varphi(\cS)\times \Alg^\varphi(\Sp) & \longrightarrow \Alg^\varphi(\Sp) \\   
(M, R) & \longmapsto R[M]\coloneqq R\wedge M_+.
\end{align*}
Informally, the twisted $G$-action is given as follows. If $g\in G$ is even, then $g$ acts on $R[M]$ via:
\[
\begin{tikzcd}
    R\wedge M_+ \ar{r}{g\wedge g_+} & R\wedge M_+,
\end{tikzcd}
\]
and if $g\in G$ is odd, then $g$ acts on $R[M]$ as:
\[
\begin{tikzcd}
    (R\wedge M_+)^\op\simeq R^\op \wedge (M^\op)_+\ar{r}{g\wedge g_+} & R\wedge M_+ \,.
\end{tikzcd}
\]
When $R=\mathbb{S}$, we write $\Sigma^\infty_+M$ for the monoid with twisted $G$-action $\mathbb{S}[M]$. 
\end{defin}

Let $(G_0,\lambda_0)$ be the canonical group with parity associated to a crossed simplicial group $\Delta \bfG$, let $M$ be an $\mathbb{E}_1$-space with twisted $G_0$-action, and let $R$ be a ring spectrum with twisted $G_0$-action. Then let $R[M]$ be the corresponding monoid ring spectrum with twisted $G_0$-action. As a sample computation, we describe $\THG(R[M])$ in the case when $\Delta \bfG$ is a self-dual crossed simplicial group.  

\begin{thm}\label{assembly theorem}
Let $(\Delta \bfG,\sdual)$ be a self-dual crossed simplicial group with associated group with parity $(G_0,\lambda_0)$. 
Let $R$ be a ring spectrum with twisted $G_0$-action and let $M$ be an $\mathbb{E}_1$-space with twisted $G_0$-action. 
Then the assembly map 
\[  
\THG(R)\wedge \Sigma^\infty_+ \HG(M/\cS)\longrightarrow \THG(R[M])
\]
is an equivalence of spectra with $|\mathbf{G}_{\sbt}|$-action. 
\end{thm}

\begin{proof}
We consider the functor 
\[
    \iota^*B^{\bfG}_{\bullet}(R)\wedge \Sigma^{\infty}_+(\iota^*B^{\bfG}_{\bullet}(M))  \colon \Delta^{\op}\times \Delta^{\op}\longrightarrow \Sp 
\] 
defined on objects by 
\[ 
    ([p],[q]) \longrightarrow \iota^*B^{\bfG}_p(R)\wedge \Sigma^{\infty}_+\iota^*B^{\bfG}_q(M) \,.
\]
Then there are equivalences
\begin{align*}
    \THG(R[M]) & \simeq \colim_{\Delta^\op}^{\Sp} \big( \iota^* B^{\bfG}_\bullet (R\otimes M_+) \big)\\
    & \simeq  \colim_{\Delta^\op\times \Delta^\op}^{\Sp} \big( \left( \iota^*B^{\bfG}_\bullet (R)\right) \wedge \left(\Sigma^{\infty}_+\iota^*B^\bfG_\bullet(M)\right)\big)\\
    & \simeq |\iota^*B^{\bfG}_{\bullet}(R)|\wedge \Sigma^{\infty}_+| \iota^*B^\bfG_\bullet(M)|\\
    & \simeq \THG(R) \wedge \Sigma^\infty_+ \HG(M/\cS) \,.\qedhere
\end{align*}
\end{proof}

Specializing to $R=\mathbb{S}$, we have the following corollary.

\begin{cor}\label{cor: group ring computation for csg with duality}
Let $\Delta \bfG$ be a self-dual crossed simplicial group with associated group with parity $(G_0,\lambda_0)$ and let $M$ be an $\mathbb{E}_1$-space with twisted $G_0$-monoid, then the assembly map 
\[
\Sigma^\infty_+ \HG(M/\cS)  \overset{\simeq}{\longrightarrow} \THG(\Sigma_+^{\infty}M)\,.
\] 
is an equivalence of spectra with left $|\mathbf{G}_{\sbt}|$-action.
\end{cor}

\begin{rem}
In the case where $\Delta \bfG=\Delta\mathbf{C}$ is the cyclic category, \autoref{cor: group ring computation for csg with duality} was first proven in \cite[Proposition 4.25]{BHM93}. 
In the case where $\Delta \bfG=\Delta \bfD$ is the dihedral category, \autoref{cor: group ring computation for csg with duality} was proven  as a consequence of \cite[Theorem 4.1]{Hog16} and \autoref{assembly theorem} was first proven in \cite[Proposition~5.12]{DMPR21}. 
Our result extends these earlier results to more general self-dual crossed simplicial groups. 
\end{rem}

\begin{thm}\label{thm: main computation}
 Let $\Delta \bfG$  be either  $\Delta \bfC$, $\Delta \bfD$ or $\Delta \bfQ$.
    Denote $\lambda_0\colon G\rightarrow C_2$ its canonical parity.
    Let $X$ be a connected pointed space with left $G_0$-action.
    Then there is an equivalence of spectra with left $|\mathbf{G}_{\sbt}|$-action: 
\[ \THG(\Sigma^\infty_+\Omega^{\lambda_0}X)\simeq \begin{cases}
    \Sigma^\infty_+\mathcal{L}X & \text{if }\Delta\bfG=\Delta\bfC, \Delta\bfD \,,\\
    \Sigma^\infty_+\mathcal{L}^{\tau}X & \text{if }\Delta\bfG=\Delta\bfQ \,.
    \end{cases}
    \]
\end{thm}

\begin{proof}
 This is a combination of \autoref{cor: group ring computation for csg with duality} with \autoref{cor: main computation but for spaces THG of loop space}.
\end{proof}

\begin{rem}
The result is  classical for $\Delta \bfC$ and dates back to Goodwillie \cite{GoodwillieCyclic}.
For $\Delta \bfD$ the result above appears in \cite{DMPR21} and~\cite{HHKWZ24} as $C_2$-spaces and in \cite{Hog16} and~\cite{DMP24} as genuine $O(2)$-spectra.
\end{rem}
The following corollary is immediate. 

\begin{cor}\label{TQ+}
Let $X$ be a connected pointed space with left $C_4$-action.
There are equivalences 
\begin{align*}
\mathrm{TQ}^{+}(\Sigma^\infty_+\Omega^{\lambda_{0}}X)&\simeq\left ( \Sigma^{\infty}_+\mathcal{L}^{\tau}X \right )_{h\Pin(2)} \\ 
\mathrm{TQ}^{-}(\Sigma^\infty_+\Omega^{\lambda_{0}}X)&\simeq (\Sigma^{\infty}_+\mathcal{L}^{\tau}X)^{h\Pin(2)}\\ 
\mathrm{TQ}^{\textup{per}}(\Sigma^\infty_+\Omega^{\lambda_{0}}X)& \simeq(\Sigma^{\infty}_+\mathcal{L}^{\tau}X)^{t\Pin(2)} \,. 
\end{align*}
\end{cor}

\section{Topological twisted symmetric homology of loop spaces}\label{twisted symmetric}
Our next goal is to provide computations of topological twisted symmetric homology of spherical group rings with twisted $G$-action. We mainly present these computations to illustrate that topological twisted symmetric homology is an interesting invariant. Recall the notation $\mathrm{T}\varphi$, $\TS$, and $\TO$ from  \autoref{twisted symmetric example}, \autoref{symmetric}, and  \autoref{hyperoctahedral}

The main result of this section is a homotopical and twisted symmetric analogue of an unpublished theorem of Fiedorowicz, appearing in~\cite{Aul10}. 
\begin{thm}\label{thm: topological positive hyperoctahedral homology}
Let $\Delta \bfGS$ be the crossed simplicial group associated to a group with parity $\varphi \colon G\to C_2$. Let $M$ be a group-like $\mathbb{E}_1$-space with a twisted $G$-action, then 
\[ 
\mathrm{T}\varphi( \Sigma^\infty_+ M)\simeq   \Sigma_+^{\infty}\left (\Omega Q B^{\varphi} M\right )_{hG} \,.
\]
Let $\varphi\colon G\to C_2$ be a group homomorphism.
If $X$ is a connected space with $G$-action, then 
\[
\mathrm{T}\varphi(\Sigma^\infty_+\Omega^{\varphi} X)\simeq (\Sigma^\infty_+ \Omega QX)_{hG}\,.
\] 
\end{thm}
This result has the following immediate corollaries. 
\begin{cor}\label{cor: topological positive symmetric homology}
Let $G$ be a group-like $\bE_1$-space, then 
\[\TS(\Sigma^\infty_+ G)\simeq \Sigma^\infty_+ \Omega QBG.
\] 
Let $X$ be a connected space, then 
\[
\TS(\Sigma^\infty_+\Omega X)\simeq \Sigma^\infty_+ \Omega QX\,.
\] 
\end{cor}

\begin{cor}\label{cor: topological positive hyperoctahedral homology}
Let $G$ be a group-like $\bE_1$-space with anti-involution, then 
\[
\TO(\Sigma^\infty_+ G)\simeq \Sigma_+^{\infty}\left ( \Omega Q B^{\sigma}G\right )_{hC_2}.
\]
Let $X$ be a connected space with $C_2$-action, then
\[
\TO(\Sigma^\infty_+\Omega^{\sigma} X)\simeq \Sigma_+^{\infty} (\Omega QX)_{hC_2}.
\]
\end{cor}

\begin{rem2}
\autoref{cor: topological positive symmetric homology} specializes to~\cite[Theorem 2]{Aul10} in the setting of the derived category $\cD(k)$ for a commutative ring $k$
(cf.~\cite[Theorem 1 (i)]{Fie}). \autoref{cor: topological positive hyperoctahedral homology} specializes to \cite[Theorem 8.8]{Gra22} in the setting of the derived category $\cD(k)$ of a commutative ring $k$ (cf.~\cite[Theorem 1 (iii)]{Fie}). In other words,  \autoref{thm: topological positive hyperoctahedral homology} generalizes \cite[Theorem 8.8]{Gra22}, \cite[Theorem 2]{Aul10}, and \cite[Theorem 1 (iii)]{Fie} in two ways: it generalizes these results from the setting of  $\cD(k)$ to spectra and it also generalizes them to $\Delta \bfGS$ for any group with parity $\varphi\colon G\to C_2$. However, we do not claim significant originality since all of the key ideas in the proofs were already present in the work of \cite{Fie}, \cite{Aul10}, and \cite{Gra22}.  
\end{rem2}

\begin{rem2}
We expect that a similar result to \autoref{thm: topological positive hyperoctahedral homology} holds in the case of the braid crossed simplicial group, but for brevity we leave this to future work. 
\end{rem2}

The goal of the remainder of this section is to prove \autoref{thm: topological positive hyperoctahedral homology}. 

\begin{lem}
Let $\varphi\colon G\to C_2$ be a group with parity. Then the forgetful functor 
\[ 
\Alg^{\varphi}(\cS) \to \cS^{BG}
\]
has a left adjoint, which we denote $J^{\varphi}$ (or $J$ if $\varphi$ is trivial). We can identify 
\[J^{\varphi}(X) \simeq \coprod_{n\ge 0}(\Sigma_{n}\ltimes G)\times_{\Sigma_{n}}(X^{\times n})\,.
\]
The forgetful functor 
$\mathrm{CAlg}(\cS^{BG})\to \cS^{BG}$
has a left adjoint that we denote by $\Sym$, which can be identified as 
\[\Sym(X)\simeq  \coprod_{n\ge 0}(X^{\times n})_{h\Sigma_{n}}\,.\]
In each case, the free forgetful functor is monadic and we can therefore identify 
\[ 
\Alg^{\varphi}(\cS)\simeq \mathrm{Mod}_{J^{\varphi}}(\cS)
\]
and 
\[ 
\mathrm{CAlg}(\cS^{BG})\simeq \mathrm{Mod}_{\mathrm{Sym}}(\cS^{BG}).
\]
\end{lem}
\begin{proof}
This follows from~\cite[3.1.3.13]{HA} as in~\cite[3.1.3.14]{HA}. The fact that the free forgetful adjunction is monadic, in each case, follows from \cite[4.7.3.5]{HA}, since the forgetful functor is conservative and preserves geometric realizations of simplicial objects.
\end{proof}

Recall from \autoref{bar-construction}, that given a group with parity $\varphi \colon G\to C_2$, and an algebra $R$ with twisted $G$-action in a cocomplete symmetric monoidal $\infty$-category $\cC$, we denoted by $B_{\varphi \wr \Sigma}^\bullet R\colon \Delta\bfGS\to \cC$ the (covariant) twisted symmetric bar construction of $R$. 
We shall also write $B_{\varphi\wr\Sigma  ,+}^{\bullet}R$ for the functor 
\[ 
\begin{tikzcd}
B_{\varphi\wr \Sigma,+}^{\bullet}R\colon \Delta \varphi \wr \Sigma_+ \simeq\Env(\Assoc^\varphi)\arrow{r}{\Env(R)}  & \Env(\cC) \arrow{r}{\otimes }& \cC
\end{tikzcd}
\]
satisfying $B_{\varphi\wr \Sigma,+}^{\bullet}R\circ i =B_{\varphi \wr \Sigma}^\bullet R$ where $i\colon \Delta \bfGS\to \Delta \bfGS_+$ is the canonical inclusion. 

\begin{lem}\label{pointed vs unpointed category}
The canonical map 
\[
  \hocolim_{\Delta\bfGS  }B_{\varphi\wr \Sigma}^{\bullet}R \overset{\simeq}{\longrightarrow} \hocolim_{\Delta\bfGS_+ }B_{\varphi\wr \Sigma,+}^{\bullet}R
\]
induced by $i\colon \Delta \bfGS\to \Delta \bfGS_+$ is an equivalence. 
\end{lem}

\begin{proof}
Since $\Delta \bfGS_+$ is just $\Delta \bfGS$ with a new initial object added and $\Delta \bfGS$ is connected, the homotopy colimit does not change.
\end{proof}

\begin{lem}\label{lem:EG}
Let $(G,\varphi)$ be a group with parity. There is an equivalence of spaces with right $G\wr \Sigma_{n+1} $-action
\[ 
(\Delta \varphi \wr \Sigma_+)_{[n]\slash} \simeq EG\wr \Sigma_{n+1}
\]
\end{lem}
\begin{proof}
The category $(\Delta \varphi \wr \Sigma_+)_{[n]\slash}$ has objects $\mathrm{Hom}_{\Delta_+}([n],[m])\times G\wr \Sigma_{n+1}$ so it has a free $G\wr \Sigma_{n+1} $-action. Moreover, (the nerve of) $(\Delta \varphi \wr \Sigma_+)_{[n]\slash}$ is contractible by \autoref{contractible} so it is a model for $EG\wr \Sigma_{n+1}$ as a right $G\wr \Sigma_{n+1}$-space. 
\end{proof}

\begin{lem}\label{lem:orbits-formula}
Suppose $X$ is a space with $G$-action and $n\ge 0$. Then $X^{\times n+1}$ has a canonical $G\wr \Sigma_{n+1}$-action and there is an equivalence of spaces
\[ 
(X^{\times n+1})_{hG\wr \Sigma_{n+1}}\simeq ( ((X_{hG})^{\times n+1})_{h\Sigma_{n+1}})\,.
\]
\end{lem}

\begin{proof}
The fact that $X^{\times n+1}$ has a canonical $G\wr \Sigma_{n+1}$-action is clear. The rest of the claim follows from the equivalences
\begin{align*}
    (X^{\times n+1})_{hG\wr \Sigma_{n+1} } &\coloneqq \colim_{BG\wr \Sigma_{n+1}} (X^{\times n+1})\\
    & \simeq  \colim_{B\Sigma_{n+1}}\colim_{BG^{\times n+1}} (X^{\times n+1})\\
    & \simeq  \colim_{B\Sigma_{n+1}}(\colim_{BG} X)^{\times n+1}\\
        & \eqqcolon ((X_{hG})^{\times n+1})_{B\Sigma_{n+1}} \,. 
\end{align*}
Here, the first identification follows by definition. 
For the second equivalence, note that the  functor $X^{\times n+1}\colon BG\wr \Sigma_{n+1} \to \cS$ corresponds to map of spaces $(X^{\times n+1})_{hG\wr \Sigma_{n+1}}\to BG\wr \Sigma_{n+1}$ by straightening--unstraightening. We can further compose with the map of spaces $BG\wr \Sigma_{n+1}\to B\Sigma_{n+1}$ and this will produce a functor $B\Sigma_{n+1}\to \cS$ by straightening--unstraightening. Therefore, the total space $(X^{\times n+1})_{hG\wr \Sigma_{n+1}}$ can be written as $Y_{h\Sigma_{n+1}}$ for some space $Y$. We identify $Y$ as $\colim_{BG^{\times n+1}}X^{n+1}$ by considering the commutative diagram 
\[
\begin{tikzcd}
G^{\times n+1}\ar[r] \ar[d] &  X^{\times n+1} \ar[d] \ar[r] &   \ar[d] Y \\ 
* \ar[r] \ar[d] &  (X^{\times n+1})_{hG\wr \Sigma_{n+1} } \ar[d] \ar[r] &   \ar[d]  Y_{h\Sigma_{n+1}}\\ 
BG^{\times n}\ar[r] &  BG\wr \Sigma_{n+1}   \ar[r] &  B\Sigma_{n+1}
\end{tikzcd}
\]
where each horizontal and vertical column is a fiber sequence. The second to last equivalence follows from the fact that $G^{\times n+1}$ acts on $X^{n+1}$ coordinate-wise. The last identification holds by definition.
\end{proof}

\begin{prop}\label{prop: formula}
Let $\varphi \colon G\to C_2$ be a group with parity and let $X$ be a space with $G$-action.  There is a natural equivalence of spaces
\[
\mathrm{H}\varphi(J^\varphi(X)/\cS)\simeq \mathrm{Sym}(X)_{hG}.
\]
\end{prop}

\begin{proof}
By \autoref{twisted symmetric example} and \autoref{pointed vs unpointed category}, there is a natural equivalence
\[
\mathrm{H}\varphi(J^\varphi(X)/\cS)\simeq  \hocolim_{\Delta\bfGS_+} B_{\varphi\wr \Sigma,+}^{\bullet}J^{\varphi}(X) \,.
\]
Here the equivalence relation is the usual one produced when one writes a homotopy colimit as a geometric realization using the Bousfield--Kan formula for the homotopy colimit. 
By definition of the colimit and \autoref{lem:EG}, there are natural equivalences  
\begin{align*}
\hocolim_{\Delta\bfGS_+} B_{\varphi\wr \Sigma,+}^{\bullet}J^{\varphi}(X) & \coloneqq \left ( \coprod_{n\ge -1} ([n]\downarrow \Delta \varphi\wr \Sigma )\times J^{\varphi}(X)^{\times n+1} \right )/\sim \\
& \simeq \left ( \coprod_{n\ge -1} EG\wr \Sigma_{n+1} \times J^{\varphi}(X)^{\times n+1} \right )/\sim  \,.
\end{align*}
By \autoref{lem:orbits-formula} and since every map $\phi$ in $\Delta \varphi \wr \Sigma_+$ can be factored uniquely as a composite of $\gamma \circ g$ where $g\in G\wr \Sigma_{n+1}$ and $\gamma$ is a morphism in $\Delta$ respectively, there are natural equivalences 
\begin{align*}
\left ( \coprod_{n\ge -1} EG\wr \Sigma_{n+1} \times J^{\varphi}(X)^{\times n+1} \right )/\sim  & \simeq  \left ( \coprod_{n\ge -1} EG\wr \Sigma_{n+1} \times_{G\wr \Sigma_{n+1} } J^{\varphi}(X)^{\times n+1} \right ) /\sim  \\ 
& \simeq  \left ( \coprod_{n\ge -1} ( (J^{\varphi}(X)_{hG})^{\times n+1})_{h\Sigma_{n+1}}\right )  /\sim  
\end{align*}
where now we abuse notation and write $\sim$ for the residual equivalence relation remaining after taking into account the automorphisms. Finally, by commuting the $G$ homotopy orbits with the left adjoint $J^{\varphi}$ and applying \cite[Lemma~33 and~36]{Aul10} respectively, there are natural equivalences 
\begin{align*}
 \left ( \coprod_{n\ge -1}  ((EG\times_GJ^{\varphi}(X))^{\times n+1})_{h\Sigma_{n+1}} \right ) /\sim  & \simeq  \left ( \coprod_{n\ge -1} ((J(X_{hG}))^{\times n+1} )_{h\Sigma_{n+1}}\right )  /\sim \\ 
& \simeq  \Sym(X_{hG}) \\
& \simeq  \Sym(X)_{hG}
\end{align*}
proving the claim. Here the last equivalence follows because $\Sym$ is a left adjoint and therefore it commutes with homotopy orbits. 
\end{proof}

\begin{prop}\label{prop: first step in Fiedorowicz thm}
Let $(G,\varphi)$ be a group with parity and let $\Delta \bfGS$ be the associated crossed simplicial group. Let $M$ be an $\mathbb{E}_1$-space with twisted $G$-action. There is an equivalence
\[ 
\mathrm{T}\varphi(\Sigma_{+}^{\infty}M)\simeq \Sigma_+^{\infty}B(\Sym,J^{\varphi},M)_{hG}
\] 
where the right hand side is the monadic bar construction. 
\end{prop} 
\begin{proof}
Since $M$ is an $\mathbb{E}_1$-space with twisted $G$-action there is an equivalence in $\Alg^\varphi(\cS)$ 
\[ 
B(\Sym, J^{\varphi}, M )\simeq M
\]
and consequently an equivalence of spectra
\[ 
\mathrm{T}\varphi 
(\Sigma_+^{\infty}B(J^{\varphi}, J^{\varphi}, M ))\simeq \mathrm{T}\varphi(\Sigma^\infty_+M) . 
\]
By commuting colimits with colimits, there is an equivalence of spaces
\[ 
\hocolim_{\Delta \varphi \wr \Sigma_+ } B_{\varphi\wr \Sigma,+}^{\bullet}(B(J^{\varphi}, J^{\varphi}, M )) \simeq  B(\hocolim_{\Delta \varphi \wr \Sigma_+ }B_{\varphi\wr \Sigma,+}^{\bullet}(J^{\varphi},J^{\varphi}, M ))
\]
and by \autoref{prop: formula}, there is an equivalence of spaces
\[
B(\hocolim_{\Delta \varphi \wr \Sigma _+}B_{\varphi \wr \Sigma}^{\bullet}J^{\varphi},J^{\varphi},M ) \simeq B(\Sym_{hG} ,J^{\varphi},M)\,.
\]
By \autoref{cor: group ring computation for csg with duality}, the result follows from the equivalence 
\[
B(\mathrm{Sym}_{hG} ,J^{\varphi},M)\simeq B(\mathrm{Sym} ,J^{\varphi},M)_{hG} \,.\qedhere
\]
\end{proof}

\begin{notation}
Given a group with parity $\varphi \colon G\to C_2$, we write $S^{\varphi}$ for the one point compactification of the representation $\varphi^*\sigma$ where $\sigma$ is the sign representation. Note that when $\varphi$ is the identity on $C_2$, $S^\sigma$ is up to isomorphism the circle with flip across the $x$-axis as introduced in  \autoref{def: the circles}.
We also write $\Sigma^{\varphi}Y=S^{\varphi}\wedge Y$ for a space with $G$-action $Y$. 
\end{notation}

We now illustrate how the proof strategy from \cite[\S~5.3]{Aul10} and \cite[\S~7]{Gra22} applies in our more general context. 
\begin{proof}[Proof of \autoref{thm: topological positive hyperoctahedral homology}]
By \autoref{prop: first step in Fiedorowicz thm}, there is an equivalence of spectra
\[
 \hocolim_{\Delta \varphi \wr \Sigma_+ } B_{\varphi\wr \Sigma,+}^{\bullet}\Sigma_+^{\infty}M\simeq \Sigma_+^{\infty}B(\Sym,J^{\varphi},M)_{hG} \,.
\]
Since $M$ is group-like, there is an equivalence
\[
B(\Sym,J^{\varphi},M)_{hG} 
\simeq  B(Q,J^{\varphi},M)_{hG}\,.
\]
The equivalence $\Omega Q \Sigma Y\simeq Q Y$ induces an equivalence 
\[
B(Q,J^{\varphi},M)_{hG}\simeq B\left (\Omega  Q \Sigma ,J^{\varphi},M \right )_{hG}
\]
and by \cite[Lemma~9.7]{May72} for example, there is an equivalence
\[
B\left (\Omega  Q \Sigma ,J^{\varphi},M \right )_{hG} \simeq  \left ( \Omega Q B(\Sigma ,J^{\varphi},M ) \right )_{hG}\,. 
\]
Since there is an equivalence $EG_+\wedge S^{\varphi}\simeq EG_+\wedge S^{1}$, we can further identify 
\[
\left ( \Omega Q B(\Sigma  ,J^{\varphi},M  )\right)_{hG}\simeq \left ( \Omega Q B(\Sigma^{\varphi} ,J^{\varphi},M  )\right)_{hG} \,.
\]
Tracing through the zig-zags of weak equivalences appearing in \cite[Theorem~7.3, Theorem~7.8]{Fie84} which are used to prove~\cite[Corollary~7.9]{Fie84}, we observe that they are equivariant using a variation on equivariant Moore loops construction~\cite[pp.~83--84]{Ryb91}, which generalizes mutatis mutandis to our setting. This yields an equivalence 
\[ B(\Sigma^{\varphi},J^{\varphi},M)\simeq B^{\varphi} M 
\]
of spaces with $G$-action. 
This implies the desired equivalence
\[
\left ( \Omega Q B(\Sigma^{\varphi} ,\mathcal{P}_{\varphi},M  )\right )_{hG}  \simeq
\left ( \Omega Q B^{\varphi} M  \right )_{hG}. \qedhere
\]
\end{proof}
We end by providing a description of the canonical maps from \autoref{rem: canonical map to top twisted symmetric homology}.

\begin{prop}\label{topological positive symmetric homology and free loop spaces}
Let $X$ be connected space.
The canonical map
\[ 
\psi \colon \THH(\Sigma^\infty_+\Omega X)\longrightarrow \TS(\Sigma^\infty_+\Omega X)
\]
can be identified with the map 
\[ 
\Sigma_+^{\infty}\cL X \to \Sigma_+^{\infty}\cL QX \to \Sigma_+^{\infty}\Omega QX
\]
induced by the unit map $\eta\colon X\to QX$ and the counit $\epsilon \colon \THH(\Sigma^\infty_+\Omega QX)\to \Sigma_+^{\infty}\Omega QX$.
\end{prop}

\begin{proof}
By \autoref{cor: topological positive symmetric homology}, we can identify the canonical map 
\[ 
\THH(\SI\Omega X)\to \THH(\SI\Omega X)_{hS^{1}} \to \TS(\SI\Omega X) 
\]
with a map 
\[\psi \colon \Sigma_+^{\infty}\cL X \longrightarrow \Sigma_{+}^{\infty}\Omega QX \,. \]
Since it factors through the $S^1$-homotopy orbits, it corresponds to an $S^1$-equivariant map 
$\Sigma_+^{\infty}\cL X \to \Sigma_{+}^{\infty}\Omega QX$
where the target has trivial $S^1$-action. The map 
\[ 
\begin{tikzcd}
\Sigma_{+}^{\infty}\cL X \arrow{r}{\eta} &  \Sigma_{+}^{\infty}\cL QX\arrow{r}{\epsilon} &  \Sigma_{+}^{\infty}\Omega QX 
\end{tikzcd}
\]
is also a map of spectra with $S^1$-action where $\Sigma_{+}^{\infty}\Omega QX$ has trivial $S^1$-action. The result then follows from the commutative diagram of spectra with $S^1$-action
\[
\begin{tikzcd}
\Sigma_{+}^{\infty}\cL X  \arrow[swap]{d}{\eta} \arrow{r}{\phi} & \Sigma_{+}^{\infty}\Omega QX \arrow{d}{\eta} \arrow[bend left=60]{dd}{\id} \\ 
\Sigma_{+}^{\infty}\cL QX  \arrow[swap]{d}{\epsilon} \arrow{r}{\phi} & \Sigma_{+}^{\infty}\Omega  QQX \arrow{d}{\epsilon}  \\
\Sigma_{+}^{\infty}\Omega QX   \arrow{r}{\id} & \Sigma_{+}^{\infty}\Omega QX   
\end{tikzcd}
\]
where the top square commutes by applying the natural transformation 
\[\THH(\SI\Omega (-))\to \TS(\SI\Omega (-))\] 
to the map $\eta \colon X\to QX$ and the bottom square commutes by the universal properties of $\THH$ and $Q$ applied to $QX$. The right vertical composite map is the identity since it is exactly the map exhibiting $QX$ as a retract of $QQX$. 
\end{proof}

\begin{prop}\label{topological positive hyperoctahedral homology and free loop spaces}
Let $X$ be connected space with $C_2$-action.
The canonical map
\[ 
\psi \colon \THR(\SI\Omega^{\sigma} X)\longrightarrow \TO(\SI\Omega^{\sigma} X)
\]
can be identified with the map 
\[ 
\Sigma_+^{\infty}\cL X \to \Sigma_+^{\infty}\cL QX \to (\Sigma_+^{\infty}\cL QX)_{hC_2}\to (\Sigma_+^{\infty}\Omega QX)_{hC_2}
\]
induced by the unit map $\eta\colon X\to Q^{C_2}X$, the canonical map to the $C_2$-homotopy orbits, and the counit $\epsilon \colon \THR(\SI\Omega QX)_{hC_2}\to (\Sigma_+^{\infty}\Omega QX)_{hC_2}$.
\end{prop}
 
\begin{proof}
By \autoref{cor: topological positive symmetric homology}, we can identify the canonical map 
\[ 
\THR(\SI\Omega^{\sigma} X)\to \THR(\SI\Omega^{\sigma} X)_{hO(2)} \to \TO(\SI\Omega^{\sigma} X) 
\]
with a map 
\[\psi \colon \Sigma_+^{\infty}\cL X \longrightarrow (\Sigma_{+}^{\infty}\Omega QX)_{hC_2} \,. \]
Since it factors through the $O(2)$-homotopy orbits, it corresponds to an $O(2)$-equivariant map 
$\Sigma_+^{\infty}\cL X \to (\Sigma_{+}^{\infty}\Omega QX)_{hC_2}$
where the target has trivial $O(2)$-action. The map 
\[ 
\begin{tikzcd}
\Sigma_{+}^{\infty}\cL X \arrow{r}{\eta} &  \Sigma_{+}^{\infty}\cL QX\to (\Sigma_{+}^{\infty}\cL QX)_{hC_2}\arrow{r}{\epsilon} &  (\Sigma_{+}^{\infty}\Omega QX)_{hC_2} 
\end{tikzcd}
\]
is also a map of spectra with $O(2)$-action where $(\Sigma_{+}^{\infty}\Omega QX)_{hC_2}$ has trivial $O(2)$-action. The result then follows from the commutative diagram of spectra with $O(2)$-action
\[
\begin{tikzcd}
\Sigma_{+}^{\infty}\cL X  \arrow[swap]{d}{\eta} \arrow{r}{\phi} & (\Sigma_{+}^{\infty}\Omega QX)_{hC_2} \arrow{d}{\eta} \arrow[bend left=70]{dd}{\id} \\ 
\Sigma_{+}^{\infty}\cL QX  \arrow[swap]{d}{\epsilon} \arrow{r}{\phi} & (\Sigma_{+}^{\infty}\Omega  QQX)_{hC_2} \arrow{d}{\epsilon}  \\
\Sigma_{+}^{\infty}\Omega QX  \arrow{r}{} & (\Sigma_{+}^{\infty}\Omega QX)_{hC_2}  
\end{tikzcd}
\]
where the top square commutes by applying the natural transformation 
\[\THR(\SI\Omega (-))\to \TO(\SI\Omega(-))\] 
to the map $\eta \colon X\to QX$, the bottom square commutes by the universal properties of $\THR$ (cf.~\cite[Theorem~1.2]{AKGH21}) applied to the canonical map $\SI\Omega^{\sigma} QX\to \SI(\Omega^{\sigma} QX)_{hC_2}$ and the universal property of $Q$ applied to the infinite loop space $QX$. The right vertical composite map is the identity since it is exactly the map exhibiting $QX$ as a retract of $QQX$.  
\end{proof}

\begin{rem}
Let $q\colon C_4\to C_2$ be the quotient homomorphism.
Let $\mathrm{T}q$ be the topological twisted symmetric homology associated to $q$ as in \autoref{twisted symmetric example}.
In a similar way as the last results, we expect that given a connected space $X$ with a $C_4$-action, the induced map $\THQ(\Sigma^\infty_+ \Omega^{q} X)\to \mathrm{T}q(\Sigma^\infty_+ \Omega^{q}X)$ can be identified as a map
\[
\Sigma^\infty_+\mathscr{L}^\tau X \longrightarrow \Sigma^\infty_+ \Omega Q X_{hC_4}
\]
that factors through the $\Pin(2)$-homotopy orbits of $\Sigma^\infty_+\mathscr{L}^\tau X$.
\end{rem}

  \bibliographystyle{alpha}
  \bibliography{akmp}
\end{document}

%% file: preamble.tex
\usepackage{comment}
\usepackage{amsmath, amsfonts, amsthm, amstext, xspace, graphicx}
\usepackage{amssymb,latexsym}
\usepackage{sseq}
\usepackage{tikz-cd}
\usepackage[all]{xy}
\usepackage{todonotes}
\usepackage{microtype}
\usepackage{thmtools}
\usepackage{mathtools}
\usepackage[normalem]{ulem}
\usepackage{quiver}
\usepackage[mathscr]{euscript} 
\usepackage{environ}
\usepackage{setspace}

\theoremstyle{plain}
\newtheorem{thm}{Theorem}[section]
\newtheorem{lem}[thm]{Lemma}

\newtheorem{cor}[thm]{Corollary}

\newtheorem{prop}[thm]{Proposition}

\newtheorem{meta-conj}[thm]{Meta-Conjecture}

\theoremstyle{definition}
\newtheorem{defin}[thm]{Definition}
\newtheorem{defn}[thm]{Definition}
\newtheorem{notation}[thm]{Notation}

\newtheorem{rem}[thm]{Remark}
\newtheorem{rem2}[thm]{Remark}
\newtheorem{rem3}[thm]{Remark}
\newtheorem{exm}[thm]{Example}
\newtheorem{const}[thm]{Construction}

\newtheorem{notn}[thm]{Notation}
\newtheorem{thmx}{Theorem}

\usepackage{xcolor}

\definecolor{darkgreen}{rgb}{0,0.30,0} 
\definecolor{darkred}{rgb}{0.75,0,0}
\definecolor{darkblue}{rgb}{0,0,0.6} 
\definecolor{lightblue}{RGB}{179, 230, 255}
\usepackage{amsthm,thmtools}
\usepackage[colorlinks,citecolor=darkgreen,linkcolor=darkgreen,urlcolor=darkblue]{hyperref}
\usepackage[english]{babel}

\addto\extrasenglish{
    
}
\usepackage{cleveref}

\def\makeautorefname#1#2{\expandafter\def\csname#1autorefname\endcsname{#2}}
\makeautorefname{equation}{Equation}%
\makeautorefname{footnote}{footnote}%
\makeautorefname{item}{item}%
\makeautorefname{figure}{Figure}%
\makeautorefname{table}{Table}%
\makeautorefname{part}{Part}%
\makeautorefname{appendix}{Appendix}%
\makeautorefname{chapter}{Chapter}%
\makeautorefname{section}{Section}%
\makeautorefname{subsection}{Section}%
\makeautorefname{subsubsection}{Section}%
\makeautorefname{paragraph}{Paragraph}%
\makeautorefname{subparagraph}{Paragraph}%
\makeautorefname{theorem}{Theorem}%
\makeautorefname{thm}{Theorem}%
\makeautorefname{addm}{Addendum}%
\makeautorefname{mainthm}{Main theorem}%
\makeautorefname{corollary}{Corollary}%
\makeautorefname{cor}{Corollary}%
\makeautorefname{lemma}{Lemma}%
\makeautorefname{lem}{Lemma}%
\makeautorefname{sublemma}{Sublemma}%
\makeautorefname{sublem}{Sublemma}%
\makeautorefname{subl}{Sublemma}%
\makeautorefname{prop}{Proposition}%
\makeautorefname{property}{Property}
\makeautorefname{pro}{Property}
\makeautorefname{sch}{Scholium}%
\makeautorefname{step}{Step}%
\makeautorefname{conject}{Conjecture}%
\makeautorefname{conjecture}{Conjecture}%
\makeautorefname{questn}{Question}
\makeautorefname{quest}{Question}
\makeautorefname{qn}{Question}
\makeautorefname{definition}{Definition}%
\makeautorefname{defn}{Definition}%
\makeautorefname{defi}{Definition}%
\makeautorefname{def}{Definition}%
\makeautorefname{dfn}{Definition}%
\makeautorefname{defin}{Definition}%
\makeautorefname{df}{Definition}%
\makeautorefname{notation}{Notation}
\makeautorefname{notn}{Notation}
\makeautorefname{rem}{Remark}%
\makeautorefname{rem2}{Remark}%
\makeautorefname{remark}{Remark}%
\makeautorefname{rems}{Remarks}%
\makeautorefname{rmk}{Remark}%
\makeautorefname{rk}{Remark}%
\makeautorefname{remarks}{Remarks}%
\makeautorefname{rems}{Remarks}%
\makeautorefname{rmks}{Remarks}%
\makeautorefname{rks}{Remarks}%
\makeautorefname{example}{Example}%
\makeautorefname{examp}{Example}%
\makeautorefname{exmp}{Example}%
\makeautorefname{exam}{Example}%
\makeautorefname{exm}{Example}%
\makeautorefname{exa}{Example}%
\makeautorefname{axiom}{Axiom}%
\makeautorefname{axi}{Axiom}%
\makeautorefname{ax}{Axiom}%
\makeautorefname{case}{Case}%
\makeautorefname{claim}{Claim}%
\makeautorefname{clm}{Claim}%
\makeautorefname{assumpt}{Assumption}%
\makeautorefname{asses}{Assumptions}%
\makeautorefname{conclusion}{Conclusion}%
\makeautorefname{concl}{Conclusion}%
\makeautorefname{conc}{Conclusion}%
\makeautorefname{cond}{Condition}%
\makeautorefname{const}{Construction}%
\makeautorefname{con}{Construction}%
\makeautorefname{criterion}{Criterion}%
\makeautorefname{criter}{Criterion}%
\makeautorefname{crit}{Criterion}%
\makeautorefname{exercise}{Exercise}%
\makeautorefname{exer}{Exercise}%
\makeautorefname{exe}{Exercise}%
\makeautorefname{problem}{Problem}%
\makeautorefname{problm}{Problem}%
\makeautorefname{prob}{Problem}%
\makeautorefname{prob}{Problem}%
\makeautorefname{soln}{Solution}%
\makeautorefname{sol}{Solution}%
\makeautorefname{sum}{Summary}%
\makeautorefname{oper}{Operation}%
\makeautorefname{obs}{Observation}%
\makeautorefname{ob}{Observation}%
\makeautorefname{conv}{Convention}%
\makeautorefname{cvn}{Convention}%
\makeautorefname{warn}{Warning}%
\makeautorefname{note}{Note}%
\makeautorefname{fact}{Fact}%
\makeautorefname{ouch0}{Counterexample}%
\makeautorefname{var}{Variant}%

\newcommand{\C}{\mathrm{C}}
\DeclareMathOperator{\Aut}{Aut}

\DeclareMathOperator{\id}{\textup{id}}

\DeclareMathOperator{\Sp}{Sp}

\DeclareMathOperator{\Cat}{\mathrm{Cat}}
\DeclareMathOperator{\Assoc}{Assoc} 
\DeclareMathOperator{\Fin}{\mathrm{Fin}}
\DeclareMathOperator{\Fun}{\mathrm{Fun}}
\DeclareMathOperator{\CAT}{\mathrm{CAT}}

\DeclareMathOperator{\Sym}{\mathrm{Sym}}
\DeclareMathOperator{\Set}{Set}

\newcommand{\Mon}{\mathrm{Mon}}
\DeclareMathOperator{\Env}{Env} 

\DeclareMathOperator{\rmO}{O}

\DeclareMathOperator{\THH}{THH}
\newcommand{\THHG}{\mathrm{THH}_{G}}
\newcommand{\TCG}{\mathrm{TC}_{G}}
\DeclareMathOperator{\THQ}{THQ}
\renewcommand{\TH}{\mathrm{TH}}
\DeclareMathOperator{\HH}{HH}
\DeclareMathOperator{\HC}{HC}
\DeclareMathOperator{\HP}{HP}
\DeclareMathOperator{\TP}{TP}
\DeclareMathOperator{\TC}{TC}
\DeclareMathOperator{\TD}{TD}
\DeclareMathOperator{\TQ}{TQ}
\newcommand{\TS}{\mathrm{T}\Sigma}
\DeclareMathOperator{\TO}{TO}
\DeclareMathOperator{\TB}{TB}
\newcommand{\TR}{\mathrm{T}\mathcal{R}}
\DeclareMathOperator{\THR}{THR}
\DeclareMathOperator{\rmK}{K}
\DeclareMathOperator{\TCR}{TCR}
\DeclareMathOperator{\THG}{\mathrm{TH}\mathbf{G}}
\DeclareMathOperator{\HG}{\mathrm{H}\mathbf{G}}
\DeclareMathOperator{\uTHG}{\mathrm{uTH}\mathbf{G}}

\DeclareMathOperator{\Pin}{Pin}
\DeclareMathOperator{\uTHH}{uTHH}
\DeclareMathOperator{\uTHR}{uTHR}
\DeclareMathOperator{\uTHQ}{uTHQ}

\newcommand{\CSG}{\mathrm{csGr}}
\newcommand{\Grparity}{\mathrm{Gr}_{/{C_2}}}

\DeclareMathOperator{\sd}{sd}
\DeclareMathOperator{\Hom}{Hom}

\DeclareMathOperator{\colim}{\operatorname{colim}}
\DeclareMathOperator{\hocolim}{\operatorname{hocolim}}
\newcommand{\sdual}{\cI}

\newcommand{\sbt}{\,\begin{picture}(-1,1)(0.5,-1)\circle*{1.8}\end{picture}\hspace{.05cm}}

\usepackage{bbm}

\DeclareMathOperator{\per}{\textup{per}}

\newcommand{\op}{{\textup{op}}}

\newcommand{\bfC}{\mathbf{C}}

\newcommand{\bfB}{\mathbf{B}}
\newcommand{\bfG}{\mathbf{G}}
\newcommand{\bfGH}{\mathbf{G}\times \mathrm{H}}
\newcommand{\bfH}{\mathbf{H}}
\newcommand{\bfD}{\mathbf{D}}

\newcommand{\bfR}{\mathbf{R}}
\newcommand{\bfS}
{\mathbf{\Sigma}}
\newcommand{\bfGS}{\varphi \wr \mathbf{\Sigma}}
\newcommand{\bfCG}{\mathbf{C}\times G}

\newcommand{\bB}{\mathbb{B}}
\DeclareMathOperator{\bE}{\mathbb{E}}

\DeclareMathOperator{\bT}{\mathbb{T}}

\newcommand{\TT}{\mathbb{T}}
\newcommand{\OO}{\mathbb{O}}
\newcommand{\PP}{\mathbb{P}}

\newcommand{\cC}{{\mathscr{C}}}
\DeclareMathOperator{\cD}{\mathscr{D}}

\newcommand{\cL}{\mathscr{L}}
\DeclareMathOperator{\cM}{\mathscr{M}}

\DeclareMathOperator{\cO}{\mathscr{O}}

\newcommand{\cS}{\mathscr{S}}

\DeclareMathOperator{\cW}{\mathscr{W}}

 \DeclareMathOperator{\cI}{\mathcal{I}}

\newcommand{\rmT}{\mathrm{T}}

\newcommand{\gp}{\mathrm{gp}}

\newcommand{\Ind}{\mathrm{Ind}}
\newcommand{\Res}{\mathrm{Res}}

\newcommand{\SI}{\Sigma^\infty_+}

\DeclareMathOperator{\kF}{\mathbf{k}}
\DeclareMathOperator{\nF}{\mathbf{n}}
\newcommand{\mF}{\mathbf{m}}
\DeclareMathOperator{\pF}{\mathbf{p}}

\newcommand{\bfQ}{\mathbf{Q}}

\DeclareMathOperator{\Alg}{Alg}

\DeclareMathOperator{\Map}{Map}
\DeclareMathOperator{\Tor}{Tor}

\DeclareMathOperator{\inc}{inc}

\newcommand{\mD}{\langle m \rangle}
\newcommand{\nplusmD}{\langle n+m \rangle}
\newcommand{\nD}{\langle n \rangle}

\definecolor{seagreen}{RGB}{46,139,87}
\definecolor{newteal}{RGB}{38, 217, 207}
\definecolor{darkviolet}{RGB}{148,0,211}

\newcommand{\vperp}{\rotatebox[origin=c]{90}{$\perp$}}

\NewEnviron{LARGER}{%
    \scalebox{2}{$\BODY$} 
} 